\documentclass[11pt,a4paper]{article}
\usepackage{amssymb}
\usepackage{amsmath}
\usepackage{latexsym}
\usepackage{amsthm}
\usepackage{dsfont}
\usepackage{enumerate}
\usepackage{hyperref}

\newtheorem{thm}{Theorem}
\newtheorem{lem}[thm]{Lemma}
\newtheorem{cor}[thm]{Corollary}
\theoremstyle{definition}
\newtheorem{remark}{Remark}

\newcommand\extrafootertext[1]{
    \bgroup
    \renewcommand\thefootnote{\fnsymbol{footnote}}
    \renewcommand\thempfootnote{\fnsymbol{mpfootnote}}
    \footnotetext[0]{#1}
    \egroup
}

\usepackage{xpatch}
\xpatchcmd{\proof}{\itshape}{\normalfont\proofnameformat}{}{}
\newcommand{\proofnameformat}{}

\begin{document}

\renewcommand{\proofnameformat}{\bfseries}

\begin{center}
{\Large\textbf{Smoothing inequalities for transport metrics in compact spaces}}

\vspace{10mm}

\textbf{Bence Borda and Jean-Claude Cuenin}

\let\thefootnote\relax

\let\thefootnote\relax
\end{center}

\vspace{10mm}

\begin{abstract}
We prove general upper estimates for the distance between two Borel probability measures in Wasserstein metric in terms of the Fourier transforms of the measures. We work in compact manifolds including the torus, the Euclidean unit sphere, compact Lie groups and compact homogeneous spaces, and treat the Wasserstein metric $W_p$ in the full range $1 \le p \le \infty$ for the first time. The proofs are based on a comparison between the Wasserstein metric and a dual Sobolev norm, Riesz transform estimates and Hausdorff--Young inequalities on compact manifolds. As an application, we show that spherical designs are optimally close to the uniform measure on the sphere in Wasserstein metric.
\end{abstract}

\tableofcontents

\section{Introduction}

The idea of estimating the distance between two probability measures in terms of their Fourier transforms goes back to Berry \cite{BE} and Esseen \cite{ES}. The celebrated Berry--Esseen inequality gives an upper bound for the distance between two probability measures in the Kolmogorov metric in terms of the characteristic functions, that is, the Fourier transforms of the measures. A large number of similar results have been established for various probability metrics both on the real line and in higher dimensions, and have become indispensible tools in probability theory. As their proofs typically involve a smoothing procedure with a suitable kernel, they are known as smoothing inequalities. We refer to Bobkov \cite{BOB} for a survey.

In the present paper, we work in compact spaces including the torus, the sphere, compact Lie groups, compact homogeneous spaces and compact Riemannian manifolds, and use the Wasserstein metric $W_p$ as probability metric. Our goal is to prove upper bounds for $W_p (\mu, \nu)$ in terms of the Fourier transforms of the probability measures $\mu$ and $\nu$.

The Wasserstein metric was introduced in the framework of optimal transport, and has found numerous applications in analysis, geometry and probability theory. The quadratic metric $W_2$ is closely related to logarithmic Sobolev inequalities, Poincar\'e inequalities and relative entropy \cite{OV}, whereas the linear metric $W_1$ has recent applications in computer vision and machine learning \cite{ACB}. We refer to Villani \cite{VI} for a comprehensive account of optimal transport and the Wasserstein metric.

Let us give a short survey of previous results in the area. Graham \cite{GR} proved smoothing inequalities for $W_p$, $2 \le p \le \infty$ on the unit interval $[0,1]$ and the 1-dimensional torus $\mathbb{R} / \mathbb{Z}$ by relying on the explicit solution of the optimal transport problem in 1-dimensional spaces. Bobkov and Ledoux \cite{BL1,BL2} worked with $W_p$, $0 < p \le 1$ on the unit cube $[0,1]^d$ and the torus $\mathbb{R}^d / \mathbb{Z}^d$. Their results were later generalized by the first author \cite{BOR1} to compact Lie groups. Grabner and Tichy \cite{GT} proved a smoothing inequality for functions with a prescribed modulus of continuity on the unit sphere $\mathbb{S}^d$, which could equivalently be formulated in terms of $W_p$, $0<p\le 1$. Brown and Steinerberger \cite{BS} and the first author \cite{BOR2} established a smoothing inequality for $W_2$ on compact Riemannian manifolds. In particular, no smoothing inqualities have been known for $W_p$, $p>2$ in compact spaces of dimension $d \ge 2$.

In this paper, we prove smoothing inequalities for $W_p$ in the full range $1 \le p \le \infty$ in compact spaces of all dimensions $d \ge 1$ for the first time. We treat $W_p$, $1 \le p < \infty$ on the torus, the sphere, compact Lie groups, compact homogeneous spaces and general compact Riemannian manifolds. We also treat $W_{\infty}$ on the torus and the sphere $\mathbb{S}^d$, $d \ge 3$. All our results are fully explicit, although we did not make a serious effort to optimize the value of the constants.

The aforementioned papers explored several applications of smoothing inequalities to optimal matchings, random walks on compact groups, quantization of measure, irregularities of distribution and numerical integration. The new results in the present paper have immediate applications to all these problems, and potentially many more. However, we restrict our attention to the inequalities themselves for now, and present only one application as an illustration: in Section \ref{sphericaldesignsection} we show that spherical designs are optimal in Wasserstein metric.

The general framework of estimating the Wasserstein metric is presented in Section \ref{approachsection}. The framework follows the approach of Ledoux \cite{LE1} based on Kantorovich duality, the Hopf--Lax semigroup and Hamilton--Jacobi equations. This approach has its roots in the PDE techniques developed for the optimal matching problem by Ambrosio, Stra and Trevisan \cite{AST}. An important technical ingredient in the novel case $W_p$, $p>2$ is a Hausdorff--Young inequality for the gradient $\nabla f$ of a smooth function $f$, stated in various spaces in Lemmas \ref{torushausdorffyounglemma}, \ref{liegrouphausdorffyounglemma}, \ref{homogeneoushausdorffyounglemma} and \ref{manifoldhausdorffyounglemma}.

The general framework reduces the problem of estimating the Wasserstein metric to constructing suitable smoothing kernels: each choice of a kernel leads to a different smoothing inequality. In Section \ref{proofsection}, we work out the details for several kernels including powers of the Fej\'er kernel on the torus, constructions based on Poisson summation on the torus and compact Lie groups, constructions based on spectral theory on the sphere, and the heat kernel on general compact manifolds.

Before we present the main results in Section \ref{resultssection}, let us fix some notation. Throughout the paper, $M$ denotes a compact, connected, smooth Riemannian manifold of dimension $1 \le d < \infty$ without boundary. Let $\rho$ denote the geodesic metric, $\mathrm{diam}(M)=\sup \{ \rho (x,y) : x, y \in M \}$ the diameter and $\mathrm{Volume}$ the Riemannian volume measure. Let $\mathcal{P}(M)$ denote the set of Borel probability measures on $M$, and let $\mathrm{Vol}=\mathrm{Volume}/\mathrm{Volume}(M)\in \mathcal{P}(M)$ be the normalized Riemannian volume measure. Let $\mathrm{C}(M)$, $\mathrm{C}^{\infty}(M)$ and $\mathrm{Lip}(M)$ denote the set of continuous, smooth and Lipschitz functions $f: M \to \mathbb{R}$, respectively, and let
\[ \| f \|_{\mathrm{Lip}} = \sup_{\substack{x,y \in M \\ x \neq y}} \frac{|f(x)-f(y)|}{\rho (x,y)} . \]
The $L^p (M)$ norm of $f:M \to \mathbb{C}$ is denoted by $\| f \|_p=(\int_M |f|^p \, \mathrm{d}\mathrm{Vol})^{1/p}$, $1 \le p < \infty$. We use the same notation for vector-valued functions, e.g.\ $\| \nabla f \|_p = (\int_M |\nabla f|^p \, \mathrm{d}\mathrm{Vol})^{1/p}$.

The Laplace--Beltrami operator on $M$ is denoted by $\Delta$. As detailed in Sections \ref{torussection}--\ref{manifoldsection}, we use different notations for the spectrum and the Fourier transform depending on the space $M$.

The Wasserstein metric of order $0<p<\infty$ is defined as
\[ W_p (\mu, \nu) = \inf_{\vartheta \in \textrm{Coupling}(\mu, \nu)} \left( \int_{M \times M} \rho (x,y)^p \, \mathrm{d} \vartheta (x,y) \right)^{\min \{ 1/p, 1 \}}, \quad \mu, \nu \in \mathcal{P}(M), \]
where $\textrm{Coupling}(\mu, \nu)$ is the set of all Borel probability measures on $M \times M$ with marginals $\mu$ and $\nu$. The Wasserstein metric of order $p=\infty$ is defined similarly, with the $L^p(\vartheta)$ norm replaced by the $L^{\infty}(\vartheta)$ norm. Recall that $W_p$ is a metric on $\mathcal{P}(M)$ for any $0<p\le \infty$, and it metrizes weak convergence for $p \neq \infty$. An application of H\"older's inequality immediately shows that $W_p \le W_q$ for all $1 \le p \le q \le \infty$. We also have $\lim_{p \to \infty} W_p (\mu, \nu) = W_{\infty} (\mu, \nu)$ for any $\mu, \nu \in \mathcal{P}(M)$, see \cite{GS}. In the case $p=1$, the Kantorovich duality formula \cite[p.\ 107]{VI} states the useful dual formulation
\begin{equation}\label{kantorovichduality}
W_1 (\mu, \nu) = \sup_{\substack{f \in \mathrm{Lip}(M) \\ \| f \|_{\mathrm{Lip}} \le 1}} \left| \int_M f \, \mathrm{d} \mu - \int_M f \, \mathrm{d}\nu \right| .
\end{equation}

Let $|x|=(\sum_{j=1}^d x_j^2)^{1/2}$ and $\langle x,y \rangle=\sum_{j=1}^d x_j y_j$ be the standard Euclidean norm and Euclidean scalar product on $\mathbb{R}^d$. The indicator function of a set or relation $S$ is denoted by $\mathds{1}_S$. The gamma function is denoted by $\Gamma(x)$.

\section{Main results}\label{resultssection}

In Section \ref{torussection}, we present our results in the simplest setting of the flat torus. In Section \ref{liegroupsection}, these are generalized to compact, connected Lie groups in terms of representation theory. Section \ref{homogeneoussection} deals with compact homogeneous spaces. A compact Riemannian manifold $M$ without boundary is called homogeneous if the Riemannian isometry group (the group of diffeomorphisms from $M$ to itself that preserve the Riemannian metric) acts transitively on $M$. Equivalently, a compact homogeneous space can be defined as a quotient $M=G/K$ with a compact Lie group $G$ and a closed subgroup $K$. All compact homogeneous spaces have positive semidefinite Ricci curvature \cite{SA}. Classical examples include Euclidean unit spheres, projective spaces and Grassmannians. An application to spherical designs is presented in Section \ref{sphericaldesignsection}. Finally, Section \ref{manifoldsection} discusses the most general setting of compact Riemannian manifolds.

\subsection{Smoothing inequalities on the torus}\label{torussection}

Consider the $d$-dimensional torus $M=\mathbb{R}^d / \mathbb{Z}^d$ with the metric $\rho (x,y) = \| x-y \|_{\mathbb{R}^d / \mathbb{Z}^d}$, where $\| x \|_{\mathbb{R}^d / \mathbb{Z}^d} = \min_{n \in \mathbb{Z}^d} |x-n|$ is the Euclidean distance to the nearest lattice point. Let $\mathrm{Vol}$ denote the Lebesgue measure on $\mathbb{R}^d / \mathbb{Z}^d$ after its standard identification with the unit cube $[0,1]^d$. Note that $\rho$ is the geodesic metric and $\mathrm{Vol}$ is the normalized Riemannian volume measure on the torus equipped with the flat Riemannian metric. In addition, $\mathrm{Vol}$ is also the normalized Haar measure on the torus as a compact group. The Fourier transform of $f \in L^1 (\mathbb{R}^d / \mathbb{Z}^d)$ resp.\ $\mu \in \mathcal{P}(\mathbb{R}^d / \mathbb{Z}^d)$ is defined as $\widehat{f}(k)=\int_{\mathbb{R}^d / \mathbb{Z}^d} f(x) e^{-2 \pi i \langle k,x \rangle} \, \mathrm{d}\mathrm{Vol}(x)$ resp.\ $\widehat{\mu}(k)=\int_{\mathbb{R}^d / \mathbb{Z}^d} e^{- 2 \pi i \langle k,x \rangle} \, \mathrm{d} \mu(x)$, $k \in \mathbb{Z}^d$.

Bobkov and Ledoux \cite{BL1,BL2} proved sharp smoothing inequalities on the torus for $W_p$, $0<p \le 1$. We now present a smoothing inequality for $W_p$, $1 \le p < \infty$ on the torus in two slightly different formulations. Theorems \ref{torusjacksonkerneltheorem} and \ref{torusheatkerneltheorem} with $p=1$ reduce to the results of Bobkov and Ledoux, whereas Theorem \ref{torusheatkerneltheorem} with $p=2$ is implicit in \cite{BOR2} and \cite{BS}. All other cases, including $p>2$ are new.
\begin{thm}\label{torusjacksonkerneltheorem} Let $\mu, \nu \in \mathcal{P}(\mathbb{R}^d / \mathbb{Z}^d)$, and assume there exists a constant $c \ge 0$ such that $\mu \ge c \mathrm{Vol}$. Let $1 \le p < \infty$ and $1<q \le \infty$ be such that $1/p+1/q=1$. Assume either $p=1$ or $c>0$. For any integer $H>H_0$, we have
\[ W_p(\mu, \nu) \le \frac{C_1}{H-H_0} + C_2 \Bigg( \sum_{\substack{k \in [-H,H]^d \\ k \neq 0}} \frac{|\widehat{\mu}(k) - \widehat{\nu}(k)|^{q_0}}{(2 \pi |k|)^{q_0}} \Bigg)^{1/q_0} , \]
where $q_0=\min \{ q, 2 \}$,
\[ \begin{split} H_0 &= \left\{ \begin{array}{ll} 1 & \textrm{if } 1 \le p \le 2, \\ (p+2)/2 & \textrm{if } p>2, \end{array} \right. \\ C_1 &= \left\{ \begin{array}{ll} 2^{1/2}(1+(1-c)^{1/p})d^{1/2} & \textrm{if } 1 \le p \le 2, \\ 6^{-1/2} (1+(1-c)^{1/p}) (p+4) d^{1/2} & \textrm{if } p>2, \end{array} \right. \\ C_2 &=\left\{ \begin{array}{ll} 1 & \textrm{if } p=1, \\ p c^{-1/q} & \textrm{if } p>1. \end{array} \right. \end{split} \]
\end{thm}

\begin{thm}\label{torusheatkerneltheorem} Let $\mu, \nu \in \mathcal{P}(\mathbb{R}^d / \mathbb{Z}^d)$, and assume there exist constants $c \ge 0$, $b \ge 0$ and $r>0$ such that $\mu \ge c \mathrm{Vol}$ and $\nu (B) \ge b$ for any closed ball $B \subseteq \mathbb{R}^d / \mathbb{Z}^d$ of radius $r$. Let $1 \le p < \infty$ and $1<q \le \infty$ be such that $1/p+1/q=1$. Assume either $p=1$ or $c>0$. For any real $t>0$, we have
\[ W_p (\mu, \nu) \le C_1 t^{1/2} + C_2(t) \Bigg( \sum_{\substack{k \in \mathbb{Z}^d \\ k \neq 0}} e^{-4 \pi^2 |k|^2  q_0 t} \frac{|\widehat{\mu}(k) - \widehat{\nu}(k)|^{q_0}}{(2 \pi |k|)^{q_0}} \Bigg)^{1/q_0} , \]
where $q_0=\min \{ q, 2 \}$,
\[ \begin{split} C_1 &= 2 (1+(1-c)^{1/p}) \Bigg( \frac{\Gamma \left( \frac{d+p}{2} \right)}{\Gamma \left( \frac{d}{2} \right)} \Bigg)^{1/p} \le (1+(1-c)^{1/p}) (2d+p)^{1/2}, \\ C_2(t) &=\left\{ \begin{array}{ll} 1 & \textrm{if } p=1, \\ \frac{p(c^{1/p}-\delta^{1/p})}{c-\delta} & \textrm{if } p>1 \end{array} \right. \qquad \textrm{with} \qquad \delta=\min \left\{ b \frac{e^{-r^2/(4t)}}{(4 \pi t)^{d/2}}, c \right\} . \end{split} \]
\end{thm}
\noindent In the case $p>1$ and $\delta=c$, the constant $C_2(t)$ is to be interpreted as
\[ C_2(t)=\lim_{\delta \to c} \frac{p(c^{1/p}-\delta^{1/p})}{c-\delta} = \frac{1}{c^{1/q}} . \]

Theorems \ref{torusjacksonkerneltheorem} and \ref{torusheatkerneltheorem} with $p=1$ and $c=0$ apply without any assumption on $\mu$. In contrast, the assumption $c>0$ is crucial in the case $p>1$. The most interesting special case $\mu = \mathrm{Vol}$ applies with $c=1$.

Theorem \ref{torusheatkerneltheorem} with $b=0$ applies without any assumption on $\nu$, in which case $C_2(t)=p c^{-1/q}$. The reason why we included the assumption $\nu (B) \ge b$ is to sharpen the dependence of $C_2(t)$ on $p$. For instance, we might be interested in estimating the distance from the empirical measure $\nu_N=N^{-1} \sum_{n=1}^N \delta_{a_n}$ of a finite point set $\{ a_1, a_2, \ldots, a_N \}$ to the Lebesgue measue $\mathrm{Vol}$. If we know that every ball of radius $r \approx N^{-1/d}$ contains at least one of the points $a_1, a_2, \ldots, a_N$, then we can choose $b=N^{-1}$. The choice $t \approx N^{-2/d}$ then leads to $\delta \approx 1$, and in particular, $C_2(t) \approx 1$, thus saving a factor of $p$.

Theorems \ref{torusjacksonkerneltheorem} and \ref{torusheatkerneltheorem} yield comparable estimates in the range $t \approx H^{-2}$. The main difference is that Theorem \ref{torusjacksonkerneltheorem} uses the $\lceil (p+2)/2 \rceil$th power of the Fej\'er kernel, whereas Theorem \ref{torusheatkerneltheorem} uses the heat kernel in the smoothing procedure.

We also treat $W_{\infty}$ on the torus, using the periodization of a compactly supported smooth function in the smoothing procedure. In this case, the assumption $b>0$ is crucial.
\begin{thm}\label{torusinfinitytheorem} Let $\mu, \nu \in \mathcal{P}(\mathbb{R}^d / \mathbb{Z}^d)$, and assume there exist constants $c>0$, $b>0$ and $r>0$ such that $\mu \ge c \mathrm{Vol}$ and $\nu (B) \ge b$ for any closed ball $B \subseteq \mathbb{R}^d / \mathbb{Z}^d$ of radius $r$. For any real $T \ge 5r$, we have
\[ W_{\infty}(\mu, \nu) \le C_1 T + C_2(T) \sum_{\substack{k \in \mathbb{Z}^d \\ k \neq 0}} A_k \frac{|\widehat{\mu}(k) - \widehat{\nu}(k)|}{2 \pi |k|}, \]
where $C_1 = 1+\mathds{1}_{\{ \mu \neq \mathrm{Vol} \}}$,
\[ \begin{split} C_2(T) &= \frac{\log c - \log \delta}{c-\delta} \quad \textrm{with} \quad \delta= \min \left\{ b \frac{\Gamma \left(\frac{d+2}{2} \right) 2^{d+1}}{27 \pi^{d/2} T^d}, c\right\}, \\ A_k &= \left\{ \begin{array}{ll} 1 & \textrm{if } |k| < \frac{d+3}{\pi T}, \\ \displaystyle{\exp \left( - \frac{d+1}{4 \log 2} \left( \log \frac{\pi T |k|}{d+3} \right) \left( \log \frac{e^2 \pi T |k|}{d+2} \right) \right)} & \textrm{if } |k| \ge \frac{d+3}{\pi T} . \end{array} \right. \end{split} \]
\end{thm}
\noindent In the case $\delta=c$, the constant $C_2(T)$ is to be interpreted as
\[ C_2(T)=\lim_{\delta \to c} \frac{\log c - \log \delta}{c-\delta} = \frac{1}{c} . \]

\subsection{Smoothing inequalities on compact Lie groups}\label{liegroupsection}

Let $\overline{z}$ be the complex conjugate of $z \in \mathbb{C}$, and let $\mathbb{C}^{n \times n}$ be the set of all $n \times n$ matrices with complex entries. For any matrix $A \in \mathbb{C}^{n \times n}$, let $A^*$ denote the conjugate transpose, let $\mathrm{Tr} (A)$ denote the trace, and let $\| A \|_{\mathrm{HS}} = \sqrt{\mathrm{Tr}(A A^*)}$ denote the Hilbert--Schmidt norm.

Let $G$ be a compact, connected Lie group of dimension $1 \le d < \infty$. Let $R^+$ denote the set of positive roots, $\rho^+=\sum_{w \in R^+} w/2$ the half-sum of positive roots, $v=\min_{w \in R^+} |w|$ and $r_G=d-2|R^+|$ the rank of $G$. Let $\widehat{G}$ be the unitary dual of $G$ with $\pi_0=1$ denoting the trivial representation. Let $d_{\pi}$ be the dimension, $w_{\pi}$ the highest weight and $\lambda_{\pi}$ the Laplace eigenvalue of $\pi \in \widehat{G}$. In particular, $\pi: G \to \mathbb{C}^{d_{\pi} \times d_{\pi}}$ and each entry $\pi_{ij}$ of $\pi$ satisfies $\Delta \pi_{ij} = - \lambda_{\pi} \pi_{ij}$. The spectrum of $G$ as a Riemannian manifold is thus $\lambda_{\pi}$ repeated with multiplicity $d_{\pi}^2$, $\pi \in \widehat{G}$, and $\{ d_{\pi}^{1/2} \pi_{ij} : \pi \in \widehat{G}, 1 \le i,j \le d_{\pi} \}$ is an orthonormal basis in $L^2 (G)$. The Fourier transform of $f \in L^1(G)$ resp.\ $\mu \in \mathcal{P}(G)$ is defined as $\widehat{f}(\pi) = \int_G f(x) \pi(x)^* \, \mathrm{d}\mathrm{Vol}(x)$ resp.\ $\widehat{\mu}(\pi) = \int_G \pi(x)^* \, \mathrm{d}\mu (x)$, $\pi \in \widehat{G}$. We refer to Bourbaki \cite{BOU} for the general theory of Lie groups, and to Section \ref{liegroupproofsection} for further notation and conventions.

The first author \cite{BOR1} proved a sharp smoothing inequality on compact, connected Lie groups for $W_p$, $0<p \le 1$. As a generalization of Theorems \ref{torusjacksonkerneltheorem} and \ref{torusheatkerneltheorem}, we now present the case $1 \le p < \infty$ in two slightly different formulations. Theorem \ref{liegroupjacksonkerneltheorem} with $p=1$ reduces to the result of the first author \cite{BOR1}, whereas Theorem \ref{liegroupheatkerneltheorem} with $p=2$ is implicit in \cite{BOR2} and \cite{BS}. All other cases, including $p>2$ are new.
\begin{thm}\label{liegroupjacksonkerneltheorem} Let $G$ be a compact, connected Lie group of dimension $1 \le d < \infty$. Let $\mu, \nu \in \mathcal{P}(G)$, and assume there exists a constant $c \ge 0$ such that $\mu \ge c \mathrm{Vol}$. Let $1 \le p < \infty$ and $1<q \le \infty$ be such that $1/p+1/q=1$. Assume either $p=1$ or $c>0$. For any real $H>H_0$, we have
\[ W_p (\mu, \nu) \le \frac{C_1}{H-H_0} + C_2 \Bigg( \sum_{\substack{\pi \in \widehat{G} \\ 0<|w_{\pi}| < H}} d_{\pi}^2 \frac{\| \widehat{\mu}(\pi) - \widehat{\nu}(\pi) \|_{\mathrm{HS}}^{q_0}}{(d_{\pi} \lambda_{\pi})^{q_0 /2}} \Bigg)^{1/q_0} , \]
where $q_0= \min \{ q, 2 \}$, $H_0=2 |\rho^+|+v$,
\[ \begin{split} C_1 &= (1+(1-c)^{1/p}) \frac{8 |\rho^+|+4v}{v} (r_G+3) 121^{1/p} 2^{p/(2r_G+2)+d/p} 15^{r_G/p}, \\ C_2 &= \left\{ \begin{array}{ll} 1 & \textrm{if } p=1, \\ pc^{-1/q} & \textrm{if } p>1. \end{array} \right. \end{split} \]
\end{thm}

\begin{thm}\label{liegroupheatkerneltheorem} Let $G$ be a compact, connected Lie group of dimension $1 \le d < \infty$. Let $\mu, \nu \in \mathcal{P}(G)$, and assume there exist constants $c \ge 0$, $b \ge 0$ and $r>0$ such that $\mu \ge c \mathrm{Vol}$ and $\nu (B) \ge b$ for any closed geodesic ball $B \subseteq G$ of radius $r$. Let $1 \le p < \infty$ and $1<q \le \infty$ be such that $1/p+1/q=1$. Assume either $p=1$ or $c>0$. For any real $t>0$, we have
\[ W_p (\mu, \nu) \le C_1 t^{1/2} + C_2(t) \Bigg( \sum_{\substack{\pi \in \widehat{G} \\ \pi \neq \pi_0}} d_{\pi}^2 e^{-\lambda_{\pi} q_0 t} \frac{\| \widehat{\mu}(\pi) - \widehat{\nu}(\pi)\|_{\mathrm{HS}}^{q_0}}{(d_{\pi} \lambda_{\pi})^{q_0/2}} \Bigg)^{1/q_0} , \]
where $q_0=\min \{ q, 2 \}$,
\[ \begin{split} C_1 &= \left\{ \begin{array}{ll} 2^{1/2} (1+(1-c)^{1/p}) d^{1/2} & \textrm{if } 1 \le p \le 2, \\ (1+(1-c)^{1/p}) (2d+p)^{1/2} & \textrm{if } p>2, \end{array} \right. \\ C_2(t) &=\left\{ \begin{array}{ll} 1 & \textrm{if } p=1, \\ \frac{p(c^{1/p}-\delta^{1/p})}{c-\delta} & \textrm{if } p>1 \end{array} \right. \qquad \textrm{with} \qquad \delta=\min \left\{ b \frac{e^{-r^2 /(4t)}}{(4 \pi t)^{d/2}} , c \right\} . \end{split} \]
\end{thm}
\noindent In the case $p>1$ and $\delta=c$, the constant $C_2(t)$ is to be interpreted as
\[ C_2(t)=\lim_{\delta \to c} \frac{p(c^{1/p}-\delta^{1/p})}{c-\delta} = \frac{1}{c^{1/q}} . \]

Theorems \ref{liegroupjacksonkerneltheorem} and \ref{liegroupheatkerneltheorem} with $p=1$ and $c=0$ apply without any assumption on $\mu$, but $c>0$ is crucial in the case $p>1$. Theorem \ref{liegroupheatkerneltheorem} with $b=0$ and $C_2(t)=pc^{-1/q}$ applies without any assumption on $\nu$, but the value of $C_2(t)$ can be sharpened if $b>0$. Theorems \ref{liegroupjacksonkerneltheorem} and \ref{liegroupheatkerneltheorem} yield comparable results in the range $t \approx H^{-2}$. Recall that $\lambda_{\pi}=|w_{\pi}|^2 + 2 \langle w_{\pi}, \rho^+ \rangle$ and in particular, $|w_{\pi}|^2 \le \lambda_{\pi} \le |w_{\pi}|^2 + O(|w_{\pi}|)$.

\subsection{Smoothing inequalities on compact homogeneous spaces}\label{homogeneoussection}

Let $M$ be a compact, connected homogeneous space of dimension $1 \le d < \infty$. Let $0=\lambda_0<\lambda_1<\lambda_2< \cdots$ denote the list of distinct eigenvalues of $-\Delta$, let $H_{\ell}=\{ \phi: M \to \mathbb{C} \textrm{ smooth} : \Delta \phi = -\lambda_{\ell} \phi \}$ be the Laplace eigenspaces, and let $d_{\ell}=\mathrm{dim}(H_{\ell})$. The spectrum of $M$ is thus $\lambda_{\ell}$ repeated $d_{\ell}$ times, $\ell \ge 0$, and $L^2(M)=\oplus_{\ell=0}^{\infty} H_{\ell}$ is an orthogonal direct sum. Let $\{ \phi_{\ell,m} : 1 \le m \le d_{\ell} \}$ be an arbitrary orthonormal basis in $H_{\ell}$. The Fourier transform of $f \in L^1(M)$ resp.\ $\mu \in \mathcal{P}(M)$ is defined as $\widehat{f}(\ell,m)=\int_M f \overline{\phi_{\ell,m}} \, \mathrm{d} \mathrm{Vol}$ resp.\ $\widehat{\mu}(\ell,m)=\int_M \overline{\phi_{\ell,m}} \, \mathrm{d}\mu$.

Theorems \ref{torusheatkerneltheorem} and \ref{liegroupheatkerneltheorem} can be further generalized to compact homogeneous spaces. The case $p=2$ is implicit in \cite{BOR2} and \cite{BS}, whereas $p \neq 2$ is new.
\begin{thm}\label{homogeneousheatkerneltheorem} Let $M$ be a compact, connected homogeneous space of dimension $1 \le d < \infty$. Let $\mu, \nu \in \mathcal{P}(M)$, and assume there exist constants $c \ge 0$, $b \ge 0$ and $r>0$ such that $\mu \ge c \mathrm{Vol}$ and $\nu(B) \ge b$ for any closed geodesic ball $B \subseteq M$ of radius $r$. Let $1 \le p < \infty$ and $1<q \le \infty$ be such that $1/p+1/q=1$. Assume either $p=1$ or $c>0$. For any real $t>0$, we have
\[ W_p (\mu, \nu) \le C_1 t^{1/2} + C_2(t) \bigg( \sum_{\ell=1}^{\infty} d_{\ell} e^{-\lambda_{\ell} q_0 t} \bigg( \frac{\sum_{m=1}^{d_{\ell}} |\widehat{\mu}(\ell,m) - \widehat{\nu}(\ell, m)|^2}{d_{\ell} \lambda_{\ell}} \bigg)^{q_0/2} \bigg)^{1/q_0}, \]
where $q_0 = \min \{ q, 2 \}$,
\[ \begin{split} C_1 &= \left\{ \begin{array}{ll} 2^{1/2} (1+(1-c)^{1/p}) d^{1/2} & \textrm{if } 1 \le p \le 2, \\ (1+(1-c)^{1/p}) (2d+p)^{1/2} & \textrm{if } p>2, \end{array} \right. \\ C_2(t) &=\left\{ \begin{array}{ll} 1 & \textrm{if } p=1, \\ \frac{p(c^{1/p}-\delta^{1/p})}{c-\delta} & \textrm{if } 1<p \le 2, \\ 2(p-1) \frac{p(c^{1/p}-\delta^{1/p})}{c-\delta} & \textrm{if } p>2 \end{array} \right. \qquad \textrm{with} \qquad \delta=\min \left\{ b \frac{e^{-r^2 /(4t)}}{(4 \pi t)^{d/2}} , c \right\} . \end{split} \]
\end{thm}
\noindent In the case $p>1$ and $\delta=c$, the fraction $\frac{p(c^{1/p} - \delta^{1/p})}{c-\delta}$ in the definition of the constant $C_2(t)$ is to be interpreted as
\[ \lim_{\delta \to c} \frac{p(c^{1/p} - \delta^{1/p})}{c-\delta} = \frac{1}{c^{1/q}} . \]
Note that $\sum_{m=1}^{d_{\ell}} |\widehat{\mu}(\ell,m) - \widehat{\nu} (\ell,m)|^2$ does not depend on the choice of the orthonormal basis $\phi_{\ell,m}$.

As an analogue of Theorems \ref{torusjacksonkerneltheorem} and \ref{liegroupjacksonkerneltheorem}, we also present a slightly different formulation in the special case of the Euclidean unit sphere $\mathbb{S}^d=\{ x \in \mathbb{R}^{d+1} : |x|=1 \}$. We use the standard normalization $\mathrm{diam}(\mathbb{S}^d)=\pi$. Note that on the sphere we have $\lambda_{\ell}=\ell (\ell+d-1)$ and $d_{\ell}=\binom{\ell+d}{d} - \binom{\ell+d-2}{d}$. The case $p=1$ in a weaker form is due to Grabner and Tichy \cite[Theorem 1]{GT}, whereas $p>1$ is new.
\begin{thm}\label{sphereprojectionkerneltheorem} Let $\mu, \nu \in \mathcal{P}(\mathbb{S}^d)$, $d \ge 2$, and assume there exists a constant $c \ge 0$ such that $\mu \ge c \mathrm{Vol}$. Let $1 \le p < \infty$ and $1<q \le \infty$ be such that $1/p+1/q=1$. Assume either $p=1$ or $c>0$. For any integer $L \ge 1$, we have
\[ W_p(\mu, \nu) \le \frac{C_1}{L} + C_2 \Bigg( \sum_{\ell=1}^L d_{\ell} \left( \frac{\sum_{m=1}^{d_{\ell}} |\widehat{\mu}(\ell,m) - \widehat{\nu}(\ell,m)|^2}{d_{\ell} \lambda_{\ell}} \right)^{q_0/2} \Bigg)^{1/q_0} , \]
where $q_0=\min \{ q, 2 \}$,
\[ \begin{split} C_1 &= 8 (1+(1-c)^{1/p}) 12^{1/p} d e^{d/p} \left( \frac{p+5}{2} \right)^{1+d/p}, \\ C_2 &=\left\{ \begin{array}{ll} 1 & \textrm{if } p=1, \\ p c^{-1/q} & \textrm{if } 1<p \le 2, \\ 2(p-1)p c^{-1/q} & \textrm{if } p>2. \end{array} \right. \end{split} \]
\end{thm}

Theorems \ref{homogeneousheatkerneltheorem} and \ref{sphereprojectionkerneltheorem} with $p=1$ and $c=0$ apply without any assumption on $\mu$, but $c>0$ is crucial in the case $p>1$. Theorem \ref{homogeneousheatkerneltheorem} with $b=0$ and $C_2(t)=pc^{-1/q}$ applies without any assumption on $\nu$, but the value of $C_2 (t)$ can be sharpened if $b>0$. Theorems \ref{homogeneousheatkerneltheorem} and \ref{sphereprojectionkerneltheorem} yield comparable results in the range $t \approx L^{-2}$.

As an analogue of Theorem \ref{torusinfinitytheorem}, we also treat $W_{\infty}$ on the sphere of dimension $d \ge 3$.
\begin{thm}\label{sphereinfinitytheorem} Let $\mu, \nu \in \mathcal{P}(\mathbb{S}^d)$, $d \ge 3$, and assume there exist constants $c>0$, $b>0$ and $0<r \le 2^{-d-3}d^{-1/2}$ such that $\mu \ge c \mathrm{Vol}$ and $\nu(B) \ge b$ for any closed geodesic ball $B \subseteq \mathbb{S}^d$ of radius $r$. For any real $2^{d+3}d^{-1/2} r \le T \le 1/d$, we have
\[ W_{\infty} (\mu, \nu) \le C_1 T + C_2(T) \sum_{\ell=1}^{\infty} A_{\ell} \left( \frac{d_{\ell}}{\lambda_{\ell}} \sum_{m=1}^{d_{\ell}} |\widehat{\mu}(\ell,m) - \widehat{\nu}(\ell,m)|^2 \right)^{1/2} , \]
where $C_1 = 1 + \mathds{1}_{\{ \mu \neq \mathrm{Vol} \}}$,
\[ \begin{split} C_2(T) &= \frac{\log c - \log \delta}{c-\delta} \quad \textrm{with} \quad \delta=  \min \left\{ b \frac{9^d}{114 d! 2^{d^2/4+d/4}T^d}, c \right\} \\ A_{\ell} &=\left\{ \begin{array}{ll} 1 & \textrm{if } \ell \le 2^{d+2}/T, \\ \displaystyle{43 \exp \left( - \frac{1}{\log 2} \left( \log \frac{T\ell}{16} \right) \left( \log \frac{T\ell}{2^{d+2}} \right) \right)} & \textrm{if } \ell > 2^{d+2}/T . \end{array} \right. \end{split} \]
\end{thm}
\noindent In the case $\delta =c$, the constant $C_2(T)$ is to be interpreted as
\[ C_2(T)= \lim_{\delta \to c} \frac{\log c - \log \delta}{c-\delta} = \frac{1}{c}. \]

\subsection{An application to spherical designs}\label{sphericaldesignsection}

A finite point set $a_1, a_2, \ldots, a_N$ in the Euclidean unit sphere $\mathbb{S}^d$ is called a spherical $t$-design if
\[ \frac{1}{N} \sum_{n=1}^N f(a_n) = \int_{\mathbb{S}^d} f(x) \, \mathrm{d}\mathrm{Vol}(x) \]
for all polynomials $f(x)=f(x_1, x_2, \ldots, x_{d+1})$ with real coefficients of degree at most $t$. An equivalent definition is that the previous formula holds for all spherical harmonics $f=\phi_{\ell,m}$ with $0 \le \ell \le t$ and $1 \le m \le d_{\ell}$. A seminal result of Bondarenko, Radchenko and Viazovska \cite{BRV} states that for any $d,t \in \mathbb{N}$, the smallest size $N(d,t)$ of a spherical $t$-design in $\mathbb{S}^d$ is of order $t^d \ll N(d,t) \ll t^d$ with implied constants depending only on $d$. Working with the empirical measure $\nu_N=N^{-1} \sum_{n=1}^N \delta_{a_n}$ of a spherical $t$-design, Grabner and Tichy \cite{GT} proved the sharp upper bound $W_1 (\nu_N, \mathrm{Vol}) \ll 1/t$. As an illustration of the possible applications of our results, we generalize their estimate to $W_p$, $1 \le p < \infty$.
\begin{cor}\label{sphericaldesigntheorem} Let $a_1, a_2, \ldots, a_N$ be a spherical $t$-design in $\mathbb{S}^d$, $d \ge 2$, and let $\nu_N=N^{-1} \sum_{n=1}^N \delta_{a_n}$. For any $1 \le p < \infty$, we have
\[ W_p \left( \nu_N , \mathrm{Vol} \right) \le \frac{C}{t} \]
with the constant $C=14d \max \{ d \log (100d), p \}$.
\end{cor}

\begin{proof} By the definition of a spherical $t$-design, $\widehat{\nu_N}(\ell,m) - \widehat{\mathrm{Vol}}(\ell,m)=0$ for all $0 \le \ell \le t$ and $1 \le m \le d_{\ell}$. An application of Theorem \ref{sphereprojectionkerneltheorem} with $\mu=\mathrm{Vol}$, $\nu = \nu_N$, $c=1$ and $L=t$ thus yields $W_p (\nu_N, \mathrm{Vol}) \le 8 \cdot 12^{1/p} d e^{d/p} \left( \frac{p+5}{2} \right)^{1+d/p}/t$. In particular, if $p \ge d \log (100d)$, then $W_p (\nu_N, \mathrm{Vol}) \le \kappa d p/t$ with
\[ \kappa = 8 \cdot 12^{1/(d \log (100d))} e^{1/\log (100d)} \left( \frac{1+5/(d \log (100d))}{2} \right)^{1+1/\log (100d)} (d \log (100d))^{1/\log (100d)} . \]
Note that $\kappa \to 4e=10.8731\ldots$ as $d \to \infty$. Numerical computations show that the maximum of $\kappa$ over $d \ge 2$ is attained at $d=2$ with maximum value $13.2461\ldots$. Hence $W_p (\nu_N, \mathrm{Vol}) \le 14dp/t$, as claimed. If $1 \le p < d \log (100d)$, then $W_p (\nu_N, \mathrm{Vol}) \le W_{d \log (100d)}(\nu_N, \mathrm{Vol}) \le 14 d^2 \log (100d)/t$, as claimed.
\end{proof}
In particular, if the spherical $t$-design has optimally small size $N \ll t^d$, then $W_p (\nu_N, \mathrm{Vol}) \ll N^{-1/d}$. The rate $N^{-1/d}$ is optimal, in fact $W_p(\nu, \mathrm{Vol}) \gg N^{-1/d}$ for any $\nu \in \mathcal{P}(\mathbb{S}^d)$ supported on at most $N$ points \cite{KL}. This proves that the upper bound $C/t$ in Corollary \ref{sphericaldesigntheorem} is sharp.

\subsection{Smoothing inequalities on compact Riemannian manifolds}\label{manifoldsection}

Let $M$ be a compact, connected, smooth Riemannian manifold of dimension $1 \le d < \infty$ without boundary. In order to have a notation that is consistent with Sections \ref{liegroupsection} and \ref{homogeneoussection}, where $\lambda$ was allowed to have multiplicity, we now denote the Laplace eigenvalues by $\Lambda$. That is, let $0=\Lambda_0<\Lambda_1 \le \Lambda_2 \le \ldots$ denote the list of eigenvalues of $-\Delta$, each repeated according to its multiplicity, with a corresponding orthonormal basis $\{ \phi_k : k \ge 0 \}$ of complex-valued functions in $L^2(M)$ such that $\Delta \phi_k = - \Lambda_k \phi_k$. Let
\[ P_t (x,y) = \sum_{k=0}^{\infty} e^{-\Lambda_k t} \phi_k (x) \overline{\phi_k (y)}, \qquad t>0, \,\, x,y \in M \]
denote the heat kernel on $M$. The Fourier transform of $f \in L^1(M)$ resp.\ $\mu \in \mathcal{P}(M)$ is defined as $\widehat{f}(k)=\int_M f \overline{\phi_k} \, \mathrm{d} \mathrm{Vol}$ resp.\ $\widehat{\mu}(k)=\int_M \overline{\phi_k} \, \mathrm{d}\mu$, $k \ge 0$.

We present the generalization of Theorems \ref{torusheatkerneltheorem}, \ref{liegroupheatkerneltheorem} and \ref{homogeneousheatkerneltheorem} to compact manifolds in the ranges $1 \le p \le 2$ and $2<p<\infty$ separately. The case $p=2$ is due to Brown and Steinerberger \cite{BS} and the first author \cite{BOR2}, whereas the case $p \neq 2$ is new. In particular, Theorem \ref{manifoldheatkernelp<2theorem} with $p=1$ generalizes a result of Bobkov and Ledoux \cite{BL1,BL2} from the torus to compact manifolds.
\begin{thm}\label{manifoldheatkernelp<2theorem} Let $M$ be a compact, connected, smooth Riemannian manifold of dimension $1 \le d < \infty$ without boundary, and assume the Ricci curvature is bounded below by $-(d-1)A$ with some constant $A \ge 0$. Let $\mu, \nu \in \mathcal{P}(M)$, and assume there exist constants $c \ge 0$, $b \ge 0$ and $r>0$ such that $\mu \ge c \mathrm{Vol}$ and $\nu (B) \ge b$ for any closed geodesic ball $B \subseteq M$ of radius $r$. Let $1 \le p \le 2$, and assume either $p=1$ or $c>0$. For any real $t>0$, we have
\[ W_p (\mu, \nu) \le C_1 t^{1/2} (1+C_3 t^{1/2})^{1/2} + C_2 (t) \Bigg( \sum_{k=1}^{\infty} e^{- 2 \Lambda_k t} \frac{|\widehat{\mu}(k) - \widehat{\nu}(k)|^2}{\Lambda_k} \Bigg)^{1/2} , \]
where
\[ \begin{split} C_1 &= 2^{1/2} (1+(1-c)^{1/p}) d^{1/2}, \\ C_2(t) &=\left\{ \begin{array}{ll} 1 & \textrm{if } p=1, \\ \frac{p(c^{1/p}-\delta^{1/p})}{c-\delta} & \textrm{if } 1<p \le 2 \end{array} \right. \qquad \textrm{with} \qquad \delta=\min \left\{ b \inf_{\substack{x,y \in M \\ \rho (x,y) \le r}} P_t (x,y) , c \right\} , \\ C_3 &= \frac{2^{3/2}(d-1) \sqrt{A}}{3d} \left( d + (d-1) \sqrt{A} \, \mathrm{diam}(M) \right)^{1/2} . \end{split} \]
\end{thm}

If the Ricci curvature of $M$ is bounded below by $-(d-1)A$ with some constant $A \ge 0$, then the heat kernel satisfies the explicit lower bound \cite[Theorem 3.1]{WA}
\begin{equation}\label{heatkernellowerbound}
P_t (x,y) \ge \sup_{\sigma >0} \frac{1}{(4 \pi t)^{d/2}} \exp \left( - \left( \frac{1}{4t} + \frac{\sigma}{\sqrt{18t}} \right) \rho (x,y)^2 - \frac{(d-1)^2 A t}{4} - \left( \frac{(d-1)^2 A}{4 \sigma} + \frac{2d \sigma}{3} \right) \sqrt{2t} \right) .
\end{equation}
In particular, the value of $\delta$, and consequently also $C_2(t)$, can be made fully explicit in terms of $A$.

We now introduce two further constants from two classical inequalities. The Poincar\'e inequality \cite[p.\ 24]{HE} states that for any $f \in \mathrm{C}^{\infty} (M)$ such that $\int_M f \, \mathrm{d} \mathrm{Vol}=0$ and any $1 \le q < d$,
\begin{equation}\label{poincare}
\| f \|_q \le K_{\mathrm{Poincar\acute{e}}} (M,q) \| \nabla f \|_q
\end{equation}
with a constant $K_{\mathrm{Poincar\acute{e}}} (M,q)$ depending only on the manifold $M$ and $q$. A pointwise version of Weyl's law \cite[Theorem 3.3.1]{SO} states that for all integers $\ell \ge 0$ and all $x \in M$,
\begin{equation}\label{weyllaw}
\sum_{\substack{k \ge 0 \\ \Lambda_k^{1/2} \in [\ell, \ell+1)}} |\phi_k (x)|^2 \le K_{\mathrm{Weyl}}(M) (\ell+1)^{d-1}
\end{equation}
with a constant $K_{\mathrm{Weyl}}(M)$ depending only on the manifold $M$.

\begin{thm}\label{manifoldheatkernelp>2theorem} Let $M$ be a compact, connected, smooth Riemannian manifold of dimension $1 \le d < \infty$ without boundary, and assume the Ricci curvature is bounded below by $-(d-1)A$ with some constant $A \ge 0$. Let $\mu, \nu \in \mathcal{P}(M)$, and assume there exist constants $c>0$, $b \ge 0$ and $r>0$ such that $\mu \ge c \mathrm{Vol}$ and $\nu (B) \ge b$ for any closed geodesic ball $B \subseteq M$ of radius $r$. Let $2<p<\infty$ and $2<q<\infty$ be such that $1/p+1/q=1$. For any real $t>0$, we have
\[ \begin{split} W_p (\mu, \nu) &\le C_1 t^{1/2} (1+C_3 t^{1/2})^{1/p} \\ &{\phantom{\le}}+ C_2(t) \Bigg( \sum_{\ell=0}^{\infty} (\ell+1)^{(p-2)(d-1)/(2p-2)} \Bigg( \sum_{\substack{k \ge 1 \\ \Lambda_k^{1/2} \in [\ell, \ell+1)}} e^{-2\Lambda_k t} \frac{|\widehat{\mu}(k) - \widehat{\nu}(k)|^2}{\Lambda_k+(d-1)A} \Bigg)^{q/2} \Bigg)^{1/q}, \end{split} \]
where
\[ \begin{split} C_1 &= (1+(1-c)^{1/p}) (2d+p)^{1/2}, \\ C_2(t)&= K_{\mathrm{Weyl}}(M)^{(p-2)/(2p)} \left( 3 \sqrt{6} (p-1) + (d-1)^{1/2}A^{1/2} K_{\mathrm{Poincar\acute{e}}}(M,q) \right) \frac{p(c^{1/p}-\delta^{1/p})}{c-\delta}, \\ \delta &=\min \left\{ b \inf_{\substack{x,y \in M \\ \rho (x,y) \le r}} P_t (x,y) , c \right\} , \\ C_3 &= 2 (d-1) \sqrt{A} \exp \left( \frac{p+2}{2} + \frac{(d-1) \sqrt{A}}{2} \, \mathrm{diam}(M) \right) . \end{split} \]
If $A=0$, then the factor $3 \sqrt{6}$ in the definition of $C_2(t)$ can be replaced by $2$.
\end{thm}

\noindent In Theorems \ref{manifoldheatkernelp<2theorem} and \ref{manifoldheatkernelp>2theorem} in the case $\delta=c$, the fraction $\frac{p(c^{1/p} - \delta^{1/p})}{c-\delta}$ in the definition of the constant $C_2(t)$ is to be interpreted as
\[ \lim_{\delta \to c} \frac{p(c^{1/p}-\delta^{1/p})}{c-\delta} = \frac{1}{c^{1/q}} . \]

Theorem \ref{manifoldheatkernelp<2theorem} with $p=1$ and $c=0$ applies without any assumption on $\mu$, but $c>0$ is crucial in the case $p>1$. Theorems \ref{manifoldheatkernelp<2theorem} and \ref{manifoldheatkernelp>2theorem} with $b=0$ apply without any assumption on $\nu$, but the value of $C_2(t)$ can be sharpened if $b>0$.

\section{Approach}\label{approachsection}

The proofs of all our results follow the same general approach, with two main ingredients. First, a smoothing procedure with a suitable nonnegative kernel $K$ on $M \times M$. Second, the Kantorovich duality theorem expressed in terms of the Hopf--Lax semigroup $Q_t$, $t \ge 0$ defined with the Lagrangian $L(x)=x^p/p$ and acting on $f \in \mathrm{C}(M)$. The function $Q_t f$ roughly speaking satisfies the Hamilton--Jacobi equation
\[ \frac{\partial Q_t f}{\partial t} + \frac{|\nabla Q_t f|^q}{q} = 0, \]
even though $Q_t f$ is only Lipschitz, but not necessarily differentiable. The precise interpretation of the Hamilton--Jacobi equation in the previous formula thus involves some technicalities, which are handled using the nonsmooth analysis approach of Villani \cite{VI}.

We work out the details of these two main ingredients in the generality of compact Riemannian manifolds $M$, and throughout this section we use the assumptions and notations given at the end of the Introduction. The main results of this section are Lemmas \ref{smoothinglemma} and \ref{ledouxlemma}. Together they reduce the problem of bounding $W_p (\mu, \nu)$ to estimating two quantities: the dispersion rate $D_p (K)$ of the kernel, and the dual Sobolev norm $\| K*\mu - K*\nu \|_{\dot{H}_{-1}^p}$, see \eqref{dispersion} and \eqref{dualsobolevnorm} for the definitions. We deal with these quantities on a case by case basis depending on the space and the kernel in Section \ref{proofsection}.

\subsection{Smoothing procedure}

The convolution of a bounded, Borel measurable kernel $K: M \times M \to [0,\infty)$ and $\mu \in \mathcal{P}(M)$ is defined as the Borel measure
\[ (K*\mu)(B) = \int_M \int_M K(x,y) \mathds{1}_B (x) \, \mathrm{d} \mathrm{Vol}(x) \mathrm{d} \mu (y), \qquad B \subseteq M \textrm{ Borel}. \]
In particular, $K*\mu \ll \mathrm{Vol}$, and
\[ \frac{\mathrm{d}(K*\mu)}{\mathrm{d}\mathrm{Vol}} (x) = \int_M K(x,y) \, \mathrm{d}\mu(y), \qquad x \in M. \]
We define the dispersion rate of the kernel $K$ as
\begin{equation}\label{dispersion}
D_p (K) = \left\{ \begin{array}{ll} \displaystyle{\sup_{y \in M} \left( \int_M K(x,y) \rho(x,y)^p \, \mathrm{d}\mathrm{Vol}(x) \right)^{1/p}} & \textrm{if } 1 \le p < \infty, \\ \displaystyle{\sup \{ \rho (x,y) : x,y \in M, \, K(x,y)>0 \}} & \textrm{if } p=\infty . \end{array} \right.
\end{equation}
The first step in bounding $W_p (\mu, \nu)$ is to replace $\mu$ and $\nu$ by the smoothed measures $K*\mu$ and $K*\nu$.
\begin{lem}\label{smoothinglemma} Let $K: M \times M \to [0,\infty)$ be a bounded, Borel measurable function that satisfies
\[ \begin{split} \int_M K(x,y) \, \mathrm{d}\mathrm{Vol}(x) &= 1 \quad \textrm{for all } y \in M, \\ \int_M K(x,y) \, \mathrm{d}\mathrm{Vol}(y) &= 1 \quad \textrm{for all } x \in M . \end{split} \]
Let $\mu, \nu \in \mathcal{P}(M)$, and assume $\mu \ge c \mathrm{Vol}$ with some constant $c \ge 0$. For any $1 \le p < \infty$, we have
\[ W_p (\mu, \nu) \le (1+(1-c)^{1/p}) D_p (K) + W_p (K*\mu, K*\nu) , \]
and
\[ W_{\infty} (\mu, \nu) \le (1+\mathds{1}_{\{ \mu \neq \mathrm{Vol} \}}) D_{\infty} (K) + W_{\infty} (K*\mu, K*\nu) . \]
\end{lem}

\begin{proof} First of all observe that $K*\mu, K*\nu \in \mathcal{P}(M)$ and $K*\mathrm{Vol}=\mathrm{Vol}$. Let $1 \le p < \infty$. The triangle inequality for $W_p$ shows that
\[ W_p (\mu, \nu) \le W_p (\mu, K*\mu) + W_p (K*\mu, K*\nu) + W_p (K*\nu, \nu) . \]
To estimate the last term, let us use the coupling $\vartheta \in \mathcal{P}(M \times M)$,
\[ \vartheta (A) = \int_M \int_M K(x,y) \mathds{1}_A (x,y) \, \mathrm{d}\mathrm{Vol}(x) \mathrm{d}\nu (y), \quad A \subseteq M \times M \textrm{ Borel}, \]
whose marginals are $\vartheta (B \times M) = (K*\nu) (B)$ and $\vartheta (M \times B) = \nu(B)$, $B \subseteq M$ Borel. In particular, $\vartheta \ll \mathrm{Vol} \otimes \nu$ on $M \times M$, and $\frac{\mathrm{d}\vartheta}{\mathrm{d}(\mathrm{Vol} \otimes \nu)}=K$. By the definition of the Wasserstein metric,
\[ \begin{split} W_p (K*\nu, \nu) &\le \left( \int_{M \times M} \rho(x,y)^p \, \mathrm{d}\vartheta (x,y) \right)^{1/p} \\ &= \left( \int_M \int_M K(x,y) \rho(x,y)^p \, \mathrm{d} \mathrm{Vol}(x) \mathrm{d}\nu (y) \right)^{1/p} \\ &\le \sup_{y \in M} \left( \int_M K(x,y) \rho(x,y)^p \, \mathrm{d} \mathrm{Vol}(x) \right)^{1/p} \\ &= D_p (K) . \end{split} \]
The assumption $\mu \ge c \mathrm{Vol}$ ensures that we can write $\mu$ and $K*\mu$ as the convex combinations $\mu=(1-c) \mu' + c \mathrm{Vol}$ and $K*\mu = (1-c) (K*\mu') + c (K* \mathrm{Vol})$ with $\mu'=(\mu - c \mathrm{Vol})/(1-c) \in \mathcal{P}(M)$. By the convexity of the optimal cost \cite[p.\ 59]{VI},
\[ W_p^p (\mu, K*\mu) \le (1-c) W_p^p (\mu', K*\mu') + c W_p^p (\mathrm{Vol}, K*\mathrm{Vol}) = (1-c) W_p^p (\mu', K*\mu') .  \]
Repeating the argument above with $\nu$ replaced by $\mu'$ leads to $W_p (\mu, K*\mu) \le (1-c)^{1/p} D_p (K)$. This finishes the proof for $1 \le p < \infty$. The claim for $p=\infty$ follows from taking the limit as $p\to \infty$, noting that the assumption $\mu \ge c \mathrm{Vol}$ holds with $c=1$ if and only if $\mu = \mathrm{Vol}$.
\end{proof}

\begin{remark} Lemma \ref{smoothinglemma} holds with $1+(1-c)^{1/p}$ and $1+\mathds{1}_{\{ \mu \neq \mathrm{Vol} \}}$ replaced by $2$ without the assumption
\[ \int_M K(x,y) \, \mathrm{d} \mathrm{Vol}(y) =1 \quad \textrm{for all } x \in M. \]
Indeed, this assumption was only used in the proof to show that $K*\mathrm{Vol}=\mathrm{Vol}$.
\end{remark}

\subsection{Sobolev norms and Ledoux's lemma}

Let $1< p \le \infty$ and $1 \le q < \infty$ be such that $1/p+1/q=1$. Let us define the Sobolev norm $\| \cdot \|_{\dot{H}_1^q}$ as
\[ \| f \|_{\dot{H}_1^q} = \bigg( \int_M |\nabla f|^q \, \mathrm{d} \mathrm{Vol} \bigg)^{1/q} . \]
This is well-defined and finite for all $f \in \mathrm{Lip}(M)$ by Rademacher's theorem \cite[p.\ 234]{VI}. The dual Sobolev norm $\| \cdot \|_{\dot{H}_{-1}^p}$ is then defined on the set of signed Borel measures $\gamma$ on $M$ such that $\gamma(M)=0$ as
\begin{equation}\label{dualsobolevnorm}
\| \gamma \|_{\dot{H}_{-1}^p} = \sup_{\substack{f \in \mathrm{Lip}(M) \\ \| f \|_{\dot{H}_1^q} \le 1}} \left| \int_M f \, \mathrm{d}\gamma \right| .
\end{equation}

The following lemma is a variation of a result of Ledoux \cite[Theorem 2]{LE1}, and can be thought of as a generalization of the Kantorovich duality formula \eqref{kantorovichduality} to $p>1$. The special case $p=2$ also appeared in a seemingly unpublished sequel to \cite{LE1}, and in \cite{PEY}. Note that Ledoux uses a different definition of Sobolev norms. The lemma with $\delta=0$ and finite $p$ applies without any assumption on $\nu$. In contrast, the assumption $\delta>0$ is crucial in the case $p=\infty$.
\begin{lem}\label{ledouxlemma} Let $\mu, \nu \in \mathcal{P}(M)$. Assume there exist constants $c>\delta \ge 0$ such that $\mu \ge c \mathrm{Vol}$ and $\nu \ge \delta \mathrm{Vol}$. For any $1<p<\infty$, we have
\[ W_p (\mu, \nu) \le \frac{p(c^{1/p} - \delta^{1/p})}{c-\delta} \| \mu - \nu \|_{\dot{H}_{-1}^p} .  \]
If in addition $\delta>0$, then
\[ W_{\infty} (\mu, \nu) \le \frac{\log c - \log \delta}{c-\delta} \| \mu - \nu \|_{\dot{H}_{-1}^{\infty}} . \]
\end{lem}

\begin{proof} Let $1<p<\infty$. Consider the Lagrangian $L(x)=x^p/p$ on $[0,\infty)$, and its Legendre transform $L^*(x)=x^q/q$ also on $[0,\infty)$. The corresponding Hopf--Lax semigroup $Q_t$, $t \ge 0$ acting on $f \in \mathrm{C}(M)$ is defined as $Q_0 f = f$, and
\[ (Q_t f)(x) = \inf_{y \in M} \left( f(y) + t L \left( \frac{\rho (x,y)}{t} \right) \right), \qquad t>0, \,\, x \in M. \]
The Kantorovich duality theorem \cite[p.\ 70]{VI} applied to the cost function $L(\rho(x,y))$ expresses $W_p(\mu, \nu)$ in terms of the Hopf--Lax semigroup at $t=1$ as
\begin{equation}\label{hopflax}
\frac{1}{p}W_p^p (\mu, \nu) = \sup_{f \in \mathrm{C}(M)} \left( \int_M Q_1 f \, \mathrm{d}\nu - \int_M f \, \mathrm{d}\mu  \right) .
\end{equation}

As a minor modification of a theorem of Villani \cite[p.\ 626]{VI}, we claim that the Hopf--Lax semigroup has the following basic properties for any given $f \in \mathrm{C}(M)$.
\begin{enumerate}[(i)]
\item For any $t>0$, we have $Q_t f \in \mathrm{Lip}(M)$ and
\begin{equation}\label{lipschitzconstant}
\| Q_t f \|_{\mathrm{Lip}} \le \left( \frac{2 \, \mathrm{diam}(M)}{t} \right)^{p-1}.
\end{equation}

\item For any $t \ge 0$, we have $Q_s f \to Q_t f$ uniformly on $M$ as $s \to t$.

\item For any $t,s>0$ and any $x \in M$,
\[ \frac{|(Q_{t+s} f)(x) - (Q_t f)(x)|}{s} \le L^* \left( \| Q_t f \|_{\mathrm{Lip}} \right) . \]

\item For any $t>0$ and any $x \in M$ such that $Q_t f$ is differentiable at $x$,
\[ \lim_{s \to 0^+} \frac{(Q_{t+s} f)(x) - (Q_t f)(x)}{s} = - L^* (|\nabla Q_t f(x)|) . \]
\end{enumerate}

Let us first prove (i). Let $t>0$ and $x,y,z \in M$. The estimate
\[ \rho (x,z)^p \le (\rho (y,z) + \rho (x,y))^p \le \rho (y,z)^p + p (2 \, \mathrm{diam}(M))^{p-1} \rho (x,y) \]
shows that
\[ f(z) + \frac{\rho(x,z)^p}{p t^{p-1}} \le f(z) + \frac{\rho(y,z)^p}{p t^{p-1}} + \left( \frac{2 \, \mathrm{diam}(M)}{t} \right)^{p-1} \rho(x,y) . \]
Taking the infimum over $z \in M$ yields
\[ (Q_t f)(x) - (Q_t f)(y) \le \left( \frac{2 \, \mathrm{diam}(M)}{t} \right)^{p-1} \rho(x,y) . \]
Repeating the same argument with the roles of $x$ and $y$ reversed establishes \eqref{lipschitzconstant}, and thus proves (i).

We refer to Villani \cite[p.\ 626]{VI} for the proof of (iii) under the sole assumption that the Lagrangian $L$ is strictly increasing, convex and continuous.

Regarding property (ii), Villani showed that $Q_s f \to Q_t f$ uniformly on $M$ as $s \to t^+$ for any given $t \ge 0$, again under the sole assumption that the Lagrangian $L$ is strictly increasing, convex and continuous. Now let $0<s<t$ and $x \in M$. Property (iii) shows that
\[ |(Q_t f)(x) - (Q_s f)(x)| \le (t-s) L^* (\| Q_s f \|_{\mathrm{Lip}}) , \]
and here $\| Q_s f \|_{\mathrm{Lip}}$ remains bounded e.g.\ for $s \in (t/2,t)$ by \eqref{lipschitzconstant}. Therefore $Q_s f \to Q_t f$ uniformly on $M$ as $s \to t^-$ for any given $t>0$. This finishes the proof of property (ii).

Finally, consider property (iv). Villani showed that if $L$ is strictly increasing, convex, continuous and (after extending $L$ as constant zero on $(-\infty, 0)$) has locally bounded second derivative, then
\[ \lim_{s \to 0^+} \frac{(Q_{t+s} f)(x) - (Q_t f)(x)}{s} = - L^* (|\nabla^{-} Q_t f|(x)) \]
for all $t>0$ and $x \in M$, where given a function $g: M \to \mathbb{R}$, the function $|\nabla^{-}g|: M \to \mathbb{R} \cup \{ \infty \}$ is defined as
\[ |\nabla^{-} g|(x) = \limsup_{y \to x} \frac{\max \{ g(x)-g(y), 0 \}}{\rho (x,y)} . \]
Following the proof given by Villani, it turns out that the assumption that $L$ has a locally bounded second derivative (which in our case fails for $1<p<2$) can be replaced by the assumption that $Q_t f$ is differentiable at $x$, in which case $|\nabla^{-}Q_t f|(x) =|\nabla Q_t f(x)|$ with the usual gradient $\nabla$. This finishes the proof of (iv).

Now fix $f \in \mathrm{C}(M)$, and let $\theta: [0,1] \to [0,1]$ be a strictly increasing function that is continuous on $[0,1]$, smooth on $(0,1)$, and satisfies $\theta (0)=0$ and $\theta(1)=1$. Consider the function
\[ F(t) = (1-\theta(t)) \int_M Q_t f \, \mathrm{d} \mu + \theta (t) \int_M Q_t f \, \mathrm{d} \nu, \qquad t \in [0,1] . \]
Property (ii) ensures that $F$ is continuous on $[0,1]$. Noting that the integrands are nonpositive, an application of Fatou's lemma, the assumption $\mu \ge c \mathrm{Vol}$ and property (iv) together with Rademacher's theorem show that for any $t>0$,
\[ \begin{split} \limsup_{s \to 0^+} \int_M \frac{Q_{t+s}f - Q_t f}{s} \, \mathrm{d}\mu &\le \int_M \limsup_{s \to 0^+} \frac{Q_{t+s}f - Q_t f}{s} \, \mathrm{d}\mu \\ &\le c \int_M \limsup_{s \to 0^+} \frac{Q_{t+s}f - Q_t f}{s} \, \mathrm{d}\mathrm{Vol} \\ &= - c \int_M \frac{|\nabla Q_t f|^q}{q} \, \mathrm{d} \mathrm{Vol}. \end{split} \]
The same holds for $\nu$ with $c$ replaced by $\delta$. In particular, for all $t \in (0,1)$, the upper right-hand derivative $F'_+(t)=\limsup_{s \to 0^+} (F(t+s)-F(t))/s$ satisfies
\[ \begin{split} F'_+ (t) &\le -\theta'(t) \int_M Q_t f \, \mathrm{d} \mu + (1-\theta(t)) \limsup_{s \to 0^+} \int_M \frac{Q_{t+s}f - Q_t f}{s} \, \mathrm{d}\mu \\ &{\phantom{\le}} + \theta'(t) \int_M Q_t f \, \mathrm{d} \nu + \theta(t) \limsup_{s \to 0^+} \int_M \frac{Q_{t+s}f - Q_t f}{s} \, \mathrm{d}\nu \\ &\le \theta'(t) \int_M Q_t f \, \mathrm{d} (\nu - \mu) - (c-(c-\delta) \theta(t)) \int_M \frac{|\nabla Q_t f|^q}{q} \, \mathrm{d}\mathrm{Vol}. \end{split} \]
By the definition \eqref{dualsobolevnorm} of the dual Sobolev norm,
\[ \int_M Q_t f \, \mathrm{d} (\nu - \mu) \le \| \mu - \nu \|_{\dot{H}_{-1}^p} \| Q_t f \|_{\dot{H}_1^q} , \]
therefore
\[ F'_+(t) \le \theta'(t) \| \mu - \nu \|_{\dot{H}_{-1}^p} \| Q_t f \|_{\dot{H}_1^q} - \frac{c-(c-\delta) \theta(t)}{q} \| Q_t f \|_{\dot{H}_1^q}^q . \]
Treating $\| Q_t f \|_{\dot{H}_1^q}$ as a variable $u$, we deduce
\[ \begin{split} F'_+(t) &\le (c-(c-\delta)\theta(t)) \sup_{u \ge 0} \left( \frac{\theta'(t) \| \mu - \nu \|_{\dot{H}_{-1}^p}}{c-(c-\delta)\theta(t)} u - L^* (u) \right) \\ &= (c-(c-\delta)\theta(t)) L \left( \frac{\theta'(t) \| \mu - \nu \|_{\dot{H}_{-1}^p}}{c-(c-\delta)\theta(t)} \right) \\ &= \frac{(\theta' (t))^p}{p (c-(c-\delta)\theta (t))^{p-1}} \| \mu - \nu \|_{\dot{H}_{-1}^p}^p . \end{split} \]
The optimal choice for the smooth function $\theta$ is
\[ \theta (t) = \frac{c-(c^{1/p}-(c^{1/p}-\delta^{1/p})t)^p}{c-\delta}, \]
in which case we arrive at the constant function
\[ \frac{(\theta'(t))^p}{p(c-(c-\delta)\theta(t)))^{p-1}} = \frac{p^{p-1} (c^{1/p}-\delta^{1/p})^{p}}{(c-\delta)^p} . \]
The function $F$ is thus continuous on $[0,1]$, and its upper right-hand derivative is bounded above by
\[ F'_+ (t) \le \frac{p^{p-1} (c^{1/p}-\delta^{1/p})^{p}}{(c-\delta)^p} \| \mu - \nu \|_{\dot{H}_{-1}^p}^p , \qquad t \in (0,1). \]
It follows that
\[ \int_M Q_1 f \, \mathrm{d} \nu - \int_M f \, \mathrm{d}\mu = F(1)-F(0) \le \frac{p^{p-1} (c^{1/p}-\delta^{1/p})^{p}}{(c-\delta)^p} \| \mu - \nu \|_{\dot{H}_{-1}^p}^p . \]
As $f \in \mathrm{C}(M)$ was arbitrary, the Kantorovich duality theorem \eqref{hopflax} finally leads to
\[ \frac{1}{p} W_p^p (\mu, \nu) \le \frac{p^{p-1} (c^{1/p}-\delta^{1/p})^{p}}{(c-\delta)^p} \| \mu - \nu \|_{\dot{H}_{-1}^p}^p . \]
This finishes the proof for $1<p<\infty$.

Note that $\lim_{p \to \infty} W_p (\mu, \nu) = W_{\infty}(\mu, \nu)$ and $\| \mu - \nu \|_{\dot{H}_{-1}^p} \le \| \mu - \nu \|_{\dot{H}_{-1}^{\infty}}$ for all $1<p<\infty$ by the definition \eqref{dualsobolevnorm} of the dual Sobolev norm. Assuming $\delta>0$, we also have $\lim_{p \to \infty} p(c^{1/p} - \delta^{1/p}) = \log c - \log \delta$. Taking the limit as $p \to \infty$ thus yields the desired estimate for $W_{\infty} (\mu, \nu)$.
\end{proof}

We note for convenience that in the case of absolutely continuous measures with bounded density it is enough to take the supremum over smooth functions instead of Lipschitz functions in the definition of the dual Sobolev norm \eqref{dualsobolevnorm} and in the Kantorovich duality formula \eqref{kantorovichduality}.
\begin{lem}\label{dualsobolevnormlemma} Let $\mu, \nu \in \mathcal{P}(M)$ be absolutely continuous with bounded density with respect to $\mathrm{Vol}$. For any $1< p \le \infty$ and $1 \le q < \infty$ such that $1/p+1/q=1$, we have
\[ \| \mu - \nu \|_{\dot{H}_{-1}^p} = \sup_{\substack{f \in \mathrm{C}^{\infty}(M) \\ \| f \|_{\dot{H}_1^q} \le 1}} \left| \int_M f \, \mathrm{d} \mu - \int_M f \, \mathrm{d} \nu \right| , \]
and
\[ W_1 (\mu, \nu) = \sup_{\substack{f \in \mathrm{C}^{\infty}(M) \\ \| f \|_{\mathrm{Lip}} \le 1}} \left| \int_M f \, \mathrm{d}\mu - \int_M f \, \mathrm{d}\nu \right| . \]
\end{lem}

\begin{proof} The first claim follows from the fact that for any $f \in \mathrm{Lip}(M)$, there exists a sequence $f_n \in \mathrm{C}^{\infty}(M)$ such that $f_n \to f$ both in $L^1 (M)$ and in the $\| \cdot \|_{\dot{H}_1^q}$ Sobolev norm \cite[p.\ 11]{HE}. The second claim follows from the fact that for any $f \in \mathrm{Lip}(M)$, there exists a sequence $f_n \in \mathrm{C}^{\infty}(M)$ such that $f_n \to f$ uniformly on $M$ and $\limsup_{n \to \infty} \| f_n \|_{\mathrm{Lip}} \le \| f \|_{\mathrm{Lip}}$, see \cite[Proposition 2.1]{GW}.
\end{proof}

\subsection{Riesz transform estimates}

As Lemma \ref{ledouxlemma} and the definition \eqref{dualsobolevnorm} of the dual Sobolev norm suggest, we have to work with $\| \nabla f \|_q$, $f \in \mathrm{C}^{\infty}(M)$. This is straightforward on the torus and more generally, on compact connected Lie groups, since these spaces are parallelizable: this means that the gradient $\nabla f$ can be thought of simply as a function taking values from a given vector space, e.g.\ the tangent space at the identity element of the group. In contrast, in more general manifolds the gradient $\nabla f$ is technically a section of the tangent bundle. In order to avoid this technical complication, in the proof of our results on compact homogeneous spaces and general compact manifolds we rely on the following Riesz transform estimates. 
\begin{lem}[Ba\~{n}uelos--Os\c{e}kowski] Assume the Ricci curvature of $M$ is bounded below by $-a$ with some constant $a \ge 0$. Let $1<p<\infty$ and $1<q<\infty$ be such that $1/p+1/q=1$, and let $p^*=\max \{ p, q \}$. If $a=0$, then for any $f \in \mathrm{C}^{\infty}(M)$, we have
\begin{equation}\label{riesztransform1}
\frac{\| (-\Delta)^{1/2} f \|_p}{2(p^*-1)} \le \| \nabla f \|_p \le 2(p^*-1) \| (-\Delta)^{1/2} f \|_p .
\end{equation}
If $a>0$, then for any $f \in \mathrm{C}^{\infty}(M)$, we have
\begin{equation}\label{riesztransform2}
\frac{\| (a-\Delta)^{1/2} f \|_p - a^{1/2} \| f \|_p}{3 \sqrt{6}(p^*-1)} \le \| \nabla f \|_p \le 3 \sqrt{6} (p^*-1) \| (a-\Delta)^{1/2} f \|_p .
\end{equation}
\end{lem}

The first result of this type with unspecified constants depending on $p$ was proved by Bakry \cite{BA} during his investigations of the Riesz transform $\nabla (-\Delta)^{-1/2}$. Improving his result by finding explicit values of the constants, the second inequalities in \eqref{riesztransform1} and \eqref{riesztransform2} were proved by Ba\~{n}uelos and Os\c{e}kowski \cite{BO}. See also \cite{AR} for the special case of the sphere. The first inequalities in \eqref{riesztransform1} and \eqref{riesztransform2} then follow from a standard duality argument.

For the sake of completeness, we include the duality argument needed to deduce the first inequalities from the second inequalities in \eqref{riesztransform1} and \eqref{riesztransform2} following Bakry \cite{BA}. First, assume $a=0$. Since the set $\{ (-\Delta)^{1/2} h : h \in \mathrm{C}^{\infty}(M) \}$ is dense in $\{ g \in L^q(M,\mathbb{R}) : \int_M g \, \mathrm{d}\mathrm{Vol}=0 \}$, we have
\[ \begin{split} \| (-\Delta)^{1/2} f \|_p &= \sup_{\substack{g \in L^q (M,\mathbb{R}) \\ \| g \|_q \le 1, \,\, \int_M g \, \mathrm{Vol}=0}} \int_M g (-\Delta)^{1/2} f \, \mathrm{d}\mathrm{Vol} \\ &= \sup_{\substack{h \in \mathrm{C}^{\infty} (M) \\ \| (-\Delta)^{1/2} h \|_q \le 1}} \int_M \left( (-\Delta)^{1/2} h \right) \left( (-\Delta)^{1/2} f \right) \, \mathrm{d} \mathrm{Vol} \\ &= \sup_{\substack{h \in \mathrm{C}^{\infty} (M) \\ \| (-\Delta)^{1/2} h \|_q \le 1}} \int_M \langle \nabla h, \nabla f \rangle \, \mathrm{d}\mathrm{Vol} \\ &\le \sup_{\substack{h \in \mathrm{C}^{\infty} (M) \\ \| (-\Delta)^{1/2} h \|_q \le 1}} \| \nabla h \|_q \| \nabla f \|_p . \end{split} \]
Note that $\langle \cdot, \cdot \rangle$ denotes the Riemannian metric on the tangent spaces of $M$, and we used Green's formula and the fact that $(-\Delta)^{1/2}$ is self-adjoint. The second inequality in \eqref{riesztransform1} shows that here $\| \nabla h \|_q \le 2(p^*-1)$, which proves the first inequality $\| (-\Delta)^{1/2} f \|_p \le 2(p^*-1) \| \nabla f \|_p$ in \eqref{riesztransform1}.

Next, assume $a>0$. Since $\{ (a-\Delta)^{1/2} h : h \in \mathrm{C}^{\infty}(M) \}$ is dense in $L^q (M,\mathbb{R})$, we have
\begin{equation}\label{bakryduality}
\begin{split} \| (a-\Delta)^{1/2} f \|_p &= \sup_{\substack{g \in L^q (M,\mathbb{R}) \\ \| g \|_q \le 1}} \int_M g (a-\Delta)^{1/2} f \, \mathrm{d}\mathrm{Vol} \\ &= \sup_{\substack{h \in \mathrm{C}^{\infty} (M) \\ \| (a-\Delta)^{1/2} h \|_q \le 1}} \int_M \left( (a-\Delta)^{1/2} h \right) \left( (a-\Delta)^{1/2} f \right) \, \mathrm{d} \mathrm{Vol} \\ &= \sup_{\substack{h \in \mathrm{C}^{\infty} (M) \\ \| (a-\Delta)^{1/2} h \|_q \le 1}} \int_M \left( ahf+ \langle \nabla h, \nabla f \rangle \right) \, \mathrm{d}\mathrm{Vol} \\ &\le \sup_{\substack{h \in \mathrm{C}^{\infty} (M) \\ \| (a-\Delta)^{1/2} h \|_q \le 1}} \left( a \| h \|_q \| f \|_p + \| \nabla h \|_q \| \nabla f \|_p \right) . \end{split}
\end{equation}
The second inequality in \eqref{riesztransform2} shows that here $\| \nabla h \|_q \le 3 \sqrt{6} (p^*-1)$. Observe that the Laplace transform of the function $e^{-at} (\pi t)^{-1/2}$, $t>0$ is
\[ \int_0^{\infty} e^{-at} (\pi t)^{-1/2} e^{-st} \, \mathrm{d}t = (a+s)^{-1/2}, \qquad s>0. \]
Hence $(a-\Delta)^{-1/2} = \int_0^{\infty} e^{-at} (\pi t)^{-1/2} e^{t \Delta} \, \mathrm{d}t$, and using the fact that the heat semigroup $e^{t\Delta}$ is contracting in $L^q (M)$, we deduce that for any $F \in \mathrm{C}^{\infty}(M)$,
\[ \| (a-\Delta)^{-1/2} F \|_q \le \int_0^{\infty} e^{-at} (\pi t)^{-1/2} \| e^{t\Delta} F \|_q \, \mathrm{d}t \le \| F \|_q \int_0^{\infty} e^{-at} (\pi t)^{-1/2} \, \mathrm{d}t = a^{-1/2} \| F \|_q . \]
An application of the previous formula with $F=(a-\Delta)^{1/2} h$ shows that in \eqref{bakryduality} we have $\| h \|_q \le a^{-1/2}$. Therefore $\| (a-\Delta)^{1/2} f \|_p \le a^{1/2} \| f \|_p + 3\sqrt{6} (p^*-1) \| \nabla f \|_p$, which is the first inequality in \eqref{riesztransform2}.

\subsection{Dispersion rate of the heat kernel}

We now estimate the dispersion rate
\[ D_p (P_t) = \sup_{y \in M} \left( \int_M P_t (x,y) \rho(x,y)^p \, \mathrm{d}\mathrm{Vol}(x) \right)^{1/p} \]
of the heat kernel $P_t (x,y)$. In the special case of the torus $M=\mathbb{R}^d/\mathbb{Z}^d$, the heat kernel has the explicit formula \cite[Section 4.4.2]{GF}
\[ P_t (x,y) = \frac{1}{(4 \pi t)^{d/2}} \sum_{k \in \mathbb{Z}^d} e^{-|x-y-k|^2/(4t)}, \quad t>0, \,\, x,y \in \mathbb{R}^d / \mathbb{Z}^d . \]
Since this is a function of $x-y$, we deduce that for all $1 \le p < \infty$,
\begin{equation}\label{heatkerneldispersion}
\begin{split} D_p^p (P_t) &= \int_{\mathbb{R}^d / \mathbb{Z}^d} P_t (x,0) \| x \|_{\mathbb{R}^d / \mathbb{Z}^d}^p \, \mathrm{d}\mathrm{Vol}(x) = \frac{1}{(4 \pi t)^{d/2}} \int_{\mathbb{R}^d} e^{-|x|^2/(4t)} \| x \|_{\mathbb{R}^d / \mathbb{Z}^d}^p \, \mathrm{d}x \\ &\le \frac{1}{(4 \pi t)^{d/2}} \int_{\mathbb{R}^d} e^{-|x|^2/(4t)} |x|^p \, \mathrm{d} x = \frac{1}{(4 \pi t)^{d/2}} \int_0^{\infty} e^{-R^2/(4t)} R^p \frac{2 \pi^{d/2}R^{d-1}}{\Gamma \left( \frac{d}{2} \right)} \, \mathrm{d}R \\ &= 2^p \frac{\Gamma \left( \frac{d+p}{2} \right)}{\Gamma \left( \frac{d}{2} \right)} t^{p/2} . \end{split}
\end{equation}

We now generelize \eqref{heatkerneldispersion} to compact manifolds. Lemma \ref{heatkerneldispersionlemma} in the special case $p=2$ appeared in \cite{BOR2}.
\begin{lem}\label{heatkerneldispersionlemma} Assume the Ricci curvature of $M$ is bounded below by $-(d-1)A$ with some constant $A \ge 0$. For any real $t>0$, we have
\[ D_p (P_t) \le \left\{ \begin{array}{ll} \displaystyle{(2dt)^{1/2} \left(1+C_3 t^{1/2} \right)^{1/2}} & \textrm{if } 1 \le p \le 2 \textrm{ is real}, \\ \displaystyle{2 \left( \frac{\Gamma \left( \frac{d+p}{2} \right)}{\Gamma \left( \frac{d}{2} \right)} \right)^{1/p} t^{1/2} \left( 1+C_3' e^{p/2} t^{1/2} \right)^{1/p}} & \textrm{if } p \ge 2 \textrm{ is an even integer}, \\ \displaystyle{(2d+p)^{1/2} t^{1/2} \left( 1+C_3' e^{(p+2)/2} t^{1/2} \right)^{1/p}} & \textrm{if } p \ge 2 \textrm{ is real,} \end{array} \right. \]
where
\[ \begin{split} C_3 &= \frac{2^{3/2}(d-1) \sqrt{A}}{3d} \left( d + (d-1) \sqrt{A} \, \mathrm{diam}(M) \right)^{1/2}, \\ C_3' &= 2 (d-1) \sqrt{A} \exp \left( \frac{(d-1)\sqrt{A}}{2} \, \mathrm{diam} (M) \right) . \end{split} \]
\end{lem}

\begin{proof} Fix $y \in M$. For any real $p \ge 0$, define
\[ F_p(t) = \int_M P_t(x,y) \rho (x,y)^p \, \mathrm{d}\mathrm{Vol}(x), \qquad t>0. \]
In particular, $F_0(t)=1$ and
\begin{equation}\label{Fpdiameterbound}
F_p(t) \le \mathrm{diam}(M) F_{p-1}(t).
\end{equation}
Since the heat kernel is an approximate identity, $\lim_{t \to 0^+} F_p(t)=\rho (y,y)^p=0$ for all $p>0$.

Using the fact that the heat kernel satisfies the heat equation, we deduce
\[ F_p'(t) = \int_M \frac{\partial}{\partial t} P_t(x,y) \rho(x,y)^p \, \mathrm{d}\mathrm{Vol}(x) = \int_M (\Delta_x P_t(x,y)) \rho (x,y)^p \, \mathrm{d}\mathrm{Vol}(x). \]
The global Laplacian comparison theorem \cite{WE} states that $\Delta_x \rho (x,y) \le (d-1)(\rho(x,y)^{-1}+\sqrt{A})$ in the sense of distributions. Clearly $|\nabla_x \rho(x,y)| \le 1$, consequently
\[ \begin{split} \Delta_x \rho(x,y)^p &=p(p-1)\rho(x,y)^{p-2} |\nabla_x \rho (x,y)|^2 + p\rho(x,y)^{p-1} \Delta_x \rho (x,y) \\ &\le p (d+p-2) \rho(x,y)^{p-2} + p (d-1)\sqrt{A} \rho(x,y)^{p-1} \end{split} \]
in the sense of distributions, so an application of Green's identity leads to
\begin{equation}\label{FpODE}
\begin{split} F_p'(t) &=\int_M P_t(x,y) \Delta_x \rho (x,y)^p \, \mathrm{d}\mathrm{Vol}(x) \\ &\le p(d+p-2) F_{p-2}(t) + p(d-1)\sqrt{A} F_{p-1}(t). \end{split}
\end{equation}

For the sake of readability, set $B=(d-1) \sqrt{A} \, \mathrm{diam}(M)$. The estimates \eqref{Fpdiameterbound} and \eqref{FpODE} imply $F_p'(t) \le p(d+B+p-2)F_{p-2}(t)$, from which we deduce the preliminary estimate
\[ F_p(t) \le 2^p \frac{\Gamma \left( \frac{d+B+p}{2} \right)}{\Gamma \left( \frac{d+B}{2} \right)} t^{p/2}, \qquad p \ge 0 \textrm{ even integer} \]
by induction on $p$. An application of the H\"older inequality in \eqref{FpODE} and the previous formula show that for all even integers $p \ge 2$,
\begin{equation}\label{FpODE2}
\begin{split} F_p'(t) &\le p(d+p-2) F_{p-2}(t) + p(d-1)\sqrt{A} F_p (t)^{(p-1)/p} \\ &\le p(d+p-2) F_{p-2}(t) + p (d-1) \sqrt{A} 2^{p-1} \Bigg( \frac{\Gamma \left( \frac{d+B+p}{2} \right)}{\Gamma \left( \frac{d+B}{2} \right)} \Bigg)^{(p-1)/p} t^{(p-1)/2} . \end{split}
\end{equation}

Setting $p=2$ in \eqref{FpODE2} leads to
\[ F_2'(t) \le 2d+ 4 (d-1) \sqrt{A} \left( \frac{d+B}{2} \right)^{1/2}t^{1/2}. \]
By integrating the previous formula, we deduce
\[ F_2(t) \le 2dt + \frac{2^{5/2}}{3}(d-1) \sqrt{A} (d+B)^{1/2} t^{3/2} = 2dt \left( 1+C_3 t^{1/2} \right) , \]
thus establishing the first line of the claim for $p=2$. The claim for any real $1 \le p \le 2$ follows from an application of the H\"older inequality.

Let $C_3' \ge 0$ be a constant depending only on the manifold $M$, to be chosen. Let us prove
\begin{equation}\label{Fpinduction}
F_p (t) \le 2^p \frac{\Gamma \left( \frac{d+p}{2} \right)}{\Gamma \left( \frac{d}{2} \right)} t^{p/2} \left( 1+C_3'  e^{p/2} t^{1/2} \right), \qquad p \ge 0 \textrm{ even integer}
\end{equation}
by induction on $p$, the base case being $F_0(t)=1$. Let $p \ge 2$ be an even integer, and assume \eqref{Fpinduction} holds for $p-2$. Formula \eqref{FpODE2} and the inductive hypothesis show
\[ F_p'(t) \le 2^p \frac{\Gamma \left( \frac{d+p}{2} \right)}{\Gamma \left( \frac{d}{2} \right)} \cdot \frac{p}{2} t^{(p-2)/2} \left( 1+C_3'  e^{p/2-1} t^{1/2} \right) + p (d-1) \sqrt{A} 2^{p-1} \frac{\Gamma \left( \frac{d+B+p}{2} \right)}{\Gamma \left( \frac{d+B}{2} \right)} t^{(p-1)/2}, \]
and by integrating, we deduce
\[ F_p(t) \le 2^p \frac{\Gamma \left( \frac{d+p}{2} \right)}{\Gamma \left( \frac{d}{2} \right)} t^{p/2} \left(1+C_3' \frac{p}{p+1} e^{p/2-1} t^{1/2} \right) + \frac{p}{p+1} (d-1) \sqrt{A} 2^p \frac{\Gamma \left( \frac{d+B+p}{2} \right)}{\Gamma \left( \frac{d+B}{2} \right)} t^{(p+1)/2} . \]
The induction can thus be completed provided
\[ C_3'  \frac{p}{p+1} e^{p/2-1} + \frac{p}{p+1} (d-1) \sqrt{A} \frac{\Gamma \left( \frac{d+B+p}{2} \right) \Gamma \left( \frac{d}{2} \right)}{\Gamma \left( \frac{d+B}{2} \right) \Gamma \left( \frac{d+p}{2} \right)} \le C_3' e^{p/2} , \]
or equivalently,
\begin{equation}\label{C3condition}
(d-1) \sqrt{A} \frac{ep}{(e-1)p+e} e^{-p/2} \frac{\Gamma \left( \frac{d+B+p}{2} \right) \Gamma \left( \frac{d}{2} \right)}{\Gamma \left( \frac{d+B}{2} \right) \Gamma \left( \frac{d+p}{2} \right)} \le C_3' \quad \textrm{for all even integers } p \ge 2.
\end{equation}
Stirling's approximation
\begin{equation}\label{stirling}
\sqrt{2 \pi} x^{x-1/2} e^{-x} \le \Gamma (x) \le \sqrt{2 \pi} x^{x-1/2} e^{-x+1/(12x)}
\end{equation}
shows
\[ \begin{split} \frac{\Gamma \left( \frac{d+B+p}{2} \right) \Gamma \left( \frac{d}{2} \right)}{\Gamma \left( \frac{d+B}{2} \right) \Gamma \left( \frac{d+p}{2} \right)} &\le \frac{(d+B+p)^{(d+B+p-1)/2} d^{(d-1)/2}}{(d+B)^{(d+B-1)/2} (d+p)^{(d+p-1)/2}} e^{1/6} \\ &= \left( 1+\frac{B}{d+p} \right)^{(d+p-1)/2} \left( 1+\frac{p}{d+B} \right)^{B/2} \left( \frac{d}{d+B} \right)^{(d-1)/2} e^{1/6} \\ &\le e^{B/2} e^{p/2} e^{1/6}, \end{split} \]
hence the left-hand side of \eqref{C3condition} has the upper bound
\[ (d-1) \sqrt{A} \frac{ep}{(e-1)p+e} e^{-p/2} \frac{\Gamma \left( \frac{d+B+p}{2} \right) \Gamma \left( \frac{d}{2} \right)}{\Gamma \left( \frac{d+B}{2} \right) \Gamma \left( \frac{d+p}{2} \right)} \le (d-1) \sqrt{A} e^{B/2} \frac{e^{7/6}p}{(e-1)p+e} \le 2 (d-1) \sqrt{A} e^{B/2} , \]
where we used $e^{7/6}/(e-1) \approx 1.87 <2$. In particular, \eqref{C3condition} is satisfied with the choice $C_3' = 2 (d-1) \sqrt{A} e^{B/2}$. This finishes the proof of \eqref{Fpinduction} by induction, and establishes the claim for even integers $p \ge 2$.

Now let $p \ge 2$ be real, and let $N \ge 1$ be the integer for which $2N \le p < 2N+2$. An application of H\"older's inequality and the claim with $2N+2$ yields
\[ \begin{split} \int_M P_t(x,y) \rho(x,y)^p \, \mathrm{d} \mathrm{Vol}(x) &\le \left( \int_M P_t(x,y) \rho(x,y)^{2N+2} \, \mathrm{d} \mathrm{Vol}(x) \right)^{p/(2N+2)} \\ &\le 2^p \left( \frac{\Gamma \left( \frac{d+2N+2}{2} \right)}{\Gamma \left( \frac{d}{2} \right)} \right)^{p/(2N+2)} t^{p/2} \left( 1+C_3' e^{N+1} t^{1/2} \right)^{p/(2N+2)} . \end{split} \]
The inequality between the geometric and arithmetic means shows that here
\[ \left( \frac{\Gamma \left( \frac{d+2N+2}{2} \right)}{\Gamma \left( \frac{d}{2} \right)} \right)^{p/(2N+2)} = \left( \prod_{\ell=0}^N \left( \frac{d}{2} +\ell \right) \right)^{p/(2N+2)} \le \left( \frac{d}{2} + \frac{N}{2} \right)^{p/2} \le \left( \frac{d}{2} + \frac{p}{4} \right)^{p/2} , \]
and clearly
\[ \left( 1+C_3' e^{N+1} t^{1/2} \right)^{p/(2N+2)} \le \left( 1+C_3' e^{(p+2)/2} t^{1/2} \right) . \]
This finishes the proof for all real $p \ge 2$.
\end{proof}

\section{Proofs of the main results}\label{proofsection}

\subsection{Torus with $1 \le p < \infty$}

In this section, we give the proofs of Theorems \ref{torusjacksonkerneltheorem} and \ref{torusheatkerneltheorem}. We start with a lemma based on the Hausdorff--Young inequality, which will be used to bound the dual Sobolev norm $\| \cdot \|_{\dot{H}_{-1}^p}$.
\begin{lem}\label{torushausdorffyounglemma} Let $f \in \mathrm{Lip}(\mathbb{R}^d / \mathbb{Z}^d)$. For any $2 \le p < \infty$ and $1 < q \le 2$ such that $1/p+1/q=1$, we have
\[ \Bigg( \sum_{k \in \mathbb{Z}^d} (2 \pi |k| |\widehat{f}(k)|)^p \Bigg)^{1/p} \le \| \nabla f \|_q \]
and
\[ \sup_{k \in \mathbb{Z}^d} 2 \pi |k| |\widehat{f}(k)| \le \| \nabla f \|_1 . \]
\end{lem}

\begin{proof} The Fourier transform of a vector-valued function $g: \mathbb{R}^d / \mathbb{Z}^d \to \mathbb{C}^n$ with coordinates $g=(g_1, g_2, \ldots, g_n)$ is defined coordinatewise as $\widehat{g}(k)=(\widehat{g_1}(k), \widehat{g_2}(k), \ldots, \widehat{g_n}(k))$, $k \in \mathbb{Z}^d$. One readily checks that
\[ \sum_{k \in \mathbb{Z}^d} |\widehat{g}(k)|^2 = \| g \|_2^2 \qquad \textrm{and} \qquad \sup_{k \in \mathbb{Z}^d} |\widehat{g}(k)| \le \| g \|_1 . \]
An application of the Riesz--Thorin interpolation theorem for operators between $L^p$ spaces of vector-valued functions \cite[p.\ 83]{HNVW} leads to the Hausdorff--Young inequality for vector-valued functions
\[ \Bigg( \sum_{k \in \mathbb{Z}^d} |\widehat{g}(k)|^p \Bigg)^{1/p} \le \| g \|_q . \]
The claim follows from the previous formula applied to $g=\nabla f$. Indeed, integration by parts (which holds for absolutely continuous, and in particular, for Lipschitz functions) shows that $\widehat{\nabla f}(k)=2 \pi i k \widehat{f}(k)$.
\end{proof}

We now give a detailed proof of Theorem \ref{torusheatkerneltheorem}, which will serve as a model for all our results.
\begin{proof}[Proof of Theorem \ref{torusheatkerneltheorem}] Let $1 \le p < \infty$ and $1<q \le \infty$ be such that $1/p+1/q=1$, and fix $t>0$. We work with the heat kernel $P_t(x,y)=G_t(x-y)$ on the torus \cite[Section 4.4.2]{GF}, where
\[ G_t (x) = \sum_{k \in \mathbb{Z}^d} e^{-4 \pi^2 |k|^2 t} e^{2 \pi i \langle k, x \rangle} = \frac{1}{(4 \pi t)^{d/2}} \sum_{k \in \mathbb{Z}^d} e^{-|x-k|^2/(4t)}, \quad t>0, \,\, x \in \mathbb{R}^d / \mathbb{Z}^d . \]
Note that $G_t>0$ and $\int_{\mathbb{R}^d / \mathbb{Z}^d} G_t \, \mathrm{d}\mathrm{Vol} =1$.

The dispersion estimate \eqref{heatkerneldispersion} and Lemma \ref{smoothinglemma} show that
\begin{equation}\label{torusheatkernelstep1}
W_p (\mu, \nu) \le C_1 t^{1/2} + W_p (P_t *\mu, P_t *\nu)
\end{equation}
with the constant
\[ C_1 = 2 (1+(1-c)^{1/p}) \Bigg( \frac{\Gamma \left( \frac{d+p}{2} \right)}{\Gamma \left( \frac{d}{2} \right)} \Bigg)^{1/p} . \]

Assume now $p>1$ and $c>0$. The assumptions $\mu \ge c \mathrm{Vol}$ and $\nu(B) \ge b$ for any closed ball $B$ of radius $r$ together with
\[ G_t (x) \ge \frac{e^{-\| x \|_{\mathbb{R}^d / \mathbb{Z}^d}^2/(4t)}}{(4 \pi t)^{d/2}} \]
imply that for any $x \in \mathbb{R}^d / \mathbb{Z}^d$,
\[ \frac{\mathrm{d}(P_t *\mu)}{\mathrm{d}\mathrm{Vol}}(x) = \int_{\mathbb{R}^d / \mathbb{Z}^d} G_t (x-y) \, \mathrm{d} \mu (y) \ge c  \int_{\mathbb{R}^d / \mathbb{Z}^d} G_t (x-y) \, \mathrm{d} \mathrm{Vol} (y) = c \]
and
\[ \begin{split} \frac{\mathrm{d}(P_t *\nu)}{\mathrm{d}\mathrm{Vol}}(x) &= \int_{\mathbb{R}^d / \mathbb{Z}^d} G_t (x-y) \, \mathrm{d} \nu (y) \ge \int_{B(x,r)} G_t (x-y) \, \mathrm{d}\nu (y) \\ &\ge \nu (B(x,r)) \frac{e^{-r^2 /(4t)}}{(4 \pi t)^{d/2}} \ge b \frac{e^{-r^2 /(4t)}}{(4 \pi t)^{d/2}}, \end{split} \]
where $B(x,r) = \{ y \in \mathbb{R}^d / \mathbb{Z}^d : \| x-y \|_{\mathbb{R}^d / \mathbb{Z}^d} \le r \}$ is the closed ball centered at $x$. In particular, $P_t *\mu \ge c \mathrm{Vol}$ and $P_t *\nu \ge \delta \mathrm{Vol}$ with
\[ \delta = \min \left\{ b \frac{e^{-r^2 /(4t)}}{(4 \pi t)^{d/2}}, c \right\} , \]
and Lemma \ref{ledouxlemma} shows that
\begin{equation}\label{torusheatkernelstep2}
W_p (P_t *\mu, P_t *\nu) \le \frac{p(c^{1/p} - \delta^{1/p})}{c-\delta} \| P_t *\mu - P_t *\nu \|_{\dot{H}_{-1}^p} .
\end{equation}

To estimate the dual Sobolev norm in the previous formula, let $f \in \mathrm{Lip}(\mathbb{R}^d / \mathbb{Z}^d)$ be arbitrary, and let $p_0=\max \{ p, 2 \}$ and $q_0=\min \{ q, 2 \}$. The Fourier transforms
\[ \widehat{\frac{\mathrm{d}(P_t *\mu)}{\mathrm{d}\mathrm{Vol}}}(k) = \widehat{G_t}(k) \widehat{\mu}(k) \quad \textrm{and} \quad \widehat{\frac{\mathrm{d}(P_t *\nu)}{\mathrm{d}\mathrm{Vol}}}(k) = \widehat{G_t}(k) \widehat{\nu}(k), \quad k \in \mathbb{Z}^d, \]
and an application of the Parseval formula, H\"older's inequality and Lemma \ref{torushausdorffyounglemma} lead to
\[ \begin{split} \left| \int_{\mathbb{R}^d / \mathbb{Z}^d} f \, \mathrm{d} (P_t *\mu - P_t *\nu) \right| &= \left| \sum_{k \in \mathbb{Z}^d} \overline{\widehat{f}(k)} \widehat{G_t}(k) (\widehat{\mu}(k) - \widehat{\nu}(k)) \right| \\ &\le \Bigg( \sum_{k \in \mathbb{Z}^d} (2 \pi |k| |\widehat{f}(k)|)^{p_0} \Bigg)^{1/p_0} \Bigg( \sum_{\substack{k \in \mathbb{Z}^d \\ k \neq 0}} |\widehat{G_t}(k)|^{q_0} \frac{|\widehat{\mu}(k) - \widehat{\nu}(k)|^{q_0}}{(2 \pi |k|)^{q_0}} \Bigg)^{1/q_0} \\ &\le \| \nabla f \|_{q_0} \Bigg( \sum_{\substack{k \in \mathbb{Z}^d \\ k \neq 0}} |\widehat{G_t}(k)|^{q_0} \frac{|\widehat{\mu}(k) - \widehat{\nu}(k)|^{q_0}}{(2 \pi |k|)^{q_0}} \Bigg)^{1/q_0}. \end{split} \]
By the definition \eqref{dualsobolevnorm} of the dual Sobolev norm, this implies
\[ \| P_t *\mu - P_t *\nu \|_{\dot{H}_{-1}^p} \le \| P_t *\mu - P_t *\nu \|_{\dot{H}_{-1}^{p_0}} \le \Bigg( \sum_{\substack{k \in \mathbb{Z}^d \\ k \neq 0}} |\widehat{G_t}(k)|^{q_0} \frac{|\widehat{\mu}(k) - \widehat{\nu}(k)|^{q_0}}{(2 \pi |k|)^{q_0}} \Bigg)^{1/{q_0}}, \]
where $\widehat{G_t}(k) = e^{-4 \pi^2 |k|^2 t}$. The claim for $1<p<\infty$ immediately follows from \eqref{torusheatkernelstep1}, \eqref{torusheatkernelstep2} and the previous formula.

The proof for $p=1$ and $c \ge 0$ is very similar. Let $f \in \mathrm{Lip}(\mathbb{R}^d / \mathbb{Z}^d)$ be arbitrary. Using $|\nabla f| \le \| f \|_{\mathrm{Lip}}$ a.e., we similarly deduce
\[ \begin{split} \left| \int_{\mathbb{R}^d / \mathbb{Z}^d} f \, \mathrm{d} (P_t *\mu - P_t *\nu) \right| &\le \| \nabla f \|_2 \Bigg( \sum_{\substack{k \in \mathbb{Z}^d \\ k \neq 0}} |\widehat{G_t}(k)|^2 \frac{|\widehat{\mu}(k) - \widehat{\nu}(k)|^2}{(2 \pi |k|)^2} \Bigg)^{1/2} \\ &\le \| f \|_{\mathrm{Lip}} \Bigg( \sum_{\substack{k \in \mathbb{Z}^d \\ k \neq 0}} |\widehat{G_t}(k)|^2 \frac{|\widehat{\mu}(k) - \widehat{\nu}(k)|^2}{(2 \pi |k|)^2} \Bigg)^{1/2}. \end{split} \]
The Kantorovich duality formula \eqref{kantorovichduality} thus implies
\[ W_1 (P_t *\mu, P_t *\nu) \le \Bigg( \sum_{\substack{k \in \mathbb{Z}^d \\ k \neq 0}} |\widehat{G_t}(k)|^2 \frac{|\widehat{\mu}(k) - \widehat{\nu}(k)|^2}{(2 \pi |k|)^2} \Bigg)^{1/2} . \]
The claim for $p=1$ immediately follows from \eqref{torusheatkernelstep1} and the previous formula.

Finally, we give a simple upper bound for the constant $C_1$ for the convenience of the reader. Since $\log \Gamma$ is convex, Jensen's inequality gives $\Gamma(x+s) \le \Gamma (x)^{1-s} \Gamma(x+1)^s = x^s \Gamma(x)$ for all $x>0$ and $0 \le s \le 1$. Letting $N=\lfloor p/2 \rfloor$, the previous inequality and an application of the inequality between the weighted geometric and arithmetic means lead to
\[ \begin{split} \frac{\Gamma \left( \frac{d+p}{2} \right)}{\Gamma \left( \frac{d}{2} \right)} &= \left( \prod_{\ell=0}^{N-1} \left( \frac{d}{2} +\ell \right) \right) \frac{\Gamma \left( \frac{d+p}{2} \right)}{\Gamma \left( \frac{d}{2}+N \right)} \le \left( \prod_{\ell=0}^{N-1} \left( \frac{d}{2} +\ell \right) \right) \left( \frac{d}{2} +N \right)^{p/2-N} \\ &\le \left( \frac{2}{p} \sum_{\ell=0}^{N-1} \left( \frac{d}{2} + \ell \right) + \left( 1- \frac{2N}{p} \right) \left( \frac{d}{2} + N \right) \right)^{p/2} = \left( \frac{d}{2} + \frac{(p-1)N-N^2}{p} \right)^{p/2} . \end{split} \]
Here
\[ \frac{(p-1)N - N^2}{p} \le \frac{(p-1)^2}{4p} \le \frac{p}{4}, \]
which finally yields the desired upper bound
\[ C_1 \le 2 (1+(1-c)^{1/p}) \left( \frac{d}{2} + \frac{p}{4} \right)^{1/2}. \]
\end{proof}

The proof of Theorem \ref{torusjacksonkerneltheorem} follows the same approach as that of Theorem \ref{torusheatkerneltheorem}, but uses a different kernel in the smoothing procedure.
\begin{proof}[Proof of Theorem \ref{torusjacksonkerneltheorem}] Let $1 \le p < \infty$ and $1<q \le \infty$ be such that $1/p+1/q=1$, and fix an integer $H \ge 0$. We start by constructing a kernel $K(x,y)$ on the $d$-dimensional torus. Our starting point is the Dirichlet kernel on the 1-dimensional torus
\[ \sum_{k=-H}^H e^{2 \pi i k x} = \frac{\sin (\pi (2H+1)x)}{\sin (\pi x)} . \]
Given an integer $N \ge 2$, consider the $N$th power
\begin{equation}\label{Npowerdirichletkernel}
\bigg( \sum_{k=-H}^H e^{2 \pi i k x} \bigg)^N = \sum_{k=-NH}^{NH} a_{N,k} e^{2 \pi i k x}
\end{equation}
with some integer coefficients $a_{N,k}$. Let
\[ F_{N,H}(x)=\frac{1}{c_N} \bigg( \sum_{k=-H}^H e^{2 \pi i k x} \bigg)^{2N} = \frac{(\sin(\pi (2H+1)x))^{2N}}{c_N (\sin (\pi x))^{2N}} \]
with
\[ c_N = \int_0^1 \bigg( \sum_{k=-H}^H e^{2 \pi i k x} \bigg)^{2N} \, \mathrm{d}x = \sum_{k=-NH}^{NH} a_{N,k}^2 , \]
where the last equality follows from the Parseval formula. In particular, $F_{N,H}$ is nonnegative and normalized.

The special case $N=2$ leads to the Jackson kernel, that is, the normalized square of the Fej\'er kernel. In formula \eqref{Npowerdirichletkernel} we then have $a_{2,k}=(2H+1-|k|)$, thus
\[ c_2 = \sum_{k=-2H}^{2H} (2H+1-|k|)^2 = \frac{(2H+1)(8H^2+8H+3)}{3} . \]
For $N>2$, the Fourier series expansion \eqref{Npowerdirichletkernel} at $x=0$ and an application of the Cauchy--Schwarz inequality show that
\[ (2H+1)^{2N} = \bigg( \sum_{k=-NH}^{NH} a_{N,k} \bigg)^2 \le (2NH+1) \sum_{k=-NH}^{NH} a_{N,k}^2 = (2NH+1) c_N . \]
The previous two formulas immediately yield the lower bound
\begin{equation}\label{cNbound}
c_N \ge \left\{ \begin{array}{ll} (2H+1)(8H^2+8H+3) / 3 & \textrm{if } N=2, \\ (2H+1)^{2N-1} / N & \textrm{if } N>2. \end{array} \right.
\end{equation}
The coefficients in the Fourier series expansion $F_{N,H}(x)=\sum_{k=-2NH}^{2NH} \widehat{F_{N,H}}(k) e^{2 \pi i k x}$ satisfy $|\widehat{F_{N,H}}(k)| \le 1$ because $F_{N,H}$ is nonnegative and normalized. Finally, let us extend $F_{N,H}$ to the $d$-dimensional torus as $G_{N,H}(x)=F_{N,H}(x_1) F_{N,H}(x_2) \cdots F_{N,H}(x_d)$, $x=(x_1, x_2, \ldots, x_d) \in \mathbb{R}^d / \mathbb{Z}^d$. We thus have
\[ G_{N,H} \ge 0, \qquad \int_{\mathbb{R}^d / \mathbb{Z}^d} G_{N,H} \, \mathrm{d}\mathrm{Vol}=1, \qquad |\widehat{G_{N,H}}(k)| \le \mathds{1}_{[-2NH,2NH]^d}(k). \]
We will use the kernel $K(x,y)=G_{N,H}(x-y)$, $x,y \in \mathbb{R}^d / \mathbb{Z}^d$ with a suitably chosen $N$.

Let us now estimate the dispersion rate $D_p(K)$. Since $K(x,y)$ is a function of $x-y$, we have
\[ D_p (K)= \Bigg( \int_{\mathbb{R}^d / \mathbb{Z}^d} G_{N,H} (x) \| x \|_{\mathbb{R}^d / \mathbb{Z}^d}^p \, \mathrm{d}\mathrm{Vol}(x) \Bigg)^{1/p} . \]
Identifying $\mathbb{R}^d / \mathbb{Z}^d$ with the unit cube $[-1/2, 1/2]^d$, we have $\| x \|_{\mathbb{R}^d / \mathbb{Z}^d} = (x_1^2+\cdots + x_d^2)^{1/2}$. If $1 \le p \le 2$, then an application of the H\"older inequality leads to
\[ \begin{split} D_p (K) &\le \Bigg( \int_{\mathbb{R}^d / \mathbb{Z}^d} G_{N,H} (x) \| x \|_{\mathbb{R}^d / \mathbb{Z}^d}^2 \, \mathrm{d}\mathrm{Vol}(x) \Bigg)^{1/2} \\ &= \Bigg( \int_{[-1/2, 1/2]^d} F_{N,H} (x_1) \cdots F_{N,H} (x_d) (x_1^2 + \cdots + x_d^2) \, \mathrm{d}x_1 \cdots \mathrm{d}x_d \Bigg)^{1/2} \\ &= \Bigg( d \int_{-1/2}^{1/2} F_{N,H}(x) x^2 \, \mathrm{d}x \Bigg)^{1/2} \\ &= d^{1/2} \Bigg( \frac{2}{c_N} \int_0^{1/2} \frac{(\sin (\pi (2H+1)x))^{2N}}{(\sin (\pi x))^{2N}} x^2 \, \mathrm{d}x \Bigg)^{1/2} . \end{split} \]
If $p>2$, then the inequality
\[ \| x \|_{\mathbb{R}^d / \mathbb{Z}^d}^p = (x_1^2 + \cdots + x_d^2)^{p/2} \le d^{p/2-1} (|x_1|^p + \cdots + |x_d|^p) \]
leads to
\[ \begin{split} D_p(K) &\le \Bigg( d^{p/2-1} \int_{[-1/2, 1/2]^d} F_{N,H}(x_1) \cdots F_{N,H}(x_d) (|x_1|^p + \cdots + |x_d|^p) \, \mathrm{d}x_1 \cdots \mathrm{d}x_d \Bigg)^{1/p} \\ &= \Bigg( d^{p/2} \int_{-1/2}^{1/2} F_{N,H}(x) |x|^p \, \mathrm{d}x \Bigg)^{1/p} \\ &= d^{1/2} \Bigg( \frac{2}{c_N} \int_{0}^{1/2} \frac{(\sin(\pi (2H+1)x))^{2N}}{(\sin (\pi x))^{2N}} x^p \, \mathrm{d}x \Bigg)^{1/p} . \end{split} \]

Observe that
\[ \int_{0}^{1/(2(2H+1))} \frac{(\sin(\pi (2H+1)x))^{2N}}{(\sin (\pi x))^{2N}} x^p \, \mathrm{d}x \le \int_{0}^{1/(2(2H+1))} (2H+1)^{2N} x^p \, \mathrm{d}x = \frac{(2H+1)^{2N-p-1}}{2^{p+1} (p+1)}, \]
and using $\sin (\pi x) \ge 2x$ on $x\in [0,1/2]$,
\[ \int_{1/(2(2H+1))}^{1/2} \frac{(\sin(\pi (2H+1)x))^{2N}}{(\sin (\pi x))^{2N}} x^p \, \mathrm{d}x \le \int_{1/(2(2H+1))}^{1/2} \frac{1}{(2x)^{2N}} x^p \, \mathrm{d}x \le \frac{(2H+1)^{2N-p-1}}{2^{p+1} (2N-p-1)} \]
provided $2N-p-1>0$. Adding the previous two formulas yields
\[ \int_{0}^{1/2} \frac{(\sin(\pi (2H+1)x))^{2N}}{(\sin (\pi x))^{2N}} x^p \, \mathrm{d}x \le \frac{N(2H+1)^{2N-p-1}}{2^p (p+1)(2N-p-1)} . \]

Let us choose $N=\lceil (p+2)/2 \rceil$. If $1 \le p \le 2$, then $N=2$, and using the previous estimates and \eqref{cNbound} we deduce
\[ D_p (K) \le d^{1/2} \left( \frac{2H+1}{3 c_2} \right)^{1/2} \le \frac{d^{1/2}}{(8H^2 + 8H+3)^{1/2}} . \]
If $p>2$, then similarly
\[ D_p (K) \le d^{1/2} \left( \frac{2N(2H+1)^{2N-p-1}}{2^p (p+1)(2N-p-1) c_N} \right)^{1/p} \le \frac{d^{1/2}}{2(2H+1)} \left( \frac{2N^2}{(p+1)(2N-p-1)} \right)^{1/p} . \]
One readily checks that the function $\frac{2x^2}{(p+1)(2x-p-1)}$ is decreasing on the interval $(p+2)/2 \le x < (p+4)/2$, hence
\[ \frac{2N^2}{(p+1)(2N-p-1)} \le \frac{(p+2)^2}{2(p+1)} . \]
Since $((p+2)^2 / (2(p+1)))^{1/p}$ is decreasing in the variable $p \ge 2$, we finally deduce
\[ D_p (K) \le \frac{d^{1/2}}{2 (2H+1)} \cdot \frac{4}{6^{1/2}} . \]
An application of Lemma \ref{smoothinglemma} thus leads to
\begin{equation}\label{torusjacksonkernelstep1}
W_p (\mu, \nu) \le W_p (K*\mu, K*\nu) + \left\{ \begin{array}{ll} (1+(1-c)^{1/p}) \frac{d^{1/2}}{(8H^2 + 8H+3)^{1/2}} & \textrm{if } 1 \le p \le 2, \\ (1+(1-c)^{1/p}) \frac{2 d^{1/2}}{6^{1/2} (2H+1)} & \textrm{if } p>2. \end{array} \right.
\end{equation}

The rest of the proof is entirely analogous to the proof of Theorem \ref{torusheatkerneltheorem}. Following the same steps with the choice $\delta=0$, and using $|\widehat{G_{N,H}}(k)| \le \mathds{1}_{\{ [-2NH, 2NH]^d \}}(k)$, we similarly deduce the following. If $p=1$, then
\[ W_1 (\mu, \nu) \le \frac{(2-c)d^{1/2}}{(8H^2 + 8H+3)^{1/2}} + \Bigg( \sum_{\substack{k \in [-4H,4H]^d \\ k \neq 0}} \frac{|\widehat{\mu}(k) - \widehat{\nu}(k)|^2}{(2 \pi |k|)^2} \Bigg)^{1/2} . \]
If $1<p \le 2$, then
\[ W_p (\mu, \nu) \le \frac{(1+(1-c)^{1/p})d^{1/2}}{(8H^2+8H+3)^{1/2}} + \frac{p}{c^{1/q}} \Bigg( \sum_{\substack{k \in [-4H,4H]^d \\ k \neq 0}} \frac{|\widehat{\mu}(k) - \widehat{\nu}(k)|^2}{(2 \pi |k|)^2} \Bigg)^{1/2} . \]
If $2<p<\infty$, then
\[ W_p(\mu, \nu) \le \frac{2(1+(1-c)^{1/p}) d^{1/2}}{6^{1/2}(2H+1)} + \frac{p}{c^{1/q}} \Bigg( \sum_{\substack{k \in [-2NH,2NH]^d \\ k \neq 0}} \frac{|\widehat{\mu}(k) - \widehat{\nu}(k)|^q}{(2 \pi |k|)^q} \Bigg)^{1/q} . \]

Letting $H' \ge 0$ be the integer for which $4H' \le H \le 4H'+3$ resp.\ $2NH' \le H \le 2NH'+2N-1$, we have $8(H')^2 +8H'+3 \ge (H-1)^2 /2$ resp.\ $2H' +1 \ge \frac{2}{p+4} ( H-\frac{p+2}{2})$. We can thus reformulate the previous three estimates as follows. If $p=1$, then for any integer $H>1$,
\[ W_1(\mu, \nu) \le \frac{2^{1/2}(2-c)d^{1/2}}{H-1} + \Bigg( \sum_{\substack{k \in [-H,H]^d \\ k \neq 0}} \frac{|\widehat{\mu}(k) - \widehat{\nu}(k)|^2}{(2 \pi |k|)^2} \Bigg)^{1/2} . \]
If $1<p \le 2$, then for any integer $H>1$,
\[ W_p (\mu, \nu) \le \frac{2^{1/2}(1+(1-c)^{1/p})d^{1/2}}{H-1} + \frac{p}{c^{1/q}} \Bigg( \sum_{\substack{k \in [-H,H]^d \\ k \neq 0}} \frac{|\widehat{\mu}(k) - \widehat{\nu}(k)|^2}{(2 \pi |k|)^2} \Bigg)^{1/2} . \]
If $2<p<\infty$, then for any integer $H>(p+2)/2$,
\[ W_p(\mu, \nu) \le \frac{(1+(1-c)^{1/p})(p+4) d^{1/2}}{6^{1/2}(H-(p+2)/2)} + \frac{p}{c^{1/q}} \Bigg( \sum_{\substack{k \in [-H,H]^d \\ k \neq 0}} \frac{|\widehat{\mu}(k) - \widehat{\nu}(k)|^q}{(2 \pi |k|)^q} \Bigg)^{1/q} . \]
This finishes the proof in the full range $1 \le p < \infty$.
\end{proof}

\subsection{Torus with $p=\infty$}

In order to prove the smoothing inequality for $W_{\infty}$ stated in Theorem \ref{torusinfinitytheorem}, we first need to construct a suitable bump function $g$ on $\mathbb{R}^d$ whose Fourier transform
\[ \widehat{g}(y) = \int_{\mathbb{R}^d} g(x) e^{-2 \pi i \langle x,y \rangle} \, \mathrm{d}x \]
satisfies an explicit upper estimate. The construction given in Lemma \ref{bumpfunctionlemma} is a higher-dimensional generalization of the Fabius function \cite{FA}. Throughout, $B(a,r)=\{ x \in \mathbb{R}^d : |x-a| \le r \}$ denotes the closed Euclidean ball centered at $a \in \mathbb{R}^d$ with radius $r$. 
\begin{lem}\label{bumpfunctionlemma} There exists a smooth function $g: \mathbb{R}^d \to \mathbb{R}$ with the following properties.
\begin{enumerate}[(i)]
\item $g(x) \ge 0$, $\mathrm{supp} \, g = B(0,1)$ and $\int_{\mathbb{R}^d} g(x) \, \mathrm{d}x=1$.

\item $g(0)=\Gamma \left( \frac{d+2}{2} \right) 2^d /\pi^{d/2}$ and $g(x) \ge (2/27) g(0)$ if $|x|\le 1/5$.

\item For all $y \in \mathbb{R}^d$ such that $|y| \ge (d+3)/\pi$, we have
\[ |\widehat{g}(y)| \le \exp \left( - \frac{d+1}{4 \log 2} \left( \log \frac{\pi |y|}{d+3} \right) \left( \log \frac{e^2 \pi |y|}{d+2} \right) \right) . \]

\item For all $y \in \mathbb{R}^d$, we have
\[ \widehat{g}(y) \ge 1- \frac{2 \pi^2 d}{3(d+2)} |y|^2 . \]
\end{enumerate}
\end{lem}

\begin{proof} Let $X_1, X_2, \ldots$ be independent random variables, each uniformly distributed in the unit ball $B(0,1)$, and let $Y=\sum_{n=1}^{\infty} X_n/2^n$. We have
\[ \mathbb{E} (|X_n|^2) = \frac{\Gamma \left( \frac{d+2}{2} \right)}{\pi^{d/2}} \int_{B(0,1)} |x|^2 \, \mathrm{d}x = \frac{\Gamma \left( \frac{d+2}{2} \right)}{\pi^{d/2}} \int_0^1 r^2 \frac{2 \pi^{d/2} r^{d-1}}{\Gamma \left( \frac{d}{2} \right)} \, \mathrm{d}r = \frac{d}{d+2}, \]
thus by independence,
\begin{equation}\label{Yvariance}
\mathbb{E} (|Y|^2) = \sum_{n=1}^{\infty} \frac{\mathbb{E} (|X_n|^2)}{2^{2n}} = \frac{d}{3(d+2)} .
\end{equation}

Let
\[ g_n (x) = \frac{\Gamma \left( \frac{d+2}{2} \right) 2^{dn}}{\pi^{d/2}} \mathds{1}_{\{ |x| \le 2^{-n} \}}, \quad x \in \mathbb{R}^d \]
be the density function of $X_n/2^n$. We claim that the sequence of convolutions $g_1 * g_2 * \cdots * g_n$ converges uniformly on $\mathbb{R}^d$. One readily checks that $g_1*g_2$ is Lipschitz by noticing that $(g_1*g_2)(x)$ is a constant times the volume of the intersection $B(0,1/4) \cap B(x,1/2)$. If $f$ is Lipschitz on $\mathbb{R}^d$, then $f*g_n$ is also Lipschitz with the same Lipschitz constant. In particular, $g_1 * g_2 * \cdots *g_n$ is Lipschitz with the same Lipschitz constant $L$ as that of $g_1 * g_2$. Hence
\[ \begin{split}| (g_1 * \cdots * g_n * g_{n+1})(x) &- (g_1 * \cdots * g_n)(x)| \\ &= \left| \int_{\mathbb{R}^d} \left( (g_1 * \cdots * g_n)(x-y) - (g_1 * \cdots * g_n)(x) \right) g_{n+1}(y) \, \mathrm{d}y \right| \\ &\le \int_{\mathbb{R}^d} L |y| g_{n+1}(y) \, \mathrm{d}y \le \frac{L}{2^{n+1}} , \end{split} \]
and consequently
\[ |(g_1 * \cdots * g_n)(x) - (g_1 * \cdots * g_{n+m})(x)| \le \frac{L}{2^{n+1}} + \cdots + \frac{L}{2^{n+m}} \le \frac{L}{2^n} . \]
The sequence $g_1 * g_2 * \cdots *g_n$ is thus Cauchy in the supremum norm, therefore it converges uniformly to some function $g$, which we claim satisfies the lemma. Clearly, $g$ satisfies the properties in (i) by construction. In particular, $Y$ is absolutely continuous with density function $g$.

Let us now prove (ii). Since $Y=X_1/2+Y'/2$, where $Y'=\sum_{n=1}^{\infty} X_{n+1}/2^n$ has the same distribution as $Y$ and is independent of $X_1$, the function $g$ is the convolution of the functions $g_1(x)$ and $2^d g(2x)$. The function $g$ thus satisfies the functional equation
\[ g(x) = \frac{\Gamma \left( \frac{d+2}{2} \right) 2^d}{\pi^{d/2}} \int_{B(2x,1)} g(y) \, \mathrm{d}y , \quad x \in \mathbb{R}^d .  \]
In particular,
\[ g(0)= \frac{\Gamma \left( \frac{d+2}{2} \right) 2^d}{\pi^{d/2}} . \]
Now let $x \in \mathbb{R}^d$ be such that $|x|<1/2$. Then $B(0,1-2|x|) \subseteq B(2x,1)$, hence an application of the Chebyshev inequality and formula \eqref{Yvariance} leads to
\[ \begin{split} g(x) &\ge \frac{\Gamma \left( \frac{d+2}{2} \right) 2^d}{\pi^{d/2}} \int_{B(0,1-2|x|)} g(y) \, \mathrm{d}y = \frac{\Gamma \left( \frac{d+2}{2} \right) 2^d}{\pi^{d/2}} \Pr (|Y| \le 1-2|x|) \\ &\ge \frac{\Gamma \left( \frac{d+2}{2} \right) 2^d}{\pi^{d/2}} \left( 1-\frac{\mathbb{E}(|Y|^2)}{(1-2|x|)^2} \right) = \frac{\Gamma \left( \frac{d+2}{2} \right) 2^d}{\pi^{d/2}} \left( 1-\frac{d}{3(d+2)(1-2|x|)^2} \right) . \end{split} \]
If $|x| \le 1/5$, then
\[ 1-\frac{d}{3(d+2)(1-2|x|)^2} \ge 1- \frac{25d}{27(d+3)} \ge \frac{2}{27}, \]
consequently $g(x) \ge (2/27) g(0)$.

Let us estimate the Fourier transform $\widehat{g}(y) = \prod_{n=1}^{\infty} \widehat{g_n}(y)$. Using the formula for the Fourier transform of the indicator function of the unit ball, we deduce
\[ \widehat{g_n}(y)= \frac{\Gamma \left( \frac{d+2}{2} \right)}{\pi^{d/2}} \int_{B(0,1)} e^{-2 \pi i \langle x, 2^{-n}y \rangle} \, \mathrm{d}x = \frac{\Gamma \left( \frac{d+2}{2} \right)}{\left( \pi 2^{-n} |y| \right)^{d/2}} J_{d/2} (\pi 2^{-n+1} |y|) , \]
where $J_{\alpha}(x)$ is the Bessel function of the first kind. An explicit upper bound for the Bessel functions \cite[Theorem 2]{KR1} states that for any $\alpha>-1/2$ and any
\[ x> \frac{((2 \alpha+1)(2 \alpha +3) + (2 \alpha+1)^{2/3}(2 \alpha+3)^{2/3})^{1/2}}{2}, \]
we have
\[ |J_{\alpha}(x)| \le \left( \frac{2}{\pi x} \right)^{1/2} \left( \frac{1-\frac{(2 \alpha+1)(2 \alpha+5)}{4x^2}}{\left( 1-\frac{(2 \alpha+1)(2 \alpha+3)}{4x^2} \right)^{3/2} - \frac{(2 \alpha+1)(2\alpha+3)}{8 x^3}} \right)^{1/2} . \]
In particular, for any $d \ge 2$ and $x \ge d+3$, letting $u=(d+1)(d+3)/x^2 \in [0,1]$, we have $(1-u/4)^{3/2} \ge 1-3u/8$ and $u/(8x) \le u/40$, consequently
\[ |J_{d/2}(x)| \le \left( \frac{2}{\pi x} \right)^{1/2} \left( \frac{1-u/4}{1-3u/8 - u/(8x)} \right)^{1/2} \le \left( \frac{2}{\pi x} \right)^{1/2} \left( 1+u/4 \right)^{1/2} . \]
The previous estimate holds also for $d=1$, as in this case
\begin{equation}\label{J1/2estimate}
|J_{1/2}(x)| = \left( \frac{2}{\pi x} \right)^{1/2} |\sin x| \le \left( \frac{2}{\pi x} \right)^{1/2} .
\end{equation}
Thus
\[ |J_{d/2}(x)| \le \left( \frac{2}{\pi x} \right)^{1/2} \left( 1+\frac{(d+3)^2}{4x^2} \right)^{1/2} \qquad \textrm{for all } d \ge 1 \textrm{ and } x \ge d+3 , \]
and we obtain
\[ |\widehat{g_n}(y)| \le \left\{ \begin{array}{ll} \displaystyle{\frac{\Gamma \left( \frac{d+2}{2} \right)}{\pi^{1/2} \left( \pi 2^{-n} |y| \right)^{(d+1)/2}} \left(1+\frac{(d+3)^2}{16 (\pi 2^{-n}|y|)^2} \right)^{1/2}} & \textrm{if } \pi 2^{-n+1} |y| \ge d+3, \\ 1 & \textrm{if } \pi 2^{-n+1} |y| < d+3 . \end{array} \right. \]
Now let $y \in \mathbb{R}^d$ be such that $|y|\ge (d+3)/\pi$, and let $N$ denote the positive integer for which $2^{N-1} \le \pi |y|/(d+3) < 2^N$. The previous estimate yields
\[ |\widehat{g}(y)| = \prod_{n=1}^{\infty} |\widehat{g_n}(y)| \le \prod_{n=1}^N \frac{\Gamma \left( \frac{d+2}{2} \right)}{\pi^{1/2} \left( \pi 2^{-n} |y| \right)^{(d+1)/2}} \left(1+\frac{(d+3)^2}{16 (\pi 2^{-n}|y|)^2} \right)^{1/2} . \]
Here
\[ \prod_{n=1}^N \left(1+\frac{(d+3)^2}{16 (\pi 2^{-n}|y|)^2} \right)^{1/2} \le \prod_{n=1}^N \exp \left( \frac{(d+3)^2}{32 (\pi 2^{-n}|y|)^2} \right) = \exp \left( \frac{(d+3)^2}{32 \pi^2 |y|^2} \cdot \frac{4^{N+1}-4}{3} \right) \le e^{1/6} , \]
and by Stirling's approximation \eqref{stirling},
\[ \begin{split} \prod_{n=1}^N \frac{\Gamma \left( \frac{d+2}{2} \right)}{\pi^{1/2} \left( \pi 2^{-n} |y| \right)^{(d+1)/2}} &= \left( \frac{\Gamma \left( \frac{d+2}{2} \right) 2^{(d+1)(N+1)/4}}{\pi^{1/2} (\pi |y|)^{(d+1)/2}} \right)^N \\ &\le \left( \frac{(2 \pi)^{1/2} \left( \frac{d+2}{2} \right)^{(d+1)/2} e^{-(d+2)/2 + 1/18}\left( \frac{4 \pi |y|}{d+3} \right)^{(d+1)/4}}{\pi^{1/2} (\pi |y|)^{(d+1)/2}} \right)^N \\ &= \left( 2^{1/2} e^{-4/9} \left( \frac{d+2}{d+3} \right)^{(d+1)/4} \left( \frac{d+2}{e^2 \pi |y|} \right)^{(d+1)/4} \right)^N . \end{split} \]
The previous two formulas together with
\[ e^{1/6} \left( 2^{1/2} e^{-4/9} \left( \frac{d+2}{d+3} \right)^{(d+1)/4} \right)^N \le e^{1/6} 2^{1/2} e^{-4/9} e^{-(d+1)/(4(d+3))} < 1 \]
imply
\[ |\widehat{g}(y)| \le \left( \frac{d+2}{e^2 \pi |y|} \right) ^{N(d+1)/4} = \exp \left( - \frac{d+1}{4} N \log \frac{e^2 \pi |y|}{d+2} \right) . \]
Here $N> (\log \frac{\pi |y|}{d+3} )/\log 2$, which finishes the proof of (iii). As $|\widehat{g}(y)|$ decays faster than any polynomial as $|y| \to \infty$, the function $g$ is smooth.

Finally, fix $y \in \mathbb{R}^d$, and let us prove (iv). The function
\[ F(t)=\widehat{g}(ty) = \int_{B(0,1)} g(x) e^{-2 \pi i t \langle x,y \rangle} \, \mathrm{d}x, \qquad t \in [0,1] \]
has first derivative at $0$
\[ F'(0) = \int_{B(0,1)} g(x) (- 2 \pi i) \langle x, y \rangle \, \mathrm{d}x = - 2 \pi i \mathbb{E} \langle Y, y \rangle = 0 \]
since the distribution of $Y$ is invariant under rotations, that is, $g(x)$ is a radial function. An application of the Cauchy--Schwarz inequality and the variance formula \eqref{Yvariance} show that for all $t \in [0,1]$,
\[ \begin{split} |F''(t)| &= \left| \int_{B(0,1)} g(x) 4 \pi^2 \langle x,y \rangle^2 e^{-2 \pi i t \langle x, y \rangle} \, \mathrm{d}x \right| \le 4 \pi^2 |y|^2 \int_{B(0,1)} g(x) |x|^2 \, \mathrm{d}x \\ &= 4 \pi^2 |y|^2 \mathbb{E} (|Y|^2) = \frac{4 \pi^2 d}{3(d+2)} |y|^2 . \end{split} \]
A Taylor expansion of degree 2 thus leads to
\[ |\widehat{g}(y)-1| = |F(1) - F(0) - F'(0)| \le \sup_{t \in [0,1]} \frac{|F''(t)|}{2} \le \frac{2 \pi^2 d}{3(d+2)} |y|^2 , \]
and (iv) follows.
\end{proof}

\begin{remark}\label{1dimremark} In dimension $d=1$, we have the slightly better estimate
\[ |\widehat{g}(y)| \le \exp \left( - \frac{1}{2 \log 2} \left( \log (\pi |y|) \right) \left( \log (2 \pi |y|) \right) \right) \]
for all $y \in \mathbb{R}$ such that $|y| \ge 1/\pi$. Indeed, estimate \eqref{J1/2estimate} for the explicit Bessel function $J_{1/2}(x)$ yields the simpler upper bound $|\widehat{g_n}(y)| \le 2^{n-1}/(\pi |y|)$, and the rest of the argument is the same.
\end{remark}

\begin{proof}[Proof of Theorem \ref{torusinfinitytheorem}] Let $T \ge 5r$ be arbitrary, and let $g$ be the smooth function from Lemma \ref{bumpfunctionlemma}. The function
\[ G(x) = \frac{1}{T^d} \sum_{k \in \mathbb{Z}^d} g \left( \frac{x+k}{T} \right) , \quad x \in \mathbb{R}^d / \mathbb{Z}^d \]
is thus nonnegative and normalized, and $\mathrm{supp} \, G = B(0,T)$. Property (ii) in Lemma \ref{bumpfunctionlemma} ensures that
\begin{equation}\label{Glowerbound}
G(x) \ge \frac{\Gamma \left( \frac{d+2}{2} \right) 2^{d+1}}{27 \pi^{d/2} T^d} \quad \textrm{if } \| x \|_{\mathbb{R}^d / \mathbb{Z}^d} \le \frac{T}{5} .
\end{equation}
The Poisson summation formula leads to the Fourier series expansion
\[ G(x) = \sum_{k \in \mathbb{Z}^d} \widehat{g}(T k) e^{2 \pi i \langle k,x \rangle} , \]
hence by property (iii) in Lemma \ref{bumpfunctionlemma},
\begin{equation}\label{Ghatupperbound}
|\widehat{G}(k)| = |\widehat{g}(Tk)| \le  \left\{ \begin{array}{ll} 1 & \textrm{if } |k| < \frac{d+3}{\pi T}, \\ \exp \left( - \frac{d+1}{4 \log 2} \left( \log \frac{\pi T |k|}{d+3} \right) \left( \log \frac{e^2 \pi T |k|}{d+2} \right) \right) & \textrm{if } |k| \ge \frac{d+3}{\pi T} . \end{array} \right.
\end{equation}

We work with the kernel $K(x,y)=G(x-y)$. The dispersion rate is $D_{\infty}(K)=T$ because $\mathrm{supp} \, G = B(0,T)$, hence an application of Lemma \ref{smoothinglemma} yields
\[ W_{\infty} (\mu, \nu) \le \left( 1+\mathds{1}_{\{ \mu \neq \mathrm{Vol} \}} \right) T + W_{\infty} (K * \mu, K * \nu) . \]
Clearly, $K*\mu \ge c \mathrm{Vol}$. The assumptions $\nu (B) \ge b$ for all closed balls $B$ of radius $r$, and $T \ge 5r$, together with formula \eqref{Glowerbound} imply
\[ \begin{split} \frac{\mathrm{d}(K*\nu)}{\mathrm{d}\mathrm{Vol}}(x) &= \int_{\mathbb{R}^d / \mathbb{Z}^d} G(x-y) \, \mathrm{d}\nu (y) \ge \int_{B(x,r)} G(x-y) \, \mathrm{d}\nu (y) \\ &\ge \nu (B(x,r)) \frac{\Gamma \left( \frac{d+2}{2} \right) 2^{d+1}}{27 \pi^{d/2} T^d} \ge b \frac{\Gamma \left( \frac{d+2}{2} \right) 2^{d+1}}{27 \pi^{d/2} T^d} . \end{split} \]
In particular, $K*\nu \ge \delta \mathrm{Vol}$ with
\[ \delta = \min \left\{ b \frac{\Gamma \left( \frac{d+2}{2} \right) 2^{d+1}}{27 \pi^{d/2} T^d}, c \right\} , \]
and Lemma \ref{ledouxlemma} gives
\[ W_{\infty} (K*\mu, K*\nu) \le \frac{\log c - \log \delta}{c-\delta} \| \mu-\nu \|_{\dot{H}_{-1}^{\infty}} . \]

To estimate the dual Sobolev norm in the previous formula, let $f \in \mathrm{Lip}(\mathbb{R}^d / \mathbb{Z}^d)$ be arbitrary. The Fourier transforms
\[ \widehat{\frac{\mathrm{d}(K*\mu)}{\mathrm{d}\mathrm{Vol}}}(k) = \widehat{G}(k) \widehat{\mu}(k) \quad \textrm{and} \quad \widehat{\frac{\mathrm{d}(K*\nu)}{\mathrm{d}\mathrm{Vol}}}(k) = \widehat{G}(k) \widehat{\nu}(k), \quad k \in \mathbb{Z}^d, \]
and an application of the Parseval formula and Lemma \ref{torushausdorffyounglemma} lead to
\[ \begin{split} \left| \int_{\mathbb{R}^d / \mathbb{Z}^d} f \, \mathrm{d} (K*\mu - K*\nu) \right| &= \left| \sum_{k \in \mathbb{Z}^d} \overline{\widehat{f}(k)} \widehat{G}(k) (\widehat{\mu}(k) - \widehat{\nu}(k)) \right| \\ &\le \left( \sup_{k \in \mathbb{Z}^d} 2 \pi |k| |\widehat{f}(k)| \right) \Bigg( \sum_{\substack{k \in \mathbb{Z}^d \\ k \neq 0}} |\widehat{G}(k)| \frac{|\widehat{\mu}(k) - \widehat{\nu}(k)|}{2 \pi |k|} \Bigg) \\ &\le \| \nabla f \|_1 \Bigg( \sum_{\substack{k \in \mathbb{Z}^d \\ k \neq 0}} |\widehat{G}(k)| \frac{|\widehat{\mu}(k) - \widehat{\nu}(k)|}{2 \pi |k|} \Bigg). \end{split} \]
By the definition \eqref{dualsobolevnorm} of the dual Sobolev norm, this implies
\[ \| K*\mu - K*\nu \|_{\dot{H}_{-1}^{\infty}} \le \sum_{\substack{k \in \mathbb{Z}^d \\ k \neq 0}} |\widehat{G}(k)| \frac{|\widehat{\mu}(k) - \widehat{\nu}(k)|}{2 \pi |k|}, \]
hence
\[ W_{\infty} (\mu, \nu) \le \left( 1+\mathds{1}_{\{ \mu \neq \mathrm{Vol} \}} \right) T + \frac{\log c - \log \delta}{c-\delta} \sum_{\substack{k \in \mathbb{Z}^d \\ k \neq 0}} |\widehat{G}(k)| \frac{|\widehat{\mu}(k) - \widehat{\nu}(k)|}{2 \pi |k|} . \]
The desired estimate follows from the upper bound \eqref{Ghatupperbound} for $|\widehat{G}(k)|$.
\end{proof}

\subsection{Compact Lie groups}\label{liegroupproofsection}

In this section, we prove Theorems \ref{liegroupjacksonkerneltheorem} and \ref{liegroupheatkerneltheorem}. Let us fix some terminology and notation first. Let $G$ be a compact, connected Lie group of dimension $1 \le d < \infty$. Let $\mathfrak{g}$ denote the Lie algebra and $\exp: \mathfrak{g} \to G$ the exponential map. Fix an Ad-invariant scalar product $\langle \cdot, \cdot \rangle$ on $\mathfrak{g}$, and let $|X| = \langle X,X \rangle^{1/2}$, $X \in \mathfrak{g}$. The scalar product $\langle \cdot, \cdot \rangle$ defines a Riemannian metric on $G$.

Fix a maximal torus $T$ in $G$, and let $r_G = \mathrm{dim} \, T$ be the rank of $G$. For a more precise notation, in this section let $\mathrm{Vol}_G$ resp.\ $\mathrm{Vol}_T$ denote the normalized Riemannian volume measure, which is also the normalized Haar measure, on $G$ resp.\ $T$. The Lie algebra $\mathfrak{t}$ of $T$ is a Euclidean space of dimension $r_G$ with the restriction of $\langle \cdot, \cdot \rangle$ as scalar product. The dual $\mathfrak{t}^* = \mathrm{Hom}(\mathfrak{t},\mathbb{R})$ is also a Euclidean space of dimension $r_G$ with the scalar product and the norm defined by duality and also denoted by $\langle \cdot , \cdot \rangle$ and $|\cdot|$.

The set $\Gamma = \{ X \in \mathfrak{t} : \exp (2 \pi X) =1 \}$, where $1 \in G$ is the identity element, and its dual $\Gamma^* = \{ w \in \mathfrak{t}^* : w(X) \in \mathbb{Z} \textrm{ for all } X \in \Gamma \}$ are both lattices of full rank. The weights of a representation are elements of $\Gamma^*$. Let $R=R(G,T)$ be the set of roots, $R^+=R^+(G,T)$ the set of positive roots, $\{ u \in \mathfrak{t}^* : \langle u, w \rangle \ge 0 \textrm{ for all } w \in R^+ \}$ the dominant Weyl chamber, $W(G,T)=N_G(T)/T$ the Weyl group, $\rho^+=\sum_{w \in R^+} w/2$ the half-sum of positive roots and $v=\min_{w \in R} |w|$.

We start by generalizing Lemma \ref{torushausdorffyounglemma} to compact Lie groups.
\begin{lem}\label{liegrouphausdorffyounglemma} Let $G$ be a compact, connected Lie group, and let $f \in \mathrm{C}^{\infty}(G)$. For any $2 \le p < \infty$ and $1<q \le 2$ such that $1/p+1/q=1$, we have
\[ \left( \sum_{\pi \in \widehat{G}} d_{\pi}^{2-p/2} \lambda_{\pi}^{p/2} \| \widehat{f} (\pi) \|_{\mathrm{HS}}^p \right)^{1/p} \le \| \nabla f \|_q \]
and
\[ \sup_{\pi \in \widehat{G}} d_{\pi}^{-1/2} \lambda_{\pi}^{1/2} \| \widehat{f} (\pi) \|_{\mathrm{HS}} \le \| \nabla f \|_1 . \]
\end{lem}

\begin{proof} Fix an orthonormal basis $X_1, X_2, \dots, X_d$ in the Lie algebra $\mathfrak{g}$ with respect to $\langle \cdot, \cdot \rangle$. The Laplace--Beltrami operator on $G$ is $\Delta = \sum_{k=1}^d X_k X_k$ as an element of the universal enveloping algebra of $\mathfrak{g}$, regardless of the choice of the orthonormal basis.

The Fourier transform of a vector-valued function $g: G \to \mathbb{C}^n$ with coordinates $g=(g_1, g_2, \ldots, g_n)$ is defined coordinatewise as $\widehat{g}(\pi) = (\widehat{g_1}(\pi), \widehat{g_2}(\pi), \ldots, \widehat{g_n}(\pi))$, $\pi \in \widehat{G}$. Applying the Parseval formula coordinatewise shows that $\sum_{\pi \in \widehat{G}} d_{\pi} \| \widehat{g}(\pi) \|_{\mathrm{HS}}^2 = \| g \|_2^2$. For any $\pi \in \widehat{G}$, we have
\[ \| \widehat{g}(\pi) \|_{\mathrm{HS}} = \left\| \int_G g \pi^* \, \mathrm{d} \mathrm{Vol}_G \right\|_{\mathrm{HS}} \le \int_G |g| \cdot \| \pi^* \|_{\mathrm{HS}} \, \mathrm{d} \mathrm{Vol}_G = d_{\pi}^{1/2} \| g \|_1 , \]
hence $\sup_{\pi \in \widehat{G}} d_{\pi}^{-1/2} \| \widehat{g}(\pi) \|_{\mathrm{HS}} \le \| g \|_1$. An application of the Riesz--Thorin interpolation theorem for operators between $L^p$ spaces of vector-valued functions \cite[p.\ 83]{HNVW} leads to the Hausdorff--Young inequality for vector-valued functions
\begin{equation}\label{HYgroups}
\left( \sum_{\pi \in \widehat{G}} d_{\pi}^{2-p/2} \| \widehat{g}(\pi) \|_{\mathrm{HS}}^p \right)^{1/p} \le \| g \|_q .
\end{equation}
The same inequality for scalar-valued functions $g$ was proved in \cite[p.\ 548]{RT}.

Now let $f \in \mathrm{C}^{\infty} (G)$. Sugiura \cite{SU} proved that $\| \widehat{f}(\pi) \|_{\mathrm{HS}} \ll \lambda_{\pi}^{-C}$ with any $C>0$, and the Fourier series of $f$ converges pointwise:
\[ f(x) = \sum_{\pi \in \widehat{G}} d_{\pi} \mathrm{Tr} \left( \widehat{f}(\pi) \pi (x) \right) , \qquad x \in G. \]
As usual, we identify the elements of $\mathfrak{g}$ (the tangent space of $G$ at the identity element) with left-invariant vector fields. The $j$th coordinate of $\nabla f$ at a point $x \in G$ is thus
\[ \langle \nabla f (x), X_j \rangle = (X_j f)(x) = \frac{\mathrm{d}}{\mathrm{d}t} f \left( x \exp (t X_j) \right) \mid_{t=0} . \]
Since $\pi$ is a homomorphism, here
\[ f \left( x \exp (t X_j) \right) = \sum_{\pi \in \widehat{G}} d_{\pi} \mathrm{Tr} \left( \widehat{f}(\pi) \pi (x) \pi (\exp (t X_j)) \right) . \]
The series converges fast enough so that we can differentiate term by term to deduce
\[ \langle \nabla f (x), X_j \rangle = \sum_{\pi \in \widehat{G}} d_{\pi} \mathrm{Tr} \left( \widehat{f}(\pi) \pi (x) \mathrm{d} \pi (X_j) \right) , \qquad x \in G, \]
where $\mathrm{d}\pi (X)= \frac{\mathrm{d}}{\mathrm{d}t} \pi (\exp (tX)) \mid_{t=0}$ is the derived representation of $\pi$. It follows that the Fourier coefficients of the $j$th coordinate of $\nabla f: G \to \mathfrak{g}$ are
\[ \widehat{\langle \nabla f, X_j \rangle} (\pi) = \mathrm{d} \pi (X_j) \widehat{f}(\pi) , \]
consequently
\[ \| \widehat{\nabla f} (\pi) \|_{\mathrm{HS}}^2 = \sum_{j=1}^d \| \mathrm{d} \pi (X_j) \widehat{f}(\pi) \|_{\mathrm{HS}}^2 = \mathrm{Tr} \Bigg( \widehat{f}(\pi)^* \Bigg( \sum_{j=1}^d \mathrm{d} \pi (X_j)^* \mathrm{d} \pi (X_j) \Bigg) \widehat{f}(\pi) \Bigg) = \lambda_{\pi} \| \widehat{f}(\pi) \|_{\mathrm{HS}}^2 . \]
Note that we used $\sum_{j=1}^d \mathrm{d} \pi (X_j)^* \mathrm{d} \pi (X_j) = \lambda_{\pi} I_{d_{\pi} \times d_{\pi}}$, where $I_{d_{\pi} \times d_{\pi}}$ is the $d_{\pi} \times d_{\pi}$ identity matrix. The claim thus follows from applying the Hausdorff--Young inequality \eqref{HYgroups} to $\nabla f$.
\end{proof}

\begin{proof}[Proof of Theorem \ref{liegroupheatkerneltheorem}] Let $1 \le p < \infty$ and $1<q \le \infty$ be such that $1/p+1/q=1$, and fix $t>0$. We work with the heat kernel $P_t(x,y)$ on $G$. Since the Ricci curvature of $G$ is positive semidefinite, the heat kernel satisfies the same lower bound $P_t(x,y) \ge (4 \pi t)^{-d/2} e^{-\rho(x,y)^2/(4t)}$ as on the torus, see \eqref{heatkernellowerbound}. An application of Lemmas \ref{smoothinglemma} and \ref{heatkerneldispersionlemma} shows
\begin{equation}\label{liegroupheatkernelstep1}
W_p (\mu, \nu) \le (1+(1-c)^{1/p}) D_p (P_t) + W_p (P_t * \mu, P_t * \nu) \le C_1 t^{1/2} + W_p (P_t * \mu, P_t * \nu)
\end{equation}
with
\[ C_1 = \left\{ \begin{array}{ll} 2^{1/2} (1+(1-c)^{1/p}) d^{1/2} & \textrm{if } 1 \le p \le 2, \\  (1+(1-c)^{1/p}) (2d+p)^{1/2} & \textrm{if } p>2. \end{array} \right. \]

Assume now $p>1$ and $c>0$. Following the steps in the proof of Theorem \ref{torusheatkerneltheorem}, we deduce $P_t * \mu \ge c \mathrm{Vol}_G$ and $P_t * \nu \ge \delta \mathrm{Vol}_G$ with
\[ \delta = \min \left\{ b \frac{e^{-r^2 / (4t)}}{(4 \pi t)^{d/2}}, c \right\} . \]
An application of Lemma \ref{ledouxlemma} thus yields
\begin{equation}\label{liegroupheatkernelstep2}
W_p (P_t * \mu, P_t * \nu) \le \frac{p(c^{1/p} - \delta^{1/p})}{c-\delta} \| \mu - \nu \|_{\dot{H}_{-1}^p} .
\end{equation}

To estimate the dual Sobolev norm in the previous formula, let $f \in \mathrm{C}^{\infty}(G)$ be arbitrary, and let $p_0=\max \{ p, 2 \}$ and $q_0=\min \{ q, 2 \}$. Since
\[ P_t(x,y) = \sum_{\pi \in \widehat{G}} d_{\pi} e^{-\lambda_{\pi}t} \sum_{i,j=1}^{d_{\pi}} \pi_{ij}(x) \overline{\pi_{ij}(y)} \]
is uniformly convergent in $x,y \in G$ for any fixed $t>0$ and $\{ d_{\pi}^{1/2} \pi_{ij} : \pi \in \widehat{G}, 1 \le i,j \le d_{\pi} \}$ is an orthonormal basis in $L^2(G)$, we deduce
\[ \begin{split} \widehat{\frac{\mathrm{d}(P_t * \mu)}{\mathrm{d}\mathrm{Vol}_G}}(\pi) &= \int_G \frac{\mathrm{d}(P_t * \mu)}{\mathrm{d}\mathrm{Vol}_G} (x) \pi (x)^* \, \mathrm{d}\mathrm{Vol}_G (x) = \int_G \int_G P_t (x,y) \pi (x)^* \, \mathrm{d}\mathrm{Vol}_G (x) \mathrm{d}\mu(y) \\ &=\sum_{\pi' \in \widehat{G}} d_{\pi'} e^{-\lambda_{\pi'}t} \sum_{i,j=1}^{d_{\pi'}} \int_G \pi'_{ij}(x) \pi(x)^* \, \mathrm{d}\mathrm{Vol}_G (x) \int_G \overline{\pi'_{ij}(y)} \, \mathrm{d}\mu(y) \\ &= e^{-\lambda_{\pi} t} \int_G \pi(y)^* \, \mathrm{d}\mu (y) = e^{-\lambda_{\pi}t} \widehat{\mu}(\pi) . \end{split} \]
The same holds for $\nu$. The previous formula and an application of the Parseval formula, the Cauchy--Schwarz inequality $|\mathrm{Tr}(A^*B)| \le \| A \|_{\mathrm{HS}} \| B \|_{\mathrm{HS}}$, H\"older's inequality and Lemma \ref{liegrouphausdorffyounglemma} lead to
\[ \begin{split} \Bigg| \int_G f \, \mathrm{d} &(P_t * \mu - P_t * \nu) \Bigg| \\ &= \left| \sum_{\pi \in \widehat{G}} d_{\pi} \mathrm{Tr} \left( \widehat{f}(\pi)^* \left( \widehat{\frac{\mathrm{d}(P_t * \mu)}{\mathrm{d}\mathrm{Vol}_G}}(\pi) - \widehat{\frac{\mathrm{d}(P_t * \nu)}{\mathrm{d}\mathrm{Vol}_G}}(\pi) \right) \right) \right| \\ &= \left| \sum_{\pi \in \widehat{G}} d_{\pi} e^{-\lambda_{\pi}t} \mathrm{Tr} \left( \widehat{f}(\mu)^* (\widehat{\mu}(\pi) - \widehat{\nu}(\pi)) \right) \right| \\ &\le \sum_{\pi \in \widehat{G}} d_{\pi} e^{-\lambda_{\pi}t} \| \widehat{f}(\pi) \|_{\mathrm{HS}} \| \widehat{\mu}(\pi) - \widehat{\nu}(\pi) \|_{\mathrm{HS}} \\ &\le \Bigg( \sum_{\pi \in \widehat{G}} d_{\pi}^{2-p_0/2} \lambda_{\pi}^{p_0/2} \| \widehat{f}(\pi) \|_{\mathrm{HS}}^{p_0} \Bigg)^{1/p_0} \Bigg( \sum_{\substack{\pi \in \widehat{G} \\ \pi \neq \pi_0}} d_{\pi}^{2-q_0/2} e^{-\lambda_{\pi} q_0 t} \frac{\| \widehat{\mu}(\pi) - \widehat{\nu}(\pi) \|_{\mathrm{HS}}^{q_0}}{\lambda_{\pi}^{q_0/2}} \Bigg)^{1/q_0} \\ &\le \| \nabla f \|_{q_0} \Bigg( \sum_{\substack{\pi \in \widehat{G} \\ \pi \neq \pi_0}} d_{\pi}^{2-q_0/2} e^{-\lambda_{\pi} q_0 t} \frac{\| \widehat{\mu}(\pi) - \widehat{\nu}(\pi) \|_{\mathrm{HS}}^{q_0}}{\lambda_{\pi}^{q_0/2}} \Bigg)^{1/q_0} . \end{split} \]
By the definition \eqref{dualsobolevnorm} of the dual Sobolev norm and Lemma \ref{dualsobolevnormlemma}, this implies
\[ \| \mu - \nu \|_{\dot{H}_{-1}^p} \le \| \mu - \nu \|_{\dot{H}_{-1}^{p_0}} \le \Bigg( \sum_{\substack{\pi \in \widehat{G} \\ \pi \neq \pi_0}} d_{\pi}^{2-q_0/2} e^{-\lambda_{\pi} q_0 t} \frac{\| \widehat{\mu}(\pi) - \widehat{\nu}(\pi) \|_{\mathrm{HS}}^{q_0}}{\lambda_{\pi}^{q_0/2}} \Bigg)^{1/q_0} . \]
The claim for $p>1$ follows from \eqref{liegroupheatkernelstep1}, \eqref{liegroupheatkernelstep2} and the previous formula.

The proof for $p=1$ and $c \ge 0$ is very similar. Let $f \in \mathrm{C}^{\infty}(G)$ be arbitrary. Using $|\nabla f| \le \| f \|_{\mathrm{Lip}}$, we similarly deduce
\[ \begin{split} \left| \int_G f \, \mathrm{d} (P_t * \mu - P_t * \nu) \right| &\le \| \nabla f \|_2 \Bigg( \sum_{\substack{\pi \in \widehat{G} \\ \pi \neq \pi_0}} d_{\pi} e^{-2 \lambda_{\pi} t} \frac{\| \widehat{\mu}(\pi) - \widehat{\nu}(\pi) \|_{\mathrm{HS}}^2}{\lambda_{\pi}} \Bigg)^{1/2} \\ &\le \| f \|_{\mathrm{Lip}} \Bigg( \sum_{\substack{\pi \in \widehat{G} \\ \pi \neq \pi_0}} d_{\pi} e^{-2 \lambda_{\pi} t} \frac{\| \widehat{\mu}(\pi) - \widehat{\nu}(\pi) \|_{\mathrm{HS}}^2}{\lambda_{\pi}} \Bigg)^{1/2} . \end{split} \]
The Kantorovich duality formula in the form given in Lemma \ref{dualsobolevnormlemma} thus implies
\[ W_1 (P_t * \mu, P_t * \nu) \le \Bigg( \sum_{\substack{\pi \in \widehat{G} \\ \pi \neq \pi_0}} d_{\pi} e^{-2 \lambda_{\pi} t} \frac{\| \widehat{\mu}(\pi) - \widehat{\nu}(\pi) \|_{\mathrm{HS}}^2}{\lambda_{\pi}} \Bigg)^{1/2} . \]
The claim for $p=1$ follows from \eqref{liegroupheatkernelstep1} and the previous formula.
\end{proof}

A trigonometric polynomial on $T$ is a function of the form $f(t)=\sum_{n=1}^N a_n e^{2 \pi i w_n (X)}$, $t=\exp (2 \pi X) \in T$, $X \in \mathfrak{t}$ with some $a_n \in \mathbb{C}$ and $w_n \in \Gamma^*$, $1 \le n \le N$. The degree of $f$ is defined as $\max\{ |w_n| : 1 \le n \le N, \,\, a_n \neq 0 \}$ (not necessarily an integer). We say that $f$ is invariant under the action of the Weyl group if $f(xtx^{-1})=f(t)$ for all $t \in T$ and $x \in W(G,T)$. A central trigonometric polynomial on $G$ is a function of the form $f(x)=\sum_{n=1}^N a_n \chi_{\pi_n}(x)$, $x \in G$ with some $a_n \in \mathbb{C}$ and $\pi_n \in \widehat{G}$, where $\chi_{\pi}(x) = \mathrm{Tr}(\pi(x))$ denotes the character of $\pi \in \widehat{G}$. The degree of $f$ is defined as $\max \{ |w_{\pi_n}| : 1 \le n \le N, a_n \neq 0 \}$. The restriction map $f \mapsto f|_T$ is a degree preserving isomorphism from the space of central trigonometric polynomials on $G$ to the space of trigonometric polynomials on $T$ that are invariant under the action of the Weyl group \cite[Lemma 1]{CK}.

Recall that the Weyl integration formula \cite[p.\ 338]{BOU} states that for any $f \in L^1 (G)$ such that $f(y^{-1}xy) = f(x)$ for all $x,y \in G$, we have
\begin{equation}\label{weylintegration}
\int_G f \, \mathrm{Vol}_G = \frac{1}{|W(G,T)|} \int_T f \delta_G \, \mathrm{Vol}_T
\end{equation}
with the trigonometric polynomial
\[ \delta_G (t) = \prod_{w \in R} (e^{2 \pi i w(X)}-1) = \prod_{w \in R^+} |e^{2 \pi i w(X)}-1|^2 , \quad t=\exp (2 \pi X) \in T, \,\, X \in \mathfrak{t}. \]
Note that $\delta_G$ is invariant under the action of the Weyl group. Let $\mathrm{Vol}_{\mathfrak{t}}$ denote the Lebesgue measure on $\mathfrak{t}$ defined by $\langle \cdot, \cdot \rangle$, and let $\mathrm{Vol}_{\mathfrak{t}}(\mathfrak{t}/\Gamma)$ be the covolume of the lattice $\Gamma$.

\begin{proof}[Proof of Theorem \ref{liegroupjacksonkerneltheorem}] Let $1 \le p < \infty$ and $1<q \le \infty$ be such that $1/p+1/q=1$. Fix a real $H > 2 |\rho^+|+v$, and let $H'=\lfloor H/(2|\rho^+|+v) \rfloor$. We start by constructing a kernel $K(x,y)$ on the group $G$ based on Poisson summation, following \cite{BOR1,CK} with some modifications.

Let $g: \mathfrak{t} \to \mathbb{R}$ be the function from Lemma \ref{bumpfunctionlemma}, whose Fourier transform $\widehat{g}(Y)=\int_{\mathfrak{t}} g(X) e^{-2\pi i \langle X,Y \rangle} \, \mathrm{d}\mathrm{Vol}_{\mathfrak{t}}(X)$ is real-valued by construction. Define $F: \mathfrak{t} \to \mathbb{R}$, $F=C^{-1} (\widehat{g})^2 = C^{-1} \widehat{g*g}$, where $*$ denotes convolution, with the constant $C=\int_{\mathfrak{t}}(\widehat{g})^2 \, \mathrm{d}\mathrm{Vol}_{\mathfrak{t}}$. By the properties of $g$ established in Lemma \ref{bumpfunctionlemma}, we have
\[ C = \int_{\mathfrak{t}} (\widehat{g})^2 \, \mathrm{Vol}_{\mathfrak{t}} = \int_{\mathfrak{t}} g^2 \, \mathrm{d}\mathrm{Vol}_{\mathfrak{t}} \ge \left( \frac{2}{27} g(0) \right)^2 \frac{\pi^{r_G /2}}{5^{r_G} \Gamma \left( \frac{r_G +2}{2} \right)} = \frac{4^{r_G+1} \Gamma \left( \frac{r_G +2}{2} \right)}{3^6 5^{r_G} \pi^{r_G /2}} . \]
In particular, $\int_{\mathfrak{t}}F \, \mathrm{d} \mathrm{Vol}_{\mathfrak{t}}=1$ and
\begin{equation}\label{FXestimate}
0 \le F(X) \le \frac{3^6 5^{r_G} \pi^{r_G /2}}{4^{r_G +1} \Gamma \left( \frac{r_G +2}{2} \right)} \times \left\{ \begin{array}{ll} 1 & \textrm{if } |X| < (r_G +3)/\pi, \\ \exp \left( - \frac{r_G +1}{2 \log 2} \left( \log \frac{\pi |X|}{r_G +3} \right) \left( \log \frac{e^2 \pi |X|}{r_G +2} \right) \right) & \textrm{if } |X| \ge (r_G +3)/\pi. \end{array} \right.
\end{equation}
The Fourier transform $\widehat{F} = C^{-1} (g*g)$ is supported in $B(0,2)$.

The function
\[ S(t) = \mathrm{Vol}_{\mathfrak{t}}(\mathfrak{t}/\Gamma) \left( \frac{H'v}{2} \right)^{r_G} \sum_{Y \in \Gamma} F \left( \frac{H'v}{2} (X+Y) \right), \quad t=\exp (2 \pi X) \in T, \,\, X \in \mathfrak{t} \]
is well-defined on $T$, and invariant under the action of the Weyl group $W(G,T)$ because $\Gamma$ is invariant under the same action. By construction, the Fourier transform of $S$ is
\[ \widehat{S}(w) = \int_{\mathfrak{t}} \left( \frac{H'v}{2} \right)^{r_G} F \left( \frac{H'v}{2} X \right) e^{-2 \pi i w(X)} \, \mathrm{d}\mathrm{Vol}_{\mathfrak{t}}(X) = \widehat{F} \left( \frac{2 w^*}{H'v} \right) , \]
where $w \in \Gamma^*$ is identified with the unique $w^* \in \Gamma$ for which $w(X)=\langle w^* , X \rangle$ for all $X \in \mathfrak{t}$. Since $\widehat{F}(Y)=0$ for all $Y \in \mathfrak{t}$ such that $|Y| \ge 2$, we have $\widehat{S}(w)=0$ for all $|w| \ge H'v$. In particular, $S$ is a trigonometric polynomial of degree $< H'v$. The function
\[ \frac{\delta_G (t^{H'})}{\delta_G (t)} = \prod_{w \in R} \frac{e^{2 \pi i H' w(X)}-1}{e^{2 \pi i w(X)}-1} = \prod_{w \in R} \sum_{k=0}^{H'-1} e^{2 \pi i k w(X)} , \quad t=\exp (2 \pi X) \in T, \,\,  X \in \mathfrak{t} \]
is also a trigonometric polynomial on $T$ invariant under the action of $W(G,T)$, and has degree $\le |2\rho^+| (H'-1)$. In particular, $S(t) \delta_G (t^{H'}) / \delta_G (t)$ is a trigonometric polynomial on $T$ invariant under the action of $W(G,T)$, and has degree $<H'v+|2\rho^+| (H'-1) \le H$. Therefore \cite[Lemma 1]{CK}
\[ S(t) \frac{\delta_G (t^{H'})}{\delta_G (t)} = \sum_{\substack{\pi \in \widehat{G} \\ |w_{\pi}| < H}} a_{\pi} \chi_{\pi} (t), \quad t \in T \]
with suitable coefficients $a_{\pi} \in \mathbb{C}$.

We will work with the kernel
\[ K(x,y) = \sum_{\substack{\pi \in \widehat{G} \\ |w_{\pi}| < H}} a_{\pi} \chi_{\pi} (xy^{-1}), \quad x,y \in G . \]
The functions $S$ and $\delta_G$ are nonnegative on $T$ by construction. Since every element of $G$ is conjugate to an element of $T$ and the characters are invariant under conjugation, we have $K(x,y) \ge 0$ for all $x, y \in G$. The translation invariance of the Haar measure shows that $K(x,y)$ is normalized as
\[ \begin{split} \int_G K(x,y) \, \mathrm{d}\mathrm{Vol}_G (x) &= \sum_{\substack{\pi \in \widehat{G} \\ |w_{\pi}|<H}} a_{\pi} \int_{G} \chi_{\pi} (x) \, \mathrm{d}\mathrm{Vol}_G (x) = a_{\pi_0}, \\ \int_G K(x,y) \, \mathrm{d}\mathrm{Vol}_G (y) &= \sum_{\substack{\pi \in \widehat{G} \\ |w_{\pi}|<H}} a_{\pi} \int_{G} \overline{\chi_{\pi} (y)} \, \mathrm{d}\mathrm{Vol}_G (y) = a_{\pi_0}. \end{split} \]

We now find $a_{\pi_0}$ and estimate $a_{\pi}$. Expanding the product in the definition of $\delta_G$ shows that
\[ \delta_G (t) = \sum_{w \in I} c_w e^{2 \pi i w(X)}, \quad t=\exp (2 \pi X) \in T, \,\, X \in \mathfrak{t} \]
with a finite index set $I$ contained in the lattice spanned by the roots. An application of the Weyl integration formula \eqref{weylintegration} with the constant 1 function shows that the constant term is
\[ c_0 = \int_T \delta_G \, \mathrm{d}\mathrm{Vol}_T = |W(G,T)| . \]
Since $\delta_G$ is nonnegative, $|c_w| \le |W(G,T)|$ for all $w \in I$. The translation invariance of the Haar measure and another application of the Weyl integration formula \eqref{weylintegration} lead to
\[ \begin{split} a_{\pi_0} &= \int_G K(x,1) \, \mathrm{d}\mathrm{Vol}_G (x) \\ &= \frac{1}{|W(G,T)|} \int_T S(t) \delta_G (t^{H'}) \, \mathrm{d}\mathrm{Vol}_T (t) \\ &=\frac{1}{|W(G,T)|} \int_{\mathfrak{t}} \left( \frac{H'v}{2} \right)^{r_G} F \left( \frac{H' v}{2} X \right) \left( \sum_{w \in I} c_w e^{2 \pi i H' w(X)} \right) \, \mathrm{Vol}_{\mathfrak{t}}(X) \\ &= \frac{1}{|W(G,T)|} \sum_{w \in I} c_w \widehat{F} \left( - \frac{2w^*}{v} \right) , \end{split} \]
where $w \in \Gamma^*$ is identified with $w^* \in \Gamma$ by duality as before. Since $\widehat{F}$ is supported in the open ball $B(0,2)$ by construction, and $|2w^*/v| \ge 2$ for all $0 \neq w \in I$ by the definition $v=\min_{w \in R} |w|$, we have
\[ a_{\pi_0} = \frac{1}{|W(G,T)|} c_0 \widehat{F}(0)=1 . \]
Therefore $K(x,y)$ is normalized. The orthogonality relations for characters, $K(x,y) \ge 0$ and $|\chi_{\pi} (x)| \le d_{\pi}$ yield
\begin{equation}\label{apibound}
|a_{\pi}| = \left| \int_G K(x,1) \overline{\chi_{\pi}(x)} \, \mathrm{d}\mathrm{Vol}_G (x) \right| \le \int_G K(x,1) d_{\pi} \, \mathrm{d}\mathrm{Vol}_G (x) = d_{\pi} .
\end{equation}

We now estimate the dispersion rate of $K(x,y)$. Recall that the geodesic metric defined by the Ad-invariant scalar product on $\mathfrak{g}$ is bi-invariant, that is, $\rho (zx,zy)=\rho (xz,yz)=\rho (x,y)$ for all $x,y,z \in G$. In particular, the integral
\[ \int_G K(x,y) \rho (x,y)^p \, \mathrm{d}\mathrm{Vol}_G(x) = \int_G K(xy^{-1},1) \rho (xy^{-1},1)^p \, \mathrm{d}\mathrm{Vol}_G(x) = \int_G K(x,1) \rho (x,1)^p \, \mathrm{d}\mathrm{Vol}_G(x) \]
does not depend on $y \in G$. An application of the Weyl integration formula \eqref{weylintegration} thus shows
\[ \begin{split} D_p^p(K) &= \int_G K(x,1) \rho (x,1)^p \, \mathrm{d}\mathrm{Vol}(x) = \frac{1}{|W(G,T)|} \int_T S(t) \delta_G (t^{H'}) \rho (t,1)^p \, \mathrm{d}\mathrm{Vol}_T (t) \\ &= \frac{1}{|W(G,T)|} \int_{\mathfrak{t}} \left( \frac{H' v}{2} \right)^{r_G} F \left( \frac{H' v}{2} X \right) \delta_G (\exp (2 \pi H' X)) \rho (\exp (2 \pi X), 1)^p \, \mathrm{d}\mathrm{Vol}_{\mathfrak{t}}(X) . \end{split} \]
The estimates $\delta_G \le 2^{|R|}=2^{d-r_G}$ and $\rho (\exp (2 \pi X), 1) \le 2 \pi |X|$ lead to
\begin{equation}\label{Dppreliminary}
D_p^p (K) \le \frac{2^{d-r_G}}{|W(G,T)|} \left( \frac{4 \pi}{H' v} \right)^p \int_{\mathfrak{t}} F(X) |X|^p \, \mathrm{d} \mathrm{Vol}_{\mathfrak{t}}(X) .
\end{equation}
The upper bound \eqref{FXestimate} for $F(X)$ shows that
\[ \begin{split} \int_{B\left( 0, \frac{r_G+3}{\pi} \right)} F(X) |X|^p \, \mathrm{d}\mathrm{Vol}_{\mathfrak{t}}(X) &\le \frac{3^6 5^{r_G} \pi^{r_G /2}}{4^{r_G +1} \Gamma \left( \frac{r_G +2}{2} \right)} \int_0^{\frac{r_G+3}{\pi}} R^p \frac{2 \pi^{r_G/2}R^{r_G-1}}{\Gamma \left( \frac{r_G}{2} \right)} \, \mathrm{d}R \\ &= \frac{3^6 5^{r_G} (r_G+3)^{p+r_G}}{4^{r_G+1} \pi^p \Gamma \left( \frac{r_G+2}{2} \right) \Gamma \left( \frac{r_G}{2} \right)} \cdot \frac{2}{p+r_G}, \end{split} \]
and
\[ \begin{split} &\int_{\mathfrak{t} \backslash B \left( 0, \frac{r_G+3}{\pi} \right)} F(X) |X|^p \, \mathrm{d}\mathrm{Vol}_{\mathfrak{t}}(X) \\ &\le \frac{3^6 5^{r_G} \pi^{r_G /2}}{4^{r_G +1} \Gamma \left( \frac{r_G +2}{2} \right)} \int_{\frac{r_G+3}{\pi}}^{\infty} \exp \left( - \frac{r_G+1}{2 \log 2} \left( \log \frac{\pi R}{r_G+3} \right) \left( 2+\log \frac{\pi R}{r_G+3} \right) \right) R^p \frac{2 \pi^{r_G/2} R^{r_G-1}}{\Gamma \left( \frac{r_G}{2} \right)} \, \mathrm{d}R \\ &= \frac{3^6 5^{r_G} (r_G+3)^{p+r_G}}{4^{r_G+1} \pi^p \Gamma \left( \frac{r_G+2}{2} \right) \Gamma \left( \frac{r_G}{2} \right)} 2 \int_{0}^{\infty} \exp \left( - \frac{r_G+1}{2 \log 2} x (2+x) + (p+r_G) x \right) \, \mathrm{d}x \\ &\le \frac{3^6 5^{r_G} (r_G+3)^{p+r_G}}{4^{r_G+1} \pi^p \Gamma \left( \frac{r_G+2}{2} \right) \Gamma \left( \frac{r_G}{2} \right)} 2 \sqrt{\frac{\pi 2 \log 2}{r_G+1}} \exp \left( \frac{\log 2}{2} \left( \frac{p^2}{(r_G+1)} + (r_G+1) (1-1/\log 2)^2 \right) \right) . \end{split} \]
We used the substitution $\log (\pi R/(r_G+3))=x$ and the general formula $\int_{-\infty}^{\infty} e^{-Bx^2+Cx} \, \mathrm{d}x = \sqrt{\pi/B} e^{C^2/(4B)}$ with $B=(r_G+1)/(2 \log 2)$ and $C=p+r_G-(r_G+1)/\log 2$ so that $C^2 \le p^2+(1-1/\log 2)^2 (r_G+1)^2$. Adding the previous two estimates and checking that
\[ \begin{split} \frac{2}{p+r_G} + 2\sqrt{\frac{\pi 2 \log 2}{r_G+1}} 2^{p^2/(2r_G+2)+(r_G+1)(1-1/\log 2)^2 /2} &\le 1+ 2.96 \cdot 2^{p^2/(2r_G+2) + 0.098 (r_G+1)} \\ &\le 4.17 \cdot 2^{p^2/(2r_G+2)+0.098 r_G} \end{split} \]
leads to
\[ \int_{\mathfrak{t}} F(X) |X|^p \, \mathrm{d}\mathrm{Vol}_{\mathfrak{t}} (X) \le 4.17 \frac{3^6 5^{r_G} (r_G+3)^{p+r_G} 2^{p^2/(2r_G+2)+0.098 r_G}}{4^{r_G+1} \pi^p \Gamma \left( \frac{r_G+2}{2} \right) \Gamma \left( \frac{r_G}{2} \right)} . \]
Using $H' \ge H/(2|\rho^+|+v)-1$, formula \eqref{Dppreliminary} thus simplifies to
\[ D_p^p (K) \le \left( \frac{8 |\rho^+|+4v}{v(H-(2 |\rho^+|+v))} \right)^p 4.17 \frac{3^6 5^{r_G} (r_G+3)^{p+r_G} 2^{p^2/(2r_G+2)+d-0.902 r_G}}{4^{r_G+1}\Gamma \left( \frac{r_G+2}{2} \right) \Gamma \left( \frac{r_G}{2} \right) |W(G,T)|} . \]
By Stirling's approximation \eqref{stirling}, here $\Gamma (\frac{r_G+2}{2}) \Gamma (\frac{r_G}{2}) \ge (r_G/2)^{r_G} e^{-r_G} 2 \pi$, hence
\[ \begin{split} D_p^p (K) &\le \left( \frac{8 |\rho^+|+4v}{v(H-(2 |\rho^+|+v))} \right)^p (r_G+3)^p 4.17 \frac{3^6}{8 \pi} 2^{p^2/(2r_G+2)+d} \left( \frac{5\cdot 2^{0.098}e(r_G+3)}{4r_G} \right)^{r_G} \\ &\le \left( \frac{8 |\rho^+|+4v}{v(H-(2 |\rho^+|+v))} \right)^p (r_G+3)^p 121 \cdot 2^{p^2 / (2r_G+2)+d} 15^{r_G} \end{split} \]
where we used $4.17 \cdot 3^6 /(8 \pi) < 121$ and $5\cdot 2^{0.098}e(r_G+3)/(4r_G) < 15$.

The rest of the proof is entirely analogous to the proof of Theorem \ref{liegroupheatkerneltheorem}. An application of Lemma \ref{smoothinglemma} and the previous formula yield
\[ W_p (\mu, \nu) \le \frac{C_1}{H-H_0} + W_p (K*\mu, K*\nu) \]
with the constants $H_0=2|\rho^+|+v$ and
\[ C_1 = (1+(1-c)^{1/p}) \frac{8 |\rho^+|+4v}{v} (r_G+3) 121^{1/p} 2^{p/(2r_G+2)+d/p} 15^{r_G /p} . \]
We have
\[ \widehat{\frac{\mathrm{d}(K*\mu)}{\mathrm{d}\mathrm{Vol}}}(\pi) = \frac{a_{\pi}}{d_{\pi}} \widehat{\mu}(\pi) , \]
and the same holds for $\nu$. Following the steps in the proof of Theorem \ref{liegroupheatkerneltheorem} with $\delta =0$, we deduce
\[ W_p (K*\mu, K*\nu) \le C_2 \Bigg( \sum_{\substack{\pi \in \widehat{G} \\ 0< |w_{\pi}|<H}} d_{\pi}^{2-q_0/2} \frac{|a_{\pi}|}{d_{\pi}} \cdot \frac{\| \widehat{\mu}(\pi) - \widehat{\nu}(\pi) \|_{\mathrm{HS}}^{q_0}}{\lambda_{\pi}^{q_0/2}} \Bigg)^{1/q_0} \]
with $q_0=\min \{ q, 2 \}$ and
\[ C_2 = \left\{ \begin{array}{ll} 1 & \textrm{if } p=1, \\ p c^{-1/q} & \textrm{if } p>1 . \end{array} \right. \]
An application of the estimate $|a_{\pi}| \le d_{\pi}$ from \eqref{apibound} concludes the proof.
\end{proof}

\subsection{Compact homogeneous spaces}

In this section, we prove Theorem \ref{homogeneousheatkerneltheorem}. Recall that the Ricci curvature of any compact homogeneous space is positive semidefinite \cite{SA}.
\begin{lem}\label{groupactionlemma} Let $X$ be a compact metric space, and let $G$ be a group acting transitively on $X$. Let $\mu$ be a Borel probability measure on $X$, and let $V$ be a finite dimensional linear subspace of the vector space of all bounded, Borel-measurable, complex-valued functions on $X$. Assume that:
\begin{enumerate}[(i)]
\item for all Borel sets $B \subseteq X$ and all $g \in G$, we have $\mu(gB)=\mu(B)$,

\item for all $f \in V$ and all $g \in G$, we have $f^g \in V$, where $f^g(x)=f(gx)$.
\end{enumerate}
Then for any orthonormal basis $f_1, f_2, \ldots, f_{\mathrm{dim}(V)}$ of $V$ equipped with the scalar product $\langle f,h \rangle = \int_X f \overline{h} \, \mathrm{d}\mu$, we have
\[ \sum_{m=1}^{\mathrm{dim}(V)} |f_m (x)|^2 = \mathrm{dim}(V) \qquad \textrm{for all } x \in X. \]
\end{lem}

\begin{proof} Assumption (i) means that $G$ acts with measure preserving transformations, therefore $\langle f^g, h^g \rangle = \langle f,h \rangle$ for all $g \in G$.

Given any $a \in X$, the map $f \mapsto f(a)$ is a linear functional on $V$, hence there exists $F_a \in V$ such that $\langle f,F_a \rangle=f(a)$ for all $f \in V$. For any $g \in G$ and $f \in V$,
\[ \langle f, F_{ga} \rangle = f(ga) = f^g(a)=\langle f^g, F_a \rangle = \langle f,F_a^{g^{-1}} \rangle , \]
consequently $F_{ga}=F_a^{g^{-1}}$. Thus $\| F_{ga} \|_2=\| F_a \|_2$, and by the assumption that $G$ acts transitively, $\| F_a \|_2$ does not depend on $a$. Therefore given any orthonormal basis $f_1, f_2, \ldots, f_{\mathrm{dim}(V)}$ of $V$, the value of
\[ \| F_a \|_2^2 = \sum_{m=1}^{\mathrm{dim}(V)} |\langle F_a, f_m \rangle|^2 = \sum_{m=1}^{\mathrm{dim}(V)} |f_m (a)|^2 \]
does not depend on $a$. On the other hand,
\[ \int_M \sum_{m=1}^{\mathrm{dim}(V)} |f_m (x)|^2 \, \mathrm{d}\mu(x)= \sum_{m=1}^{\mathrm{dim}(V)} \| f_m \|_2^2 = \mathrm{dim}(V), \]
which proves that $\sum_{m=1}^{\mathrm{dim}(V)} |f_m(a)|^2=\mathrm{dim}(V)$ for all $a \in X$.
\end{proof}

The following lemma is the Hausdorff--Young inequality on homogeneous spaces. An equivalent form of \eqref{homogeneousHY1} and \eqref{homogeneousHY2} in terms of representation theory was given in \cite{NRT}. For the sake of completeness, we include a different proof.
\begin{lem}\label{homogeneoushausdorffyounglemma} Let $M$ be a compact, connected, smooth Riemannian manifold without boundary, and assume that the Riemannian isometry group acts transitively on $M$. Let $f \in \mathrm{C}^{\infty}(M)$. For any $2 \le p < \infty$ and $1<q \le 2$ such that $1/p+1/q=1$, we have
\begin{align}
\Bigg( \sum_{\ell=0}^{\infty} d_{\ell}^{1-p/2} \Bigg( \sum_{m=1}^{d_{\ell}} |\widehat{f}(\ell,m)|^2 \Bigg)^{p/2} \Bigg)^{1/p} &\le \| f \|_q , \label{homogeneousHY1} \\ \sup_{\ell \ge 0} d_{\ell}^{-1/2} \Bigg( \sum_{m=1}^{d_{\ell}} |\widehat{f}(\ell,m)|^2 \Bigg)^{1/2} &\le \| f \|_1 , \label{homogeneousHY2} \\ \Bigg( \sum_{\ell=0}^{\infty} d_{\ell}^{1-p/2} \lambda_{\ell}^{p/2} \Bigg( \sum_{m=1}^{d_{\ell}} |\widehat{f}(\ell,m)|^2 \Bigg)^{p/2} \Bigg)^{1/p} &\le 2(p-1) \| \nabla f \|_q , \label{homogeneousHY3} \\ \sup_{\ell \ge 0} d_{\ell}^{-1/2} \lambda_{\ell}^{1/2} \Bigg( \sum_{m=1}^{d_{\ell}} |\widehat{f}(\ell,m)|^2 \Bigg)^{1/2} &\le \| \nabla f \|_1 . \label{homogeneousHY4}
\end{align}
\end{lem}

\begin{proof} Claim \eqref{homogeneousHY1} for $p=q=2$ reduces to the Parseval formula. Now fix $\ell \ge 0$. Since the Laplace--Beltrami operator commutes with Riemannian isometries, the vector space $H_{\ell}$ is invariant under the action of the Riemannian isometry group. The Riemannian volume measure is also invariant under Riemannian isometries. An application of Lemma \ref{groupactionlemma} shows that the vector-valued function $\Phi_{\ell}=(\phi_{\ell,1}, \phi_{\ell,2}, \ldots, \phi_{\ell, d_{\ell}})$ has constant norm
\begin{equation}\label{additionformula}
|\Phi_{\ell}(x)|^2 = \sum_{m=1}^{d_{\ell}} |\phi_{\ell,m}(x)|^2 = d_{\ell} .
\end{equation}
The triangle inequality for vector-valued integrals
\[ \left| \int_M f \Phi_{\ell} \, \mathrm{d}\mathrm{Vol} \right| \le \int_M |f| |\Phi_{\ell}| \, \mathrm{d} \mathrm{Vol} \]
thus reads
\[ \Bigg( \sum_{m=1}^{d_{\ell}} |\widehat{f}(\ell,m)|^2 \Bigg)^{1/2} \le d_{\ell}^{1/2} \| f \|_1, \]
and claim \eqref{homogeneousHY2} follows.

Let $X$ be the measure space on the set of nonnegative integers with the measure $\mu(\{ \ell \})=d_{\ell}$, $\ell \ge 0$, and let $Y$ be the Hilbert space of square summable complex sequences\footnote{We avoid the standard notation $\ell^2$ of this Hilbert space as it might be confused with the indices $\ell$.}. Consider $L^p (X,Y)$, the $L^p$ space of $Y$-valued functions on the measure space $X$. Given $f \in L^1 (M)$, let $Tf$ be the map $X \to Y$,
\[ \ell \mapsto d_{\ell}^{-1/2} (\widehat{f}(\ell,1), \widehat{f}(\ell,2), \ldots, \widehat{f}(\ell,d_{\ell}), 0, 0, \ldots). \]
The linear operator $T$ thus has operator norms $\| T \|_{L^2(M) \to L^2 (X,Y)} \le 1$ and $\| T \|_{L^1 (M) \to L^{\infty}(X,Y)} \le 1$. An application of the Riesz--Thorin interpolation theorem for operators between $L^p$ spaces of Hilbert space-valued functions \cite[p.\ 83]{HNVW} yields $\| T \|_{L^q(M) \to L^p(X,Y)} \le 1$, which is exactly claim \eqref{homogeneousHY1}.

Claim \eqref{homogeneousHY3} follows from applying \eqref{homogeneousHY1} to $(-\Delta)^{1/2} f$, the Riesz transform estimate \eqref{riesztransform1} and the fact that the Fourier transform of $(-\Delta)^{1/2}f$ is $\widehat{(-\Delta)^{1/2}f}(\ell, m) = \lambda_{\ell}^{1/2} \widehat{f}(\ell,m)$.

In order to prove \eqref{homogeneousHY4}, we may assume that $\phi_{\ell,m}$ is real-valued since $\sum_{m=1}^{d_{\ell}} |\widehat{f}(\ell,m)|^2$ does not depend on the choice of the orthonormal basis $\phi_{\ell,m}$. Applying the Laplace--Beltrami operator to both sides of the identity \eqref{additionformula} and using the product rule $\Delta (fg) = f \Delta g + 2 \langle \nabla f, \nabla g \rangle + g \Delta f$ leads to
\[ 0= \sum_{m=1}^{d_{\ell}} \left( 2 \phi_{\ell,m} \Delta \phi_{\ell,m} + 2 |\nabla \phi_{\ell,m}|^2 \right) = -2 \lambda_{\ell} \sum_{m=1}^{d_{\ell}} \phi_{\ell,m}^2 + 2 \sum_{m=1}^{d_{\ell}} |\nabla \phi_{\ell,m}|^2 = -2d_{\ell} \lambda_{\ell} + 2 \sum_{m=1}^{d_{\ell}} |\nabla \phi_{\ell,m}|^2 , \]
and in particular to the identity
\begin{equation}\label{nablaidentity}
\sum_{m=1}^{d_{\ell}} |\nabla \phi_{\ell,m}|^2 = d_{\ell} \lambda_{\ell} .
\end{equation}
An application of Green's formula shows that for any $1 \le m \le d_{\ell}$,
\[ \lambda_{\ell} \widehat{f}(\ell,m) = \int_M f (-\Delta \phi_{\ell,m}) \, \mathrm{d}\mathrm{Vol} = \int_M \langle \nabla f, \nabla \phi_{\ell,m} \rangle \, \mathrm{d}\mathrm{Vol} . \]
Letting $Z=(Z_1, Z_2, \ldots, Z_{d_{\ell}})$ with coordinates $Z_m=\langle \nabla f, \nabla \phi_{\ell,m} \rangle$, the previous formula can be written in the form of the vector-valued integral
\[ \lambda_{\ell} (\widehat{f}(\ell,1), \widehat{f}(\ell,2), \ldots, \widehat{f}(\ell,d_{\ell})) = \int_M Z \, \mathrm{d}\mathrm{Vol} . \]
An application of the triangle inequality for vector-valued integrals, the Cauchy--Schwarz inequality and the identity \eqref{nablaidentity} shows that for any $\ell \ge 0$,
\[ \begin{split} \lambda_{\ell} \Bigg( \sum_{m=1}^{d_{\ell}} |\widehat{f}(\ell,m)|^2 \Bigg)^{1/2} &= \left| \int_M Z \, \mathrm{d}\mathrm{Vol} \right| \le \int_M |Z| \, \mathrm{d}\mathrm{Vol} = \int_M \Bigg( \sum_{m=1}^{d_{\ell}} |\langle \nabla f, \nabla \phi_{\ell,m} \rangle|^2 \Bigg)^{1/2} \, \mathrm{d}\mathrm{Vol} \\ &\le \int_M |\nabla f| \Bigg( \sum_{m=1}^{d_{\ell}} |\nabla \phi_{\ell,m}|^2 \Bigg)^{1/2} \, \mathrm{d}\mathrm{Vol} = d_{\ell}^{1/2} \lambda_{\ell}^{1/2} \| \nabla f \|_1 . \end{split} \]
This finishes the proof of \eqref{homogeneousHY4}.
\end{proof}

\begin{proof}[Proof of Theorem \ref{homogeneousheatkerneltheorem}] This is entirely analogous to the proofs of Theorems \ref{torusheatkerneltheorem} and \ref{liegroupheatkerneltheorem}, with Lemma \ref{homogeneoushausdorffyounglemma} playing the role of Lemmas \ref{torushausdorffyounglemma} and \ref{liegrouphausdorffyounglemma}. See also the proofs of Theorems \ref{manifoldheatkernelp<2theorem} and \ref{manifoldheatkernelp>2theorem} below.
\end{proof}

\subsection{Sphere with $1 \le p < \infty$}

\begin{proof}[Proof of Theorem \ref{sphereprojectionkerneltheorem}] Let $1 \le p < \infty$ and $1<q \le \infty$ be such that $1/p+1/q=1$, and fix an integer $L \ge 0$. We start by constructing a kernel $K(x,y)$ on the sphere $\mathbb{S}^d$. We work with the standard model $\mathbb{S}^d=\{ x \in \mathbb{R}^{d+1} : |x|=1 \}$, whereby the geodesic distance $\rho(x,y)=\theta \in [0,\pi]$ between the unit vectors $x,y \in \mathbb{S}^d$ is expressed in terms of the scalar product $\langle x,y \rangle = \cos \theta$. The general integration formula
\begin{equation}\label{sphereintegration}
\int_{\mathbb{S}^d} f(\langle x,y \rangle) \, \mathrm{d}\mathrm{Vol}(x) =\frac{\mathrm{Volume}(\mathbb{S}^{d-1})}{\mathrm{Volume}(\mathbb{S}^{d})} \int_{-1}^1 f(t) (1-t^2)^{d/2-1} \, \mathrm{d}t
\end{equation}
follows from integrating in cylindrical coordinates. Note that the left-hand side does not depend on $y \in \mathbb{S}^d$ because of the rotation invariance of the Riemannian volume measure.

Let $P_{\ell}^{(\alpha, \beta)}(t)$ denote the Jacobi polynomials, and $C_{\ell}^{(\lambda)}(t)$ the Gegenbauer polynomials with the standard normalization $P_{\ell}^{(\alpha, \beta)}(1)=\binom{\ell+\alpha}{\ell}$ and $C_{\ell}^{(\lambda)} (1)=\binom{\ell+2\lambda-1}{\ell}$. We refer to Szeg\H{o} \cite{SZ} for their definition and basic properties.

Consider the decomposition $L^2(\mathbb{S}^d)=\oplus_{\ell=0}^{\infty} H_{\ell}$ into an orthogonal direct sum of Laplace eigenspaces $H_{\ell}$ of dimension $d_{\ell} = \mathrm{dim}(H_{\ell})= \binom{\ell+d}{d} - \binom{\ell+d-2}{d}$, $\ell \ge 0$. In particular, $\oplus_{\ell=0}^L H_{\ell}$ has dimension
\[ \sum_{\ell=0}^L d_{\ell} = \binom{L+d}{d} + \binom{L+d-1}{d} . \]
Let $\phi_{\ell,m}$, $1 \le m \le d_{\ell}$ be an arbitrary orthonormal basis in $H_{\ell}$. Our starting point is the projection kernel $\sum_{\ell=0}^L \sum_{m=1}^{d_{\ell}} \phi_{\ell,m}(x) \overline{\phi_{\ell,m}(y)}$, which in fact does not depend on the choice of the orthonormal basis $\phi_{\ell,m}$. The addition formula for spherical harmonics states that
\begin{equation}\label{additiongegenbauer}
\sum_{m=1}^{d_{\ell}} \phi_{\ell,m}(x) \overline{\phi_{\ell,m}(y)} = \frac{2 \ell +d-1}{d-1} C_{\ell}^{(\frac{d-1}{2})} (\langle x, y \rangle) .
\end{equation}
An application of the summation formula \cite[p.\ 71, Equation (4.5.3)]{SZ} (which follows from the Christoffel--Darboux formula \cite[p.\ 43, Theorem 3.2.2]{SZ}) then yields
\begin{equation}\label{projectionkernel}
\sum_{\ell=0}^L \sum_{m=1}^{d_{\ell}} \phi_{\ell,m}(x) \overline{\phi_{\ell,m}(y)} = \frac{\binom{L+d}{d} + \binom{L+d-1}{d}}{\binom{L+\frac{d}{2}}{L}} P_L^{(\frac{d}{2}, \frac{d}{2}-1)} (\langle x,y \rangle) .
\end{equation}

Now let $N$ be a positive integer, and let us raise the projection kernel \eqref{projectionkernel} to the $N$th power. Since $P_{\ell}^{(\frac{d}{2}, \frac{d}{2}-1)}(t)$ has degree $\ell$, the Jacobi polynomials form a basis in the set of all polynomials $\mathbb{R}[t]$. The polynomial $P_L^{(\frac{d}{2}, \frac{d}{2}-1)} (t)^N$ of degree $NL$ is thus a linear combination of $P_{\ell}^{(\frac{d}{2}, \frac{d}{2}-1)}(t)$, $0 \le \ell \le NL$. On the other hand, \eqref{projectionkernel} shows that the functions $P_{\ell}^{(\frac{d}{2}, \frac{d}{2}-1)}(\langle x,y \rangle)$, $0 \le \ell \le NL$ are all linear combinations of $\sum_{m=1}^{d_{\ell}} \phi_{\ell,m}(x) \overline{\phi_{\ell,m}(y)}$, $0 \le \ell \le NL$, hence
\begin{equation}\label{Nthpowerkernel}
\left( \sum_{\ell=0}^L \sum_{m=1}^{d_{\ell}} \phi_{\ell,m}(x) \overline{\phi_{\ell,m}(y)} \right)^N = \sum_{\ell=0}^{NL} a_{\ell} \sum_{m=1}^{d_{\ell}} \phi_{\ell,m}(x) \overline{\phi_{\ell,m}(y)}
\end{equation}
with some real coefficients $a_{\ell}=a_{\ell}(d,N,L)$.

The Parseval formula and the addition formula \eqref{additionformula} show that for any $y \in \mathbb{S}^d$,
\[ c_N = \int_{\mathbb{S}^d} \left( \sum_{\ell=0}^L \sum_{m=1}^{d_{\ell}} \phi_{\ell,m}(x) \overline{\phi_{\ell,m}(y)} \right)^{2N} \, \mathrm{d}\mathrm{Vol}(x) = \sum_{\ell=0}^{NL} a_{\ell}^2 \sum_{m=1}^{d_{\ell}} |\phi_{\ell, m}(y)|^2 = \sum_{\ell=0}^{NL} a_{\ell}^2 d_{\ell} . \]
On the other hand, setting $x=y$ in \eqref{Nthpowerkernel}, an application of the Cauchy--Schwarz inequality leads to
\[ \left( \sum_{\ell=0}^L d_{\ell} \right)^N = \sum_{\ell=0}^{NL} a_{\ell} d_{\ell} \le \left( \sum_{\ell=0}^{NL} d_{\ell} \right)^{1/2} \left( \sum_{\ell=0}^{NL} a_{\ell}^2 d_{\ell} \right)^{1/2} . \]
Therefore the normalizing constant $c_N$ does not depend on $y \in \mathbb{S}^d$, and satisfies the lower bound
\begin{equation}\label{cNlowerbound}
c_N \ge \frac{\left( \sum_{\ell=0}^L d_{\ell} \right)^{2N}}{\sum_{\ell=0}^{NL} d_{\ell}} = \frac{\left( \binom{L+d}{d} + \binom{L+d-1}{d} \right)^{2N}}{\binom{NL+d}{d} + \binom{NL+d-1}{d}} .
\end{equation}

We will work with the kernel
\[ K(x,y)  = \frac{1}{c_N} \left( \sum_{\ell=0}^L \sum_{m=1}^{d_{\ell}} \phi_{\ell,m}(x) \overline{\phi_{\ell,m}(y)} \right)^{2N} = \frac{\left( \binom{L+d}{d} + \binom{L+d-1}{d} \right)^{2N}}{c_N \binom{L+\frac{d}{2}}{L}^{2N}} P_L^{(\frac{d}{2}, \frac{d}{2}-1)} (\langle x,y \rangle)^{2N} \]
with $N=\lfloor (p+5)/2 \rfloor \ge 3$.
By construction, $K(x,y)=K(y,x) \ge 0$ and $\int_{\mathbb{S}^d} K(x,y) \, \mathrm{d}\mathrm{Vol}(x)=1$ for all $y \in \mathbb{S}^d$.

Let us estimate the Fourier coefficients of the kernel $K$. By construction, $K(x,y)$ is a polynomial in the variable $\langle x,y \rangle$ of degree $2NL$. Since $C_{\ell}^{(\frac{d-1}{2})}(t)$ has degree $\ell$, the Gegenbauer polynomials form a basis of the set of all polynomials $\mathbb{R}[t]$. Hence
\begin{equation}\label{kernelfourierexpansion}
K(x,y) = \sum_{\ell=0}^{2NL} b_{\ell} \frac{2\ell +d-1}{d-1} C_{\ell}^{(\frac{d-1}{2})}(\langle x,y \rangle) = \sum_{\ell=0}^{2NL} b_{\ell} \sum_{m=1}^{d_{\ell}} \phi_{\ell,m}(x) \overline{\phi_{\ell,m}(y)}
\end{equation}
with suitable real coefficients $b_{\ell}=b_{\ell}(d,N,L)$. An application of the general integration formula \eqref{sphereintegration} and the orthogonality relation for Gegenbauer polynomials \cite[p.\ 81]{SZ}
\begin{equation}\label{gegenbauerorthogonality}
\int_{-1}^1 C_{\ell}^{(\lambda)}(t) C_m^{(\lambda)}(t) (1-t^2)^{\lambda-1/2} \, \mathrm{d}t = \left\{ \begin{array}{ll} \displaystyle{\frac{2^{1-2\lambda} \pi \Gamma (\ell+2\lambda)}{\Gamma (\lambda)^2 \ell! (\ell+\lambda)}} & \textrm{if } \ell=m, \\ 0 & \textrm{if } \ell \neq m \end{array} \right.
\end{equation}
express the coefficients $b_{\ell}$, $0 \le \ell \le 2NL$ in the form
\[ b_{\ell} \frac{2\ell +d-1}{d-1} \cdot \frac{\mathrm{Volume}(\mathbb{S}^{d-1}) 2^{2-d} \pi (\ell+d-2)!}{\mathrm{Volume}(\mathbb{S}^d) \Gamma \left( \frac{d-1}{2} \right)^2 \ell! \left( \ell +\frac{d-1}{2} \right)} = \int_{\mathbb{S}^d} K(x,y) C_{\ell}^{(\frac{d-1}{2})}(\langle x,y \rangle) \, \mathrm{d} \mathrm{Vol}(x) . \]
Using $\int_{\mathbb{S}^d} K(x,y) \, \mathrm{d}\mathrm{Vol}(x)=1$ and \cite[p.\ 171, Theorem 7.33.1]{SZ}
\begin{equation}\label{gegenbauersupremum}
\max_{t \in [-1,1]} |C_{\ell}^{(\lambda)}(t)| = C_{\ell}^{(\lambda)}(1) = \binom{\ell+2\lambda-1}{\ell},
\end{equation}
after some simplification we deduce
\begin{equation}\label{bellbound}
\begin{split} |b_{\ell}| &= \frac{\Gamma \left( \frac{d}{2} \right) \Gamma \left( \frac{d-1}{2} \right) \ell !}{2^{2-d} \pi^{1/2} (\ell+d-2)!} \left| \int_{\mathbb{S}^d} K(x,y) C_{\ell}^{(\frac{d-1}{2})}(\langle x,y \rangle) \, \mathrm{d}\mathrm{Vol}(x) \right| \\ &\le \frac{\Gamma \left( \frac{d}{2} \right) \Gamma \left( \frac{d-1}{2} \right) \ell !}{2^{2-d} \pi^{1/2} (\ell+d-2)!} \max_{t \in [-1,1]} |C_{\ell}^{(\frac{d-1}{2})}(t)| \\ &= \frac{\Gamma \left( \frac{d}{2} \right) \Gamma \left( \frac{d-1}{2} \right)}{2^{2-d} \pi^{1/2} (d-2)!} =1 . \end{split}
\end{equation}
In the last step we used the Legendre duplication formula $\Gamma (z) \Gamma (z+1/2) = 2^{1-2z} \pi^{1/2} \Gamma (2z)$.

Assume $L \ge d$, and let us now estimate the dispersion rate of $K(x,y)$. Since the geodesic distance is $\rho(x,y)=\arccos (\langle x,y \rangle)$, an application of the general integration formula \eqref{sphereintegration} leads to
\[ \begin{split} D_p^p (K) &= \inf_{y \in \mathbb{S}^d} \int_{\mathbb{S}^d} K(x,y) \rho(x,y)^p \, \mathrm{d}\mathrm{Vol}(x) \\ &= \frac{\Gamma \left( \frac{d+1}{2} \right) \left( \binom{L+d}{d} + \binom{L+d-1}{d}\right)^{2N}}{\pi^{1/2} \Gamma \left( \frac{d}{2} \right) c_N \binom{L+\frac{d}{2}}{L}^{2N}} \int_{-1}^1 P_L^{(\frac{d}{2}, \frac{d}{2}-1)} (t)^{2N} (\arccos t)^p (1-t^2)^{d/2-1} \, \mathrm{d}t . \end{split} \]
The estimates $\Gamma((d+1)/2) / \Gamma(d/2) \le (d/2)^{1/2}$ and $\arccos t \le (\pi/2^{1/2}) (1-t)^{1/2}$, $t \in [-1,1]$ and the lower bound \eqref{cNlowerbound} show that
\begin{equation}\label{Dpsphere}
D_p^p(K) \le \frac{d^{1/2} \pi^{p-1/2} \binom{NL+d}{d}}{2^{(p-1)/2} \binom{L+\frac{d}{2}}{L}^{2N}} \int_{-1}^1 P_L^{(\frac{d}{2}, \frac{d}{2}-1)}(t)^{2N} (1-t)^{(d+p)/2-1}(1+t)^{d/2-1} \, \mathrm{d}t .
\end{equation}

Let $\varepsilon = 2d^2/(2L+d)^2$, and note that $\varepsilon \le 2/9$ by the assumption $L \ge d$. The absolute value of the Jacobi polynomials attain their maxima at one of the endpoints of $[-1,1]$. More precisely \cite[p.\ 168]{SZ}, if $\alpha, \beta >-1/2$, then letting $t_0=(\beta-\alpha)/(\alpha+\beta+1)$, we have
\[ \begin{split} \max_{t \in [-1, t_0]} |P_L^{(\alpha, \beta)}(t)|&=|P_L^{(\alpha, \beta)}(-1)|=\binom{L+\beta}{L}, \\ \max_{t \in [t_0,1]} |P_L^{(\alpha, \beta)}(t)|&=|P_L^{(\alpha, \beta)}(1)|=\binom{L+\alpha}{L} . \end{split} \]
Setting $\alpha=d/2$ and $\beta=d/2-1$, we have $t_0=-1/d$. In particular,
\[ \begin{split} |P_L^{(\frac{d}{2}, \frac{d}{2}-1)}(t)| &\le \binom{L+\frac{d}{2}-1}{L} = \frac{d}{2L+d} \binom{L+\frac{d}{2}}{L} \textrm{ for all } t \in [-1, -1+\varepsilon], \\ |P_L^{(\frac{d}{2}, \frac{d}{2}-1)}(t)| &\le \binom{L+\frac{d}{2}}{L} \textrm{ for all } t \in [1-\varepsilon, 1] . \end{split} \]
Therefore
\begin{equation}\label{near-1}
\begin{split} \int_{-1}^{-1+\varepsilon} P_L^{(\frac{d}{2}, \frac{d}{2}-1)}(t)^{2N} &(1-t)^{(d+p)/2-1}(1+t)^{d/2-1} \, \mathrm{d}t \\ &\le \frac{d^{2N}}{(2L+d)^{2N}} \binom{L+\frac{d}{2}}{L}^{2N} \int_{-1}^{-1+\varepsilon} 2^{(d+p)/2-1} (1+t)^{d/2-1} \, \mathrm{d}t \\ &= \frac{d^{2N-1}2^{(d+p)/2}}{(2L+d)^{2N}} \binom{L+\frac{d}{2}}{L}^{2N} \varepsilon^{d/2} \end{split}
\end{equation}
and
\begin{equation}\label{near1}
\begin{split} \int_{1-\varepsilon}^1 P_L^{(\frac{d}{2}, \frac{d}{2}-1)}(t)^{2N} &(1-t)^{(d+p)/2-1}(1+t)^{d/2-1} \, \mathrm{d}t \\ &\le \binom{L+\frac{d}{2}}{L}^{2N} \int_{1-\varepsilon}^1 (1-t)^{(d+p)/2-1} 2^{d/2-1} \, \mathrm{d}t \\ &= \frac{2^{d/2}}{d+p} \binom{L+\frac{d}{2}}{L}^{2N} \varepsilon^{(d+p)/2} . \end{split}
\end{equation}

The Jacobi polynomials satisfy the following explicit upper bound \cite[Theorem 1]{KR2}. For any $\alpha \ge \beta \ge 0$ and any $L \ge 1$, letting
\[ \delta_{\pm 1} = \frac{\beta^2 - \alpha^2 \pm \sqrt{(2L+1)(2L+2\alpha+1)(2L+2\beta+1)(2L+2\alpha+2\beta+1)}}{(2L+\alpha+\beta+1)^2}, \]
we have $[\delta_{-1}, \delta_1] \subseteq [-1,1]$ and
\begin{equation}\label{Palphabeta}
\max_{t \in [\delta_{-1}, \delta_1]} (\delta_1-t)^{1/4} (-\delta_{-1}+t)^{1/4} (1-t)^{\alpha/2}(1+t)^{\beta/2} |P_L^{(\alpha, \beta)}(t)| \le \frac{3^{1/2}}{5^{1/4}} \| P_L^{(\alpha, \beta)} \|_2 ,
\end{equation}
where
\[ \| P_L^{(\alpha, \beta)} \|_2^2 = \int_{-1}^1 (1-t)^{\alpha}(1+t)^{\beta} P_L^{(\alpha, \beta)}(t)^2 \, \mathrm{d}t = \frac{2^{\alpha+\beta+1}\Gamma (L+\alpha+1) \Gamma (L+\beta+1)}{(2L+\alpha+\beta+1)\Gamma (L+\alpha+\beta+1)L!} . \]
Setting $\alpha=d/2$ and $\beta=d/2-1$, in the definition of $\delta_{\pm 1}$ we have
\[ \begin{split} \frac{\sqrt{(2L+1)(2L+d+1)(2L+d-1)(2L+2d-1)}}{(2L+d)^2} &= \sqrt{1 - \frac{d^2-2d+2}{(2L+d)^2} + \frac{d^2-2d+1}{(2L+d)^4}} \\ &\ge 1-\frac{d^2-2d+2}{(2L+d)^2} , \end{split} \]
hence
\[ \begin{split} \delta_1 &\ge 1-\frac{d^2-d+1}{(2L+d)^2} \ge 1-\frac{\varepsilon}{2}, \\ \delta_{-1} &\le -1+\frac{d^2-3d+3}{(2L+d)^2} \le -1+\frac{\varepsilon}{2} . \end{split} \]
In particular, $\delta_1 - t \ge (1-t)/2$ and $-\delta_{-1}+t \ge (1+t)/2$ for all $t \in [-1+\varepsilon, 1-\varepsilon]$. The estimate \eqref{Palphabeta} thus implies that for all $t \in [-1+\varepsilon, 1-\varepsilon]$,
\[ |P_L^{(\frac{d}{2}, \frac{d}{2}-1)}(t)| \le B_L^{1/2} (1-t)^{-(d+1)/4} (1+t)^{-(d-1)/4} \]
with the constant
\[ B_L= \frac{6}{5^{1/2}} \| P_L^{(\frac{d}{2}, \frac{d}{2}-1)} \|_2^2 = \frac{6 \cdot 2^d \Gamma \left( L+\frac{d}{2}+1 \right) \Gamma \left( L + \frac{d}{2} \right)}{5^{1/2} (2L+d) \Gamma (L+d) L!} . \]
Noting that the choice $N = \lfloor (p+5)/2 \rfloor$ ensures $N(d-1)/2 - d/2 \ge 1/2$ and $N(d+1)/2 - (d+p)/2 \ge 2$ as well as $N(d+1)/2 - (d+p)/2 \ge N(d-1)/2 - d/2 + 3/2$, we deduce
\[ \begin{split} \int_{-1+\varepsilon}^{1-\varepsilon} P_L^{(\frac{d}{2}, \frac{d}{2}-1)}(t)^{2N} &(1-t)^{(d+p)/2-1} (1+t)^{d/2-1} \, \mathrm{d}t \\ &\le B_L^N \int_{-1+\varepsilon}^{1-\varepsilon} (1-t)^{(d+p)/2-N(d+1)/2-1} (1+t)^{d/2-N(d-1)/2-1} \, \mathrm{d}t \\ &\le B_L^N \int_{-1+\varepsilon}^0 (1+t)^{d/2-N(d-1)/2-1} \, \mathrm{d}t + B_L^N \int_0^{1-\varepsilon} (1-t)^{(d+p)/2-N(d+1)/2-1} \, \mathrm{d}t \\ &\le B_L^N \left( 2 \varepsilon^{d/2-N(d-1)/2} + \frac{1}{2} \varepsilon^{(d+p)/2-N(d+1)/2}\right) \\ &\le B_L^{N} \varepsilon^{(d+p)/2-N(d+1)/2} \left(2 \varepsilon^{3/2}+\frac{1}{2} \right) . \end{split} \]
The previous formula together with $(2 \varepsilon^{3/2}+1/2) \le (2 (2/9)^{3/2}+1/2) < 0.71$ and \eqref{Dpsphere}--\eqref{near1} thus yield
\begin{equation}\label{Dplong}
\begin{split} D_p^p (K) \le \frac{d^{1/2} \pi^{p-1/2} \binom{NL+d}{d}}{2^{(p-1)/2}} \left( \frac{d^{2N-1} 2^{(d+p)/2}}{(2L+d)^{2N}} \varepsilon^{d/2} + \frac{2^{d/2}}{d+p} \varepsilon^{(d+p)/2} + \frac{0.71 B_L^N}{\binom{L+\frac{d}{2}}{L}^{2N}} \varepsilon^{(d+p)/2-N(d+1)/2} \right) . \end{split}
\end{equation}

Let us simplify the terms in \eqref{Dplong}. Substituting $\varepsilon=2d^2/(2L+d)^2$ and using $N \ge (p+3)/2$ and $L \ge d$ shows
\begin{equation}\label{Dpfirstterm}
\frac{d^{2N-1} 2^{(d+p)/2}}{(2L+d)^{2N}} \varepsilon^{d/2} = \left( \frac{d}{2L+d} \right)^{d+2N} \frac{2^{d+p/2}}{d} \le \left( \frac{d}{2L+d} \right)^{d+p+3} \frac{2^{d+p/2}}{d} \le \frac{d^{d+p-1}2^{d+p/2}}{27 (2L+d)^{d+p}}
\end{equation}
and
\begin{equation}\label{Dpsecondterm}
\frac{2^{d/2}}{d+p} \varepsilon^{(d+p)/2} = \frac{d^{d+p}2^{d+p/2}}{(d+p) (2L+d)^{d+p}} \le \frac{d^{d+p-1}2^{d+p/2}}{(2L+d)^{d+p}} .
\end{equation}
Consider now
\[ \frac{B_L}{\binom{L+\frac{d}{2}}{L}^2} = \frac{12 \cdot 2^d \Gamma \left( \frac{d}{2}+1 \right)^2 L!}{5^{1/2} (2L+d)^2 (L+d-1)!} . \]
Stirling's approximation \eqref{stirling} shows that here $\Gamma(\frac{d}{2} +1)^2 \le \pi d^{d+1} 2^{-d} e^{-d+1/6}$ and
\[ \begin{split} \frac{L!}{(L+d-1)!} &\le \frac{L^{L+1/2}}{(L+d-1)^{L+d-1/2}} e^{d-1+1/(12L)} \\ &= \left( 1-\frac{d-1}{L+d-1} \right)^{L+1/2} \left( 1-\frac{d/2-1}{L+d-1} \right)^{d-1} \frac{1}{(L+d/2)^{d-1}} e^{d-1+1/(12L)} \\ &\le \exp \left( - \frac{(d-1)(L+1/2)}{L+d-1} \right) \exp \left( - \frac{(d/2-1)(d-1)}{L+d-1} \right) \frac{2^{d-1}}{(2L+d)^{d-1}} e^{d-1+1/(12L)} \\ &\le \frac{2^{d-1}}{(2L+d)^{d-1}} e^{d^2/(2L)} . \end{split} \]
Using $6\pi e^{1/6}/5^{1/2}=9.95\ldots <10$ and $d^2/(2L) \le d/2$, we obtain
\[ \frac{B_L}{\binom{L+\frac{d}{2}}{L}^2} \le \frac{10 d^{d+1}2^d}{e^{d/2} (2L+d)^{d+1}} . \]
Substituting $\varepsilon=2d^2/(2L+d)^2$ thus leads to
\[ \frac{0.71 B_L^N}{\binom{L+\frac{d}{2}}{L}^{2N}} \varepsilon^{(d+p)/2-N(d+1)/2} \le 0.71 \left( \frac{10 \cdot 2^{1/2}}{e} \right)^N \frac{d^{d+p}2^{(d+p)/2}}{(2L+d)^{d+p}} . \]
The previous formula and \eqref{Dplong}--\eqref{Dpsecondterm} yield
\[ D_p^p(K) \le \frac{d^{1/2} \pi^{p-1/2} \binom{NL+d}{d}}{2^{(p-1)/2}} \cdot \frac{d^{d+p}2^{d+p/2}}{(2L+d)^{d+p}} \left( \frac{28}{27d} + 0.71 \left( \frac{10 \cdot 2^{1/2}}{e} \right)^N 2^{-d/2} \right) . \]
Stirling's approximation \eqref{stirling}, $N \ge 3$ and $L \ge d$ show that
\[ \begin{split} \binom{NL+d}{d} &\le \left(1+\frac{d}{NL} \right)^{NL+1/2} (NL+d)^{d} \frac{e^{1/(12(NL+d))}}{2^{1/2} \pi^{1/2} d^{d+1/2}} \\ &\le e^{d(NL+1/2)/NL} \frac{N^d}{2^d} (2L+d)^d \frac{e^{1/96}}{2^{1/2} \pi^{1/2} d^{d+1/2}} \\ &\le N^d (2L+d)^d \frac{e^{d+17/96}}{2^{d+1/2}\pi^{1/2} d^{d+1/2}} , \end{split} \]
consequently
\[ D_p^p (K) \le \frac{d^p \pi^{p-1}N^d e^{d+17/96}}{(2L+d)^p} \left( \frac{14}{27} + \frac{0.71}{2^{1/2}} \left( \frac{10 \cdot 2^{1/2}}{e} \right)^N \right) . \]
Using
\[ \frac{14}{27} + \frac{0.71}{2^{1/2}} \left( \frac{10 \cdot 2^{1/2}}{e} \right)^N \le \left( \frac{14}{27} \left( \frac{10 \cdot 2^{1/2}}{e} \right)^{-3} + \frac{0.71}{2^{1/2}} \right) \left( \frac{10 \cdot 2^{1/2}}{e} \right)^{(p+5)/2} \]
and
\[ \begin{split} \pi^{-1} e^{17/96} \left( \frac{14}{27} \left( \frac{10 \cdot 2^{1/2}}{e} \right)^{-3} + \frac{0.71}{2^{1/2}} \right) \left( \frac{10 \cdot 2^{1/2}}{e} \right)^{5/2} &= 11.86\ldots < 12 , \\ \pi \left( \frac{10 \cdot 2^{1/2}}{e} \right)^{1/2} &= 7.16\ldots <8 \end{split} \]
finally leads to
\[ D_p^p (K) \le 12 e^d \left(\frac{p+5}{2} \right)^d  \frac{8^p d^p}{(2L+d)^p} . \]

We are now ready to prove the theorem. An application of Lemma \ref{smoothinglemma} and the previous formula yields
\begin{equation}\label{spherestep1}
W_p (\mu, \nu) \le \frac{C_1'}{2L+d} + W_p (K*\mu, K*\nu)
\end{equation}
with $C_1'=(1+(1-c)^{1/p}) 12^{1/p} e^{d/p}(\frac{p+5}{2})^{d/p} 8d$.

Assume now $p>1$ and $c>0$. The assumption $\mu \ge c \mathrm{Vol}$ ensures that $K*\mu \ge c \mathrm{Vol}$, and an application of Lemma \ref{ledouxlemma} with $\delta=0$ gives
\begin{equation}\label{spherestep2}
W_p(K*\mu, K*\nu) \le \frac{p}{c^{1/q}} \| K*\mu - K*\nu \|_{\dot{H}_{-1}^p} .
\end{equation}
To estimate the dual Sobolev norm in the previous formula, let $f\in \mathrm{C}^{\infty}(\mathbb{S}^d)$ be arbitrary. The expansion \eqref{kernelfourierexpansion} of $K(x,y)$ shows that the Fourier transform of $K*\mu$ is
\[ \widehat{\frac{\mathrm{d} (K*\mu)}{\mathrm{d}\mathrm{Vol}}}(\ell, m) = \int_{\mathbb{S}^d} \int_{\mathbb{S}^d} K(x,y) \overline{\phi_{\ell,m}(x)} \mathrm{d}\mathrm{Vol}(x) \mathrm{d}\mu(y) = \mathds{1}_{[0,2NL]}(\ell) b_{\ell} \widehat{\mu}(\ell,m) , \]
and the same holds for $\nu$. In the case $p>2$, an application of the Parseval formula, the estimate $|b_{\ell}| \le 1$ from \eqref{bellbound}, H\"older's inequality and Lemma \ref{homogeneoushausdorffyounglemma} lead to
\[ \begin{split} \left| \int_{\mathbb{S}^d} f \, \mathrm{d}(K*\mu - K*\nu) \right| &= \left| \sum_{\ell=0}^{2NL} \sum_{m=1}^{d_{\ell}} \overline{\widehat{f}(\ell,m)} b_{\ell} (\widehat{\mu}(\ell,m) - \widehat{\nu}(\ell,m)) \right| \\ &\le \sum_{\ell=1}^{2NL} \left( \sum_{m=1}^{d_{\ell}} |\widehat{f}(\ell,m)|^2 \right)^{1/2} \left( \sum_{m=1}^{d_{\ell}} |\widehat{\mu}(\ell,m) - \widehat{\nu}(\ell,m)|^2 \right)^{1/2} \\ &\le \left( \sum_{\ell=1}^{2NL} d_{\ell}^{1-p/2} \lambda_{\ell}^{p/2} \left( \sum_{m=1}^{d_{\ell}} |\widehat{f}(\ell,m)|^2 \right)^{p/2} \right)^{1/p} \\ &{\phantom{\le}} \times \left( \sum_{\ell=1}^{2NL} d_{\ell}^{1-q/2} \lambda_{\ell}^{-q/2} \left( \sum_{m=1}^{d_{\ell}} |\widehat{\mu}(\ell,m) - \widehat{\nu}(\ell,m)|^2 \right)^{q/2} \right)^{1/q} \\ &\le 2(p-1) \| \nabla f \|_q \left( \sum_{\ell=1}^{2NL} d_{\ell}^{1-q/2} \lambda_{\ell}^{-q/2} \left( \sum_{m=1}^{d_{\ell}} |\widehat{\mu}(\ell,m) - \widehat{\nu}(\ell,m)|^2 \right)^{q/2} \right)^{1/q} . \end{split} \]
By the definition \eqref{dualsobolevnorm} of the dual Sobolev norm and Lemma \ref{dualsobolevnormlemma}, this implies
\begin{equation}\label{spherestep3}
\| K*\mu - K*\nu \|_{\dot{H}_{-1}^p} \le 2(p-1) \left( \sum_{\ell=1}^{2NL} d_{\ell}^{1-q/2} \lambda_{\ell}^{-q/2} \left( \sum_{m=1}^{d_{\ell}} |\widehat{\mu}(\ell,m) - \widehat{\nu}(\ell,m)|^2 \right)^{q/2} \right)^{1/q} .
\end{equation}
In the case $1<p \le 2$, instead of Lemma \ref{homogeneoushausdorffyounglemma} we simply use
\[ \| \nabla f \|_2^2 = \int_{\mathbb{S}^d} (-\Delta f) f \, \mathrm{d}\mathrm{Vol} = \sum_{\ell=0}^{\infty} \lambda_{\ell} \sum_{m=1}^{d_{\ell}} |\widehat{f}(\ell,m)|^2 \]
to similarly deduce
\begin{equation}\label{1<ple2estimate}
\left| \int_{\mathbb{S}^d} f \, \mathrm{d}(K*\mu - K*\nu) \right| \le \| \nabla f \|_2 \left( \sum_{\ell=1}^{2NL} \lambda_{\ell}^{-1} \sum_{m=1}^{d_{\ell}} |\widehat{\mu}(\ell,m) - \widehat{\nu}(\ell,m)|^2 \right)^{1/2} ,
\end{equation}
and thus
\begin{equation}\label{spherestep4}
\| K*\mu - K*\nu \|_{\dot{H}_{-1}^p} \le \| K*\mu - K*\nu \|_{\dot{H}_{-1}^2} \le \left( \sum_{\ell=1}^{2NL} \lambda_{\ell}^{-1} \sum_{m=1}^{d_{\ell}} |\widehat{\mu}(\ell,m) - \widehat{\nu}(\ell,m)|^2 \right)^{1/2} .
\end{equation}
Finally, in the case $p=1$ and $c \ge 0$ the estimate \eqref{1<ple2estimate} together with $\| \nabla f \|_2 \le \| f \|_{\mathrm{Lip}}$ and the Kantorovich duality formula in the form of Lemma \ref{dualsobolevnormlemma} show
\[ W_1 (K*\mu, K*\nu) \le \left( \sum_{\ell=1}^{2NL} \lambda_{\ell}^{-1} \sum_{m=1}^{d_{\ell}} |\widehat{\mu}(\ell,m) - \widehat{\nu}(\ell,m)|^2 \right)^{1/2} . \]

A combination of \eqref{spherestep1}--\eqref{spherestep3}, \eqref{spherestep4} and the previous formula shows that for any $1 \le p < \infty$ and any integer $L \ge d$,
\[ W_p (\mu, \nu) \le \frac{C_1'}{2L+d} + C_2 \left( \sum_{\ell=1}^{2NL} d_{\ell}^{1-q_0/2} \lambda_{\ell}^{-q_0/2} \left( \sum_{m=1}^{d_{\ell}} |\widehat{\mu}(\ell,m) - \widehat{\nu}(\ell,m)|^2 \right)^{q_0/2} \right)^{1/q_0} \]
with $q_0=\min \{ q,2 \}$ and $C_2$ as in the Theorem. Given now any integer $L \ge (p+5)d$, the integer $L'$ for which $2NL' \le L < 2N(L'+1)$ satisfies
\[ \frac{1}{2L'+d} \le \frac{N}{L+(d-2) N} \le \frac{p+5}{2L}, \]
therefore we can reformulate the estimate as
\[ W_p (\mu, \nu) \le \frac{C_1}{L} + C_2 \left( \sum_{\ell=1}^L d_{\ell}^{1-q_0/2} \lambda_{\ell}^{-q_0/2} \left( \sum_{m=1}^{d_{\ell}} |\widehat{\mu}(\ell,m) - \widehat{\nu}(\ell,m)|^2 \right)^{q_0/2} \right)^{1/q_0} \]
with $C_1=C_1' (p+5)/2$ as in the Theorem. Finally, note that the Theorem is trivial for integers $1 \le L < (p+5)d$, as in this case $C_1/L>\mathrm{diam}(\mathbb{S}^d) \ge W_p (\mu, \nu)$.
\end{proof}

\subsection{Sphere with $p=\infty$}

In this section, we prove Theorem \ref{sphereinfinitytheorem}. Let $d \ge 3$ and fix $2^{d+3} d^{-1/2} r \le T \le 1/d$. We start by constructing a kernel $K(x,y)$ on the sphere $\mathbb{S}^d$. Let $g: \mathbb{R} \to \mathbb{R}$ be the smooth function from Lemma \ref{bumpfunctionlemma} in 1 dimension, see also Remark \ref{1dimremark}. In particular, $g$ satisfies the following properties.
\begin{enumerate}[(i)]
\item \label{propi} $g$ is even, $g(x) \ge 0$, $\mathrm{supp} \, g = [-1,1]$, $g(0)=1$ and $\int_{-1}^1 g(x) \, \mathrm{d}x =1$.

\item \label{propii} For any $y \in \mathbb{R}$ such that $|y| \ge 1/\pi$,
\[ |\widehat{g}(y)| \le \exp \left( - \frac{1}{2 \log 2} \left( \log (\pi |y|) \right) \left( \log (2 \pi |y|) \right) \right) . \]

\item \label{propiii} $\widehat{g}$ is real-valued, and $\widehat{g}(y) \ge 1-2 \pi^2 |y|^2 /9$ for all $y \in \mathbb{R}$.
\end{enumerate}

The convolution $g*g$ is supported on $[-2,2]$, and its Fourier transform is $(\widehat{g})^2$. Let $\psi (t) = (\widehat{g}(\frac{Tt}{4 \pi}))^2$, $t \in \mathbb{R}$. Note that $0 \le \psi \le 1$, and its Fourier transform $\widehat{\psi}(t)=\frac{4\pi}{T} (g*g)(\frac{4 \pi t}{T})$, $t \in \mathbb{R}$ is supported on $[-T/(2 \pi),T/(2 \pi)]$. Let
\[ K_0 (x,y) = \sum_{\ell=0}^{\infty} \psi (\lambda_{\ell}^{1/2}) \sum_{m=1}^{d_{\ell}} \phi_{\ell,m}(x) \overline{\phi_{\ell,m}(y)} = \sum_{\ell=0}^{\infty} \psi (\lambda_{\ell}^{1/2}) \frac{2 \ell +d-1}{d-1} C_{\ell}^{(\frac{d-1}{2})}(\langle x, y \rangle), \qquad x, y \in \mathbb{S}^d , \]
and note that $K_0(x,y)=K_0(y,x)$ is real-valued. By the spectral theorem and the fact that $\widehat{\psi}$ is even, $K_0(x,y)$ is the kernel of the operator
\[ \psi ((-\Delta)^{1/2}) = \int_{-T/(2 \pi)}^{T/(2 \pi)} \widehat{\psi}(t) e^{2 \pi i t (-\Delta)^{1/2}} \, \mathrm{d}t = \frac{1}{\pi} \int_0^T \widehat{\psi} \left( \frac{t}{2 \pi} \right) \cos (t (-\Delta)^{1/2}) \, \mathrm{d}t . \]
Observe that $(\partial_t^2 - \Delta) \cos (t (-\Delta)^{1/2}) = 0$, $\cos (t (-\Delta)^{1/2})|_{t=0}=I$ and $\partial_t \cos (t (-\Delta)^{1/2})|_{t=0}=0$. By finite speed of propagation of the wave equation \cite[Theorem 2.4.2]{SO}, the support of $K_0(x,y)$ is contained in the set $\{ (x,y) \in \mathbb{S}^d : \rho (x,y) \le T \}$, that is, the kernel $K_0 (x,y)$ has dispersion rate $D_{\infty}(K_0) \le T$.

An application of the Parseval formula shows that the constant
\[ c_T = \int_{\mathbb{S}^d} K_0(x,y)^2 \, \mathrm{d}\mathrm{Vol}(x) = \sum_{\ell=0}^{\infty} \psi (\lambda_{\ell}^{1/2})^2 d_{\ell} \]
does not depend on $y \in \mathbb{S}^d$. We will work with the kernel
\[ K(x,y) = \frac{K_0(x,y)^2}{c_T} , \qquad x, y \in \mathbb{S}^d . \]
We have $K(x,y)=K(y,x) \ge 0$, $\int_{\mathbb{S}^d} K(x,y) \, \mathrm{d}\mathrm{Vol}(x)=1$ and $D_{\infty}(K) \le T$ by construction.

We give the proof of Theorem \ref{sphereinfinitytheorem} after three lemmas.
\begin{lem}\label{cTlemma} We have
\[ \frac{6^d}{32 d! T^d} \le c_T \le 6 \frac{2^{d^2/4+9d/4}}{d! T^d} . \]
\end{lem}

\begin{proof} Recall that
\[ \begin{split} \lambda_{\ell}&=\ell (\ell+d-1), \\ d_{\ell}&=\binom{\ell+d}{d} - \binom{\ell+d-2}{d} = \frac{2 \ell +d-1}{d-1} \binom{\ell+d-2}{d-2}, \\ \sum_{\ell=0}^L d_{\ell} &= \binom{L+d}{d} + \binom{L+d-1}{d}. \end{split} \]

Let $L_1 = \lceil 2^{(d+5)/4}/T \rceil$, and note that the assumption $T \le 1/d$ ensures $L_1 \ge d 2^{(d+5)/4} \ge (d+5)/2$. The estimate $0 \le \psi \le 1$ and the inequality between the arithmetic and geometric means lead to
\begin{equation}\label{cTfirstsum}
\begin{split} \sum_{\ell=0}^{L_1} \psi (\lambda_{\ell}^{1/2})^2 d_{\ell} &\le \sum_{\ell=0}^{L_1} d_{\ell} = \binom{L_1+d}{d} + \binom{L_1+d-1}{d} \le 2 \frac{(L_1+d) (L_1+d-1) \cdots (L_1+1)}{d!} \\ &\le \frac{2}{d!} \left( L_1 + \frac{d+1}{2} \right)^d \le \frac{2}{d!} (2(L_1-1))^d \le \frac{2^{d^2/4 + 9d/4+1}}{d! T^d} . \end{split}
\end{equation}
Let us use
\[ d_{\ell} = \frac{2\ell+d-1}{d-1} \binom{\ell+d-2}{d-2} \le \frac{2}{(d-1)!} \left( \ell + \frac{d-1}{2} \right)^{d-1} \le \frac{2}{(d-1)!}(2\ell)^{d-1}, \quad \ell \ge L_1+1, \]
$\lambda_{\ell}^{1/2} \ge \ell$ and property \eqref{propii} to estimate the remaining sum as
\[ \begin{split} \sum_{\ell=L_1+1}^{\infty} \psi(\lambda_{\ell}^{1/2})^2 d_{\ell} &= \sum_{\ell=L_1+1}^{\infty} \widehat{g} \bigg(\frac{T \lambda_{\ell}^{1/2}}{4 \pi} \bigg)^4 d_{\ell} \\ &\le \frac{2}{(d-1)!} \sum_{\ell=L_1+1}^{\infty} \exp \left( - \frac{2}{\log 2} \left( \log \frac{T \ell}{4} \right) \left( \log \frac{T \ell}{2} \right) + (d-1) \log (2 \ell) \right) . \end{split} \]
Let
\[ H(x) = - \frac{2}{\log 2} \left( \log \frac{Tx}{4} \right) \left( \log \frac{Tx}{2} \right) + (d-1) \log (2x) , \]
and note that $H'(x) \le 0$ is equivalent to $2^{(d+5)/4} \le Tx$ after some simplification. The choice of $L_1$ ensures that this inequality holds on the interval $x \in [L_1, \infty)$, hence $H(x)$, and consequently also $\exp (H(x))$ are nonincreasing. Therefore
\[ \begin{split} \sum_{\ell=L_1+1}^{\infty} \psi(\lambda_{\ell}^{1/2})^2 d_{\ell} &\le \frac{2}{(d-1)!} \int_{L_1}^{\infty} \exp \left( - \frac{2}{\log 2} \left( \log \frac{Tx}{4} \right) \left( \log \frac{Tx}{2} \right) + (d-1) \log (2x) \right) \, \mathrm{d}x \\ &\le \frac{8^d}{(d-1)! T^d} \int_{-\infty}^{\infty} \exp \left( - \frac{2}{\log 2} y^2 + (d-2) y \right) \, \mathrm{d}y \\ &= \frac{8^d}{(d-1)! T^d} \sqrt{\frac{\pi \log 2}{2}} 2^{(d-2)^2 /8}, \end{split} \]
where we applied the substitution $y=\log (Tx/4)$, and then extended the range of integration.
Adding \eqref{cTfirstsum} and the previous formula yields the upper bound of the claim:
\[ c_T \le \frac{2^{d^2/4+9d/4}}{d! T^d} \left( 2+\sqrt{\pi \log 2} d 2^{-d^2/8+d/4} \right) \le 6 \frac{2^{d^2/4+9d/4}}{d! T^d}. \]

Let $L_2=\lceil 6/T \rceil \ge 6d$, and consider the lower bound
\[ c_T \ge \sum_{\ell=0}^{L_2} \psi (\lambda_{\ell}^{1/2})^2 d_{\ell} = \sum_{\ell=0}^{L_2} \widehat{g}\bigg( \frac{T\lambda_{\ell}^{1/2}}{4 \pi} \bigg)^4 d_{\ell} . \]
Property \eqref{propiii} ensures that for all $0 \le \ell \le L_2$,
\begin{equation}\label{ghatL2}
\begin{split} \widehat{g} \bigg( \frac{T\lambda_{\ell}^{1/2}}{4 \pi} \bigg) &\ge 1- \frac{2\pi^2}{9} \left( \frac{T \lambda_{\ell}^{1/2}}{4 \pi} \right)^2 \ge 1- \frac{T^2}{72} L_2 (L_2+d-1) \ge 1-\frac{T^2}{72} \left( L_2 + \frac{d-1}{2} \right)^2 \\ &\ge 1-\frac{T^2}{72} \left( \frac{25}{22} (L_2-1) \right)^2 \ge 1-\frac{25^2 6^2}{72 \cdot 22^2} = \frac{343}{968} , \end{split}
\end{equation}
and the lower bound of the claim
\[ c_T \ge \sum_{\ell=0}^{L_2} \left( \frac{343}{968} \right)^4 d_{\ell} = \left( \frac{343}{968} \right)^4 \left( \binom{L_2+d}{d} + \binom{L_2+d-1}{d} \right) \ge \left( \frac{343}{968} \right)^4 \frac{2 L_2^d}{d!} \ge \frac{6^d}{32 d! T^d} \]
follows. In the last step we used $2 (343/968)^4 >1/32$.
\end{proof}

\begin{lem}\label{Kxylemma} For any $x,y \in \mathbb{S}^d$ with geodesic distance $\rho (x,y) \le d^{1/2} 2^{-d-3} T$, we have
\[ K(x,y) \ge \frac{9^d}{114 d! 2^{d^2/4+d/4} T^d} . \]
\end{lem}

\begin{proof} Fix $x,y \in \mathbb{S}^d$ such that $\rho (x,y) \le d^{1/2} 2^{-d-3} T$. In particular,
\[ 1 \ge \langle x, y \rangle = \cos \rho (x,y) \ge 1-\frac{\rho(x,y)^2}{2} \ge 1-\frac{d T^2}{2^{2d+7}} . \]
The derivative of the Gegenbauer polynomials are \cite[p.\ 81]{SZ} $\frac{\mathrm{d}}{\mathrm{d}t}C_{\ell}^{(\lambda)} (t) = 2 \lambda C_{\ell-1}^{(\lambda+1)}(t)$. The formula for the extreme values of the Gegenbauer polynomials given in \eqref{gegenbauersupremum} thus shows that the Lipschitz constant of $C_{\ell}^{(\frac{d-1}{2})}(t)$ on $[-1,1]$ is
\[ \max_{t \in [-1,1]} \left| \frac{\mathrm{d}}{\mathrm{d}t}C_{\ell}^{(\frac{d-1}{2})} (t) \right| = \max_{t \in [-1,1]} |(d-1) C_{\ell-1}^{(\frac{d+1}{2})}(t)| = (d-1) \binom{\ell+d-1}{d} . \]
Therefore
\begin{equation}\label{cellxybound}
C_{\ell}^{(\frac{d-1}{2})} (\langle x,y \rangle) \ge C_{\ell}^{(\frac{d-1}{2})}(1) - \frac{d T^2}{2^{2d+7}} (d-1) \binom{\ell+d-1}{d} = \left( 1 - \frac{T^2 \lambda_{\ell}}{2^{2d+7}} \right) \binom{\ell+d-2}{d-2} .
\end{equation}
Let again $L_2 = \lceil 6/T \rceil$, and let $L_3 = \lceil 2^{d+3}/T \rceil \ge 21d+22$. We will estimate the terms in the sum
\[ K_0(x,y) = \sum_{\ell=0}^{\infty} \widehat{g} \bigg(\frac{T \lambda_{\ell}^{1/2}}{4 \pi} \bigg)^2 \frac{2\ell+d-1}{d-1} C_{\ell}^{(\frac{d-1}{2})}(\langle x, y \rangle) \]
in the ranges $0 \le \ell \le L_2$, $L_2+1 \le \ell \le L_3$ and $L_3+1 \le \ell$ separately.

For all $0 \le \ell \le L_2$,
\[ 1-\frac{T^2 \lambda_{\ell}}{2^{2d+7}} \ge 1-\frac{T^2}{2^{2d+7}} L_2 (L_2+d-1) \ge 1-\frac{T^2}{2^{2d+7}} \left(\frac{25}{22}(L_2-1) \right)^2 \ge 1- \frac{25^2 6^2}{2^{2d+7}22^2} \ge 0.97. \]
The previous formula, \eqref{ghatL2} and \eqref{cellxybound} thus yield
\begin{equation}\label{firstsum}
\begin{split} \sum_{\ell=0}^{L_2} \widehat{g}\bigg( \frac{T \lambda_{\ell}^{1/2}}{4 \pi} \bigg)^2 \frac{2\ell+d-1}{d-1} C_{\ell}^{(\frac{d-1}{2})}(\langle x, y \rangle) &\ge \sum_{\ell=0}^{L_2} \left( \frac{343}{968} \right)^2 0.97 d_{\ell} \\ &= \left( \frac{343}{968} \right)^2 0.97 \left( \binom{L_2+d}{d} + \binom{L_2+d-1}{d} \right) \\ &\ge \left( \frac{343}{968} \right)^2 0.97 \frac{2 L_2^d}{d!} \ge 0.24 \frac{6^d}{d! T^d} . \end{split}
\end{equation}
In the last step we used $(343/968)^2 0.97 \cdot 2>0.24$. For any $L_2+1 \le \ell \le L_3$, we have
\[ \begin{split} 1 - \frac{T^2 \lambda_{\ell}}{2^{2d+7}} &\ge 1-\frac{T^2}{2^{2d+7}} \left( L_3 +\frac{d-1}{2} \right)^2 \ge 1-\frac{T^2}{2^{2d+7}} \left( \frac{43}{42} (L_3-1) \right)^2 \ge 1-\frac{43^2 2^{2d+6}}{2^{2d+7} 42^2} >0. \end{split} \]
Estimate \eqref{cellxybound} thus yields $C_{\ell}(\langle x, y  \rangle) >0$, consequently
\begin{equation}\label{secondsum}
\sum_{\ell=L_2+1}^{L_3} \widehat{g}\bigg( \frac{T \lambda_{\ell}^{1/2}}{4 \pi} \bigg)^2 \frac{2\ell+d-1}{d-1} C_{\ell}^{(\frac{d-1}{2})}(\langle x, y \rangle) \ge 0.
\end{equation}
Arguing as before,
\[ d_{\ell} \le \frac{2}{(d-1)!} \left( \ell+\frac{d-1}{2} \right)^{d-1} \le \frac{2}{(d-1)!} \left( \frac{43}{42} \ell \right)^{d-1}, \qquad \ell \ge L_3+1, \]
property \eqref{propii} and $TL_3 /4 \ge 2^{d+1}$ show that
\[ \begin{split} \bigg| \sum_{\ell=L_3+1}^{\infty} \widehat{g}\bigg( &\frac{T \lambda_{\ell}^{1/2}}{4 \pi} \bigg)^2 \frac{2\ell+d-1}{d-1} C_{\ell}^{(\frac{d-1}{2})}(\langle x, y \rangle) \bigg| \\ &\le \sum_{\ell=L_3+1}^{\infty} \widehat{g}\bigg( \frac{T\lambda_{\ell}^{1/2}}{4 \pi} \bigg)^2 d_{\ell} \\ &\le \frac{2}{(d-1)!} \sum_{\ell=L_3+1}^{\infty} \exp \left( - \frac{1}{\log 2} \left( \log \frac{T \ell}{4} \right) \left( \log \frac{T \ell}{2} \right) + (d-1) \log \frac{43\ell}{42} \right) \\ &\le \frac{2}{(d-1)!} \int_{L_3}^{\infty} \exp \left( - \frac{1}{\log 2} \left( \log\frac{Tx}{4} \right) \left( \log \frac{Tx}{2} \right) + (d-1) \log \frac{43x}{42} \right) \, \mathrm{d}x \\ &\le \frac{86^{d-1} 8}{(d-1)! 21^{d-1} T^d} \int_{\log \frac{T L_3}{4}}^{\infty} \exp \left( - \frac{1}{\log 2} y^2 + (d-1) y \right) \, \mathrm{d}y \\ &\le \frac{86^{d-1}8}{(d-1)! 21^{d-1} T^d} \int_{(d+1) \log 2}^{\infty} \frac{\frac{2}{\log 2} y - (d-1)}{d+3} \exp \left( - \frac{1}{\log 2} y^2 + (d-1) y \right) \, \mathrm{d}y \\ &= \frac{86^{d-1} 8}{(d-1)! 21^{d-1} T^d} \cdot \frac{2^{-2d-2}}{d+3}. \end{split} \]
Estimates \eqref{firstsum}, \eqref{secondsum} and the previous formula lead to the lower bound
\[ K_0 (x,y) \ge 0.24 \frac{6^d}{d! T^d} - \frac{86^{d-1}8}{21^{d-1} (d-1)! T^d} \cdot \frac{2^{-2d-2}}{d+3} \ge 0.23 \frac{6^d}{d! T^d} . \]
In particular, the previous formula and the upper bound bound part of Lemma \ref{cTlemma} establish the claim
\[ K(x,y) = \frac{K_0(x,y)^2}{c_T} \ge \frac{9^d}{114 d! 2^{d^2/4+d/4} T^d} . \]
\end{proof}

\begin{lem}\label{belllemma} In the expansion
\[ K(x,y) = \sum_{\ell=0}^{\infty} b_{\ell} \frac{2\ell +d-1}{d-1} C_{\ell}^{(\frac{d-1}{2})}(\langle x,y \rangle) = \sum_{\ell=0}^{\infty} b_{\ell} \sum_{m=1}^{d_{\ell}} \phi_{\ell,m}(x) \overline{\phi_{\ell,m}(y)} , \]
the real coefficients $b_{\ell}$, $\ell \ge 0$ satisfy
\[ |b_{\ell}| \le \left\{ \begin{array}{ll} 1 & \textrm{if } \ell \le 2^{d+2}/T, \\ 43 \exp \left( - \frac{1}{\log 2} \left( \log \frac{T\ell}{16} \right) \left( \log \frac{T\ell}{2^{d+2}} \right) \right) & \textrm{if } \ell > 2^{d+2}/T . \end{array} \right. \]
\end{lem}

\begin{proof} Since $K(x,y)$ is nonnegative and normalized, orthogonality shows
\[ \begin{split} |b_{\ell}| &= \left| \int_{\mathbb{S}^d} \int_{\mathbb{S}^d} K(x,y) \frac{1}{d_{\ell}} \sum_{m=1}^{d_{\ell}} \overline{\phi_{\ell,m}(x)} \phi_{\ell,m}(y) \, \mathrm{d} \mathrm{Vol}(x) \mathrm{d}\mathrm{Vol}(y) \right| \\ &\le \int_{\mathbb{S}^d} \int_{\mathbb{S}^d} K(x,y) \, \mathrm{d} \mathrm{Vol}(x) \mathrm{d}\mathrm{Vol}(y) = 1 . \end{split} \]
In particular, this proves the claim for $\ell \le 2^{d+2}/T$.

Assume now $\ell > 2^{d+2}/T$. Expanding the square in the definition of $K(x,y)$ leads to
\[ K(x,y) = \frac{1}{c_T} \sum_{\ell_1, \ell_2=0}^{\infty} \psi (\lambda_{\ell_1}^{1/2}) \psi (\lambda_{\ell_2}^{1/2}) \frac{(2 \ell_1+d-1)(2\ell_2 +d-1)}{(d-1)^2} C_{\ell_1}^{(\frac{d-1}{2})} (\langle x, y \rangle) C_{\ell_2}^{(\frac{d-1}{2})} (\langle x, y \rangle) . \]
The orthogonality relations for the Gegenbauer polynomials \eqref{gegenbauerorthogonality} expresses $b_{\ell}$, $\ell \ge 0$ in the form
\[ \begin{split} b_{\ell} \frac{2\ell +d-1}{d-1} \cdot \frac{2^{2-d}\pi (\ell+d-2)!}{\Gamma \left( \frac{d-1}{2} \right)^2 \ell!  \left( \ell +\frac{d-1}{2} \right)} = \frac{1}{c_T} \sum_{\ell_1, \ell_2=0}^{\infty} &\psi (\lambda_{\ell_1}^{1/2}) \psi (\lambda_{\ell_2}^{1/2}) \frac{(2\ell_1 +d-1)(2\ell_2 +d-1)}{(d-1)^2} \\ &\times \int_{-1}^1 C_{\ell_1}^{(\frac{d-1}{2})}(t) C_{\ell_2}^{(\frac{d-1}{2})}(t) C_{\ell}^{(\frac{d-1}{2})}(t) (1-t^2)^{d/2-1} \, \mathrm{d}t . \end{split} \]
The explicit formula for the integral of a triple product of Gegenbauer polynomials \cite[p.\ 321]{AAR} states that
\[ \int_{-1}^1 C_{\ell}^{(\lambda)}(t) C_{m}^{(\lambda)}(t) C_{n}^{(\lambda)}(t) (1-t^2)^{\lambda-1/2} \, \mathrm{d}t = \frac{(\lambda)_{s-\ell} (\lambda)_{s-m} (\lambda)_{s-n} s!}{(s-\ell)! (s-m)! (s-n)! (\lambda)_s} \cdot \frac{2^{1-2 \lambda} \pi \Gamma (s+2\lambda)}{\Gamma (\lambda)^2 s! (s+\lambda)} \]
if $\ell +m+n=2s$ is an even integer and the sum of any two of $\ell,m,n$ is greater or equal than the third, and the integral vanishes otherwise. Note that $(x)_n=x(x+1)\cdots (x+n-1)$ denotes the rising factorial with the convention $(x)_0=1$. The previous two formulas thus yield the explicit formula
\begin{equation}\label{bell1}
\begin{split} b_{\ell} = \frac{\ell !}{(\ell+d-2)! (d-1) c_T} \sum_{(\ell_1, \ell_2) \in B_{\ell}} &\psi (\lambda_{\ell_1}^{1/2}) \psi (\lambda_{\ell_2}^{1/2}) (2\ell_1 +d-1)(2\ell_2 +d-1) \\ &\times \frac{\left( \frac{d-1}{2} \right)_{s-\ell_1} \left( \frac{d-1}{2} \right)_{s-\ell_2} \left( \frac{d-1}{2} \right)_{s-\ell} (s+d-2)!}{(s-\ell_1)! (s-\ell_2)! (s-\ell)! \left( \frac{d-1}{2} \right)_s (2s+d-1)} , \end{split}
\end{equation}
where $s=(\ell_1+\ell_2+\ell)/2$ and the summation is over the index set
\[ B_{\ell} = \{ (\ell_1, \ell_2) \in \mathbb{Z}_{\ge 0}^2 : \ell_1+\ell_2+\ell \in 2 \mathbb{Z}, \quad \ell \le \ell_1+\ell_2, \quad \ell_1 \le \ell+\ell_2, \quad \ell_2 \le \ell+\ell_1 \} . \]

Observe that for any integer $m \ge 1$,
\[ \frac{\left( \frac{d-1}{2} \right)_m}{m!} = \prod_{j=1}^m \left( 1+\frac{d-3}{2j} \right) \le \exp \Bigg( \frac{d-3}{2} \sum_{j=1}^m \frac{1}{j} \Bigg) \le \exp \left( \frac{d-3}{2} (1+\log m) \right) . \]
In particular, for all $m \ge 0$,
\[ \frac{\left( \frac{d-1}{2} \right)_m}{m!} \le e^{(d-3)/2} (m)_+^{(d-3)/2}, \]
where we use the notation $(m)_+=\mathds{1}_{\{ m=0 \}} + \mathds{1}_{\{ m \ge 1 \}} m$. Hence
\[ \frac{\left( \frac{d-1}{2} \right)_{s-\ell_1} \left( \frac{d-1}{2} \right)_{s-\ell_2} \left( \frac{d-1}{2} \right)_{s-\ell}}{(s-\ell_1)!(s-\ell_2)!(s-\ell)!} \le e^{3(d-3)/2} \left( (s-\ell_1)_+ (s-\ell_2)_+ (s-\ell)_+ \right)^{(d-3)/2} . \]
Stirling's approximation \eqref{stirling} shows
\[ \begin{split} \frac{(s+d-2)!}{\left( \frac{d-1}{2} \right)_s (2s+d-1)} &= \frac{(s+d-2)! \Gamma \left( \frac{d-1}{2} \right)}{\Gamma \left( s+ \frac{d-1}{2} \right) 2 \left( s+\frac{d-1}{2} \right)} \\ &\le \frac{(s+d-2)^{s+d-3/2} e^{-(s+d-2)+1/12} \Gamma \left( \frac{d-1}{2} \right)}{2 \left( s+\frac{d-1}{2} \right)^{s+d/2} e^{-(s+(d-1)/2)}} \\ &= \left( 1+\frac{\frac{d-3}{2}}{s+\frac{d-1}{2}} \right)^{s+d/2} \frac{(s+d-2)^{(d-3)/2} e^{-(d-3)/2+1/12} \Gamma \left( \frac{d-1}{2} \right)}{2} \\ &\le \exp \left( \frac{\frac{d-3}{2} \left( s+\frac{d}{2} \right)}{s+\frac{d-1}{2}} \right) \frac{(s+d-2)^{(d-3)/2} e^{-(d-3)/2+1/12} \Gamma \left( \frac{d-1}{2} \right)}{2} \\ &\le \frac{(s+d-2)^{(d-3)/2} e^{(d-3)/(2d-2)+1/12} \Gamma \left( \frac{d-1}{2} \right)}{2} . \end{split} \]
Formula \eqref{bell1} thus simplifies to
\[ \begin{split} |b_{\ell}| \le \frac{e^{(d-3)(3d-2)/(2d-2)+1/12} \Gamma \left( \frac{d-1}{2} \right)}{2(d-1) \ell^{d-2} c_T} \sum_{(\ell_1, \ell_2) \in B_{\ell}} &\psi (\lambda_{\ell_1}^{1/2}) \psi (\lambda_{\ell_2}^{1/2}) (2\ell_1 +d-1)(2\ell_2 +d-1) \\ &\times \left( (s-\ell_1)_+ (s-\ell_2)_+ (s-\ell)_+ (s+d-2) \right)^{(d-3)/2} . \end{split} \]
The terms of the sum and the index set $B_{\ell}$ are both symmetric in $\ell_1, \ell_2$. It will thus be enough to sum over $\ell_1 \le \ell_2$, in which case $\ell_2 \ge \ell/2$, and
\[ s-\ell_1 \le \ell_2, \qquad s-\ell_2 \le \ell_1 , \qquad s-\ell \le \ell_1, \qquad s \le 2 \ell_2, \]
as well as $2 \ell_2 +d-1 \le (2+(d-1)2^{-d-1}/d)\ell_2$ since $\ell \ge 2^{d+2}/T \ge d2^{d+2}$. Therefore
\begin{equation}\label{bell2}
\begin{split} |b_{\ell}| \le &\frac{e^{(d-3)(3d-2)/(2d-2)+1/12} (2+\frac{d-1}{d2^{d+1}})^{(d-1)/2} \Gamma \left( \frac{d-1}{2} \right)}{(d-1)\ell^{d-2} c_T} \\ &\times \sum_{\ell_1=0}^{\infty} \psi (\lambda_{\ell_1}^{1/2}) (2\ell_1 +d-1) (\ell_1)_+^{d-3} \sum_{\ell_2 \ge \ell/2} \psi (\lambda_{\ell_2}^{1/2}) \ell_2^{d-2} . \end{split}
\end{equation}
The remaining two series in \eqref{bell2} can be bounded the same way $c_T$ was bounded above. Let $L_4=\lceil 2^{(d+1)/2}/T \rceil \ge 12$. Then $0 \le \psi \le 1$ shows
\[ \begin{split} \sum_{\ell_1=0}^{L_4} \psi (\lambda_{\ell_1}^{1/2}) (2 \ell_1+d-1) (\ell_1)_+^{d-3} &\le \sum_{\ell_1=0}^{L_4} (d-1) (\ell_1+1)^{d-2} \le \int_0^{L_4+1} (d-1) (x+1)^{d-2} \, \mathrm{d}x \\ &\le (L_4+2)^{d-1} \le \left( \frac{14}{11}(L_4-1) \right)^{d-1} \le \frac{\left( \frac{14}{11} \right)^{d-1} 2^{(d^2-1)/2}}{T^{d-1}} . \end{split} \]
Property \eqref{propii} and $2\ell_1+d-1 \le (28/13) \ell_1$ ensure that
\[ \begin{split} \sum_{\ell_1=L_4+1}^{\infty} \psi (\lambda_{\ell_1}^{1/2}) &(2 \ell_1+d-1) (\ell_1)_+^{d-3} \\ &\le \frac{28}{13} \sum_{\ell_1=L_4+1}^{\infty} \exp \left( - \frac{1}{\log 2} \left( \log \frac{T \ell_1}{4} \right) \left( \log \frac{T \ell_1}{2} \right) + (d-2) \log \ell_1 \right) \\ &\le \frac{28}{13} \int_{L_4}^{\infty} \exp \left( - \frac{1}{\log 2} \left( \log \frac{Tx}{4} \right) \left( \log \frac{Tx}{2} \right) + (d-2) \log x \right) \, \mathrm{d}x \\ &= \frac{28 \cdot 4^{d-1}}{13 T^{d-1}} \int_{-\infty}^{\infty} \exp \left( - \frac{1}{\log 2} y^2 + (d-2) y \right) \, \mathrm{d}y \\ &= \frac{14\sqrt{\pi \log 2}}{13} \cdot \frac{2^{d^2/4+d}}{T^{d-1}} . \end{split} \]
Adding the previous two formulas leads to
\begin{equation}\label{ell1sum}
\sum_{\ell_1=0}^{\infty} \psi (\lambda_{\ell_1}^{1/2}) (2 \ell_1+d-1) (\ell_1)_+^{d-3} \le \left( \left( \frac{14}{11} \right)^{d-1} 2^{(d^2-1)/2} + \frac{14 \sqrt{\pi \log 2}}{13}2^{d^2/4+d} \right) \frac{1}{T^{d-1}}.
\end{equation}
The assumptions $\ell \ge 2^{d+2}/T$ and $T \le 1/d$ ensure
\[ \frac{T(\ell/2-1)}{4} \ge \frac{T \ell}{16} \quad \textrm{and} \quad \frac{2}{\log 2} \log \frac{T\ell}{16} - (d-2) \ge d-2, \]
and we deduce
\[ \begin{split} \sum_{\ell_2 \ge \ell/2}^{\infty} \psi (\lambda_{\ell_2}^{1/2}) \ell_2^{d-2} &\le \sum_{\ell_2 \ge \ell/2} \exp \left( - \frac{1}{\log 2} \left( \log \frac{T \ell_2}{4} \right) \left( \log \frac{T \ell_2}{2} \right) + (d-2) \log \ell_2 \right) \\ &\le \int_{\ell/2-1}^{\infty} \exp \left( - \frac{1}{\log 2} \left( \log \frac{Tx}{4} \right) \left( \log \frac{Tx}{2} \right) + (d-2) \log x \right) \, \mathrm{d}x \\ &= \frac{4^{d-1}}{T^{d-1}} \int_{\log \frac{T(\ell/2-1)}{4}}^{\infty} \exp \left( - \frac{1}{\log 2} y^2 + (d-2)y \right) \, \mathrm{d}y \\ &\le \frac{4^{d-1}}{T^{d-1}} \int_{\log \frac{T\ell}{16}}^{\infty} \frac{\frac{2}{\log 2} y - (d-2)}{d-2} \exp \left( - \frac{1}{\log 2} y^2 + (d-2)y \right) \, \mathrm{d}y \\ &= \frac{4^{d-1}}{(d-2)T^{d-1}} \exp \left( - \frac{1}{\log 2} \left( \log \frac{T\ell}{16} \right) \left( \log \frac{T\ell}{2^{d+2}} \right) \right) . \end{split} \]
Formulas \eqref{bell2}, \eqref{ell1sum} and the previous formula together with the lower bound part of Lemma \ref{cTlemma} leads to
\[ |b_{\ell}| \le \kappa \exp \left( - \frac{1}{\log 2} \left( \log \frac{T\ell}{16} \right) \left( \log \frac{T\ell}{2^{d+2}} \right) \right) \]
with
\[ \begin{split} \kappa &= \frac{e^{(d-3)(3d-2)/(2d-2)+1/12} (2+\frac{d-1}{d2^{d+1}})^{(d-1)/2} \Gamma \left( \frac{d-1}{2} \right)}{(d-1)\ell^{d-2} \frac{6^d}{32 d! T^d}} \\ &{\phantom{=}} \times \left( \left( \frac{14}{11} \right)^{d-1} 2^{(d^2-1)/2} + \frac{14 \sqrt{\pi \log 2}}{13} 2^{d^2/4+d} \right) \frac{1}{T^{d-1}} \cdot \frac{4^{d-1}}{(d-2)T^{d-1}} \\ &\le \frac{e^{(d-3)(3d-2)/(2d-2)+1/12} (2+\frac{d-1}{d2^{d+1}})^{(d-1)/2} \Gamma \left( \frac{d-1}{2} \right)}{(d-1) 2^{d^2-4} \frac{6^d}{32 d!}} \\ &{\phantom{=}} \times \left( \left( \frac{14}{11} \right)^{d-1} 2^{(d^2-1)/2} + \frac{14 \sqrt{\pi \log 2}}{13} 2^{d^2/4+d} \right) \frac{4^{d-1}}{(d-2)} . \end{split} \]
In the last step we used $(T\ell)^{d-2} \ge 2^{d^2-4}$. The right-hand side of the previous formula converges to zero as $d \to \infty$. Numerical computations show that it is in fact decreasing in the variable $d \ge 3$, with maximal value $\le 43$ attained at $d=3$.
\end{proof}

\begin{proof}[Proof of Theorem \ref{sphereinfinitytheorem}] The kernel $K(x,y)$ satisfies $D_{\infty}(K) \le T$ by construction, hence an application of Lemma \ref{smoothinglemma} yields
\begin{equation}\label{sphereinfinity1}
W_{\infty} (\mu, \nu) \le C_1 T + W_{\infty} (K*\mu, K*\nu)
\end{equation}
with $C_1 = 1+\mathds{1}_{\{ \mu \neq \mathrm{Vol} \}}$. Clearly, $K*\mu \ge c \mathrm{Vol}$. The assumptions $\nu(B) \ge b$ for all closed geodesic balls $B \subseteq \mathbb{S}^d$ of radius $r$, and $T \ge 2^{d+3}d^{-1/2} r$, together with Lemma \ref{Kxylemma} show
\[ \begin{split} \frac{\mathrm{d}(K*\nu)}{\mathrm{d}\mathrm{Vol}}(x) &= \int_{\mathbb{S}^d} K(x,y) \, \mathrm{d}\nu(y) \ge \int_{B(x,r)} K(x,y) \, \mathrm{d}\nu(y) \\ &\ge \nu (B(x,r)) \frac{9^d}{114 d! 2^{d^2/4+d/4}T^d} \ge b \frac{9^d}{114 d! 2^{d^2/4+d/4}T^d} . \end{split} \]
In particular, $K*\nu \ge \delta \mathrm{Vol}$ with
\[ \delta = \min \left\{ b \frac{9^d}{114 d! 2^{d^2/4+d/4}T^d}, c \right\} , \]
and Lemma \ref{ledouxlemma} gives
\begin{equation}\label{sphereinfinity2}
W_{\infty} (K*\mu, K*\nu ) \le \frac{\log c - \log \delta}{c-\delta} \| K*\mu - K*\nu \|_{\dot{H}_{-1}^{\infty}} .
\end{equation}

To estimate the dual Sobolev norm in the previous formula, let $f \in \mathrm{C}^{\infty} (\mathbb{S}^d)$ be arbitrary. The expansion of $K(x,y)$ in Lemma \ref{belllemma} shows that the Fourier transform of $K*\mu$ is
\[ \widehat{\frac{\mathrm{d}(K*\mu)}{\mathrm{d}\mathrm{Vol}}} (\ell, m) = \int_{\mathbb{S}^d} \int_{\mathbb{S}^d} K(x,y) \overline{\phi_{\ell,m}} \, \mathrm{d}\mathrm{Vol}(x) \mathrm{d}\mu (y) = b_{\ell} \widehat{\mu}(\ell,m), \]
and the same holds for $\nu$. An application of the Parseval formula, the Cauchy--Schwarz inequality and Lemma \ref{homogeneoushausdorffyounglemma} lead to
\[ \begin{split} \left| \int_{\mathbb{S}^d} f \, \mathrm{d}(K*\mu - K*\nu) \right| &= \left| \sum_{\ell=1}^{\infty} \sum_{m=1}^{d_{\ell}} \overline{\widehat{f}(\ell,m)} b_{\ell} (\widehat{\mu}(\ell,m) - \widehat{\nu}(\ell,m)) \right| \\ &\le \sum_{\ell=1}^{\infty} |b_{\ell}| \bigg( \sum_{m=1}^{d_{\ell}} |\widehat{f}(\ell,m)|^2 \bigg)^{1/2} \bigg( \sum_{m=1}^{d_{\ell}} |\widehat{\mu}(\ell,m) - \widehat{\nu}(\ell,m)|^2 \bigg)^{1/2} \\ &\le \sup_{\ell \ge 0} d_{\ell}^{-1/2} \lambda_{\ell}^{1/2} \bigg( \sum_{m=1}^{d_{\ell}} |\widehat{f}(\ell,m)|^2 \bigg)^{1/2} \\ &{\phantom{\le}} \times \sum_{\ell=1}^{\infty} |b_{\ell}| d_{\ell}^{1/2} \lambda_{\ell}^{-1/2} \bigg( \sum_{m=1}^{d_{\ell}} |\widehat{\mu}(\ell,m) - \widehat{\nu}(\ell,m)|^2 \bigg)^{1/2} \\ &\le \| \nabla f \|_1 \sum_{\ell=1}^{\infty} |b_{\ell}| d_{\ell}^{1/2} \lambda_{\ell}^{-1/2} \bigg( \sum_{m=1}^{d_{\ell}} |\widehat{\mu}(\ell,m) - \widehat{\nu}(\ell,m)|^2 \bigg)^{1/2} . \end{split} \]
By the definition \eqref{dualsobolevnorm} of the dual Sobolev norm and Lemma \ref{dualsobolevnormlemma}, this implies
\[ \| K*\mu - K*\nu \|_{\dot{H}_{-1}^{\infty}} \le \sum_{\ell=1}^{\infty} |b_{\ell}| d_{\ell}^{1/2} \lambda_{\ell}^{-1/2} \bigg( \sum_{m=1}^{d_{\ell}} |\widehat{\mu}(\ell,m) - \widehat{\nu}(\ell,m)|^2 \bigg)^{1/2} . \]
Formulas \eqref{sphereinfinity1} and \eqref{sphereinfinity2}, the previous formula and the upper estimate for $|b_{\ell}|$ from Lemma \ref{belllemma} concludes the proof.
\end{proof}

\subsection{Compact Riemannian manifolds}

In this section, we give the proof of Theorems \ref{manifoldheatkernelp<2theorem} and \ref{manifoldheatkernelp>2theorem}.

\begin{proof}[Proof of Theorem \ref{manifoldheatkernelp<2theorem}] Let $1 \le p \le 2$, and fix $t>0$. We work with the heat kernel $P_t(x,y)$ on the manifold $M$. An application of Lemmas \ref{smoothinglemma} and \ref{heatkerneldispersionlemma} yields
\begin{equation}\label{manifoldheatkernelstep1}
W_p (\mu, \nu) \le C_1 t^{1/2} \left(1+C_3 t^{1/2} \right)^{1/2} + W_p(P_t * \mu, P_t * \nu)
\end{equation}
with
\[ \begin{split} C_1 &= 2^{1/2} (1+(1-c)^{1/p}) d^{1/2}, \\ C_3 &= \frac{2^{3/2}(d-1) \sqrt{A}}{3d} \left( d + (d-1) \sqrt{A} \, \mathrm{diam}(M) \right)^{1/2} . \end{split} \]

Assume now $1<p \le 2$ and $c>0$. The assumptions $\mu \ge c \mathrm{Vol}$ and $\nu(B) \ge b$ for all closed balls $B$ of radius $r$ show that for all $x \in M$,
\[ \frac{\mathrm{d}(P_t * \mu)}{\mathrm{d}\mathrm{Vol}}(x) = \int_M P_t (x,y) \, \mathrm{d}\mu(y) \ge c \int_M P_t (x,y) \, \mathrm{d}\mathrm{Vol}(y) \ge c \]
and
\[ \begin{split} \frac{\mathrm{d}(P_t * \nu)}{\mathrm{d}\mathrm{Vol}}(x) &= \int_M P_t (x,y) \, \mathrm{d}\nu(y) \ge \int_{B(x,r)} P_t(x,y) \, \mathrm{d}\nu (y) \\ &\ge \nu(B(x,r)) \inf_{\substack{y \in M \\ \rho (x,y) \le r}} P_t (x,y) \ge b \inf_{\substack{y \in M \\ \rho (x,y) \le r}} P_t (x,y) . \end{split} \]
In particular, $P_t * \mu \ge c \mathrm{Vol}$ and $P_t * \nu \ge \delta \mathrm{Vol}$ with
\[ \delta = \min \left\{ b \inf_{\substack{x,y \in M \\ \rho (x,y) \le r}} P_t (x,y), c \right\}, \]
and an application of Lemma \ref{ledouxlemma} gives
\begin{equation}\label{manifoldheatkernelstep2}
W_p (P_t * \mu, P_t * \nu) \le \frac{p(c^{1/p} - \delta^{1/p})}{c-\delta} \| \mu - \nu \|_{\dot{H}_{-1}^p} .
\end{equation}

To estimate the dual Sobolev norm in the previous formula, let $f \in \mathrm{C}^{\infty}(M)$ be arbitrary. An application of Green's formula shows that the Fourier transform of $\Delta f$ is
\[ \widehat{\Delta f}(k) = \int_M (\Delta f) \overline{\phi_k} \, \mathrm{d} \mathrm{Vol} = \int_M f \Delta \overline{\phi_k} \, \mathrm{d}\mathrm{Vol} = \int_M f (-\Lambda_k \overline{\phi_k}) \, \mathrm{d}\mathrm{Vol} = - \Lambda_k \widehat{f}(k), \quad k \ge 0. \]
Another application of Green's formula and the Parseval formula then yields
\[ \int_M |\nabla f|^2 \, \mathrm{d}\mathrm{Vol} = - \int_M f \Delta f \, \mathrm{d} \mathrm{Vol} = - \sum_{k=0}^{\infty} \widehat{f}(k) \overline{\widehat{\Delta f}(k)} = \sum_{k=0}^{\infty} \Lambda_k |\widehat{f}(k)|^2 . \]
Since $P_t(x,y) = \sum_{k=0}^{\infty} e^{-\Lambda_k t} \phi_k(x) \overline{\phi_k(y)}$ is uniformly convergent in $x,y \in M$ for any fixed $t>0$ and $\phi_k$, $k \ge 0$ are orthonormal in $L^2(M)$,
\[ \begin{split} \widehat{\frac{\mathrm{d}(P_t * \mu)}{\mathrm{d}\mathrm{Vol}}}(k) &= \int_M \frac{\mathrm{d}(P_t * \mu)}{\mathrm{d}\mathrm{Vol}} (x) \overline{\phi_k(x)} \, \mathrm{d}\mathrm{Vol}(x) = \int_M \int_M P_t(x,y) \overline{\phi_k(x)} \, \mathrm{d} \mathrm{Vol}(x) \mathrm{d}\mu (y) \\ &=\sum_{\ell=0}^{\infty} e^{-\Lambda_{\ell} t} \int_M \phi_{\ell}(x) \overline{\phi_k (x)} \, \mathrm{d}\mathrm{Vol}(x) \int_M \overline{\phi_{\ell}(y)} \, \mathrm{d}\mu(y) = e^{-\Lambda_k t} \widehat{\mu}(k), \end{split} \]
and the same holds for $\nu$. The previous two formulas and an application of the Parseval formula and the Cauchy--Schwarz inequality lead to
\[ \begin{split} \left| \int_M f \, \mathrm{d}(P_t * \mu - P_t * \nu) \right| &= \left| \sum_{k=0}^{\infty} \overline{\widehat{f}(k)} e^{-\Lambda_k t} (\widehat{\mu}(k) - \widehat{\nu}(k)) \right| \\ &\le \left( \sum_{k=0}^{\infty} \Lambda_k |\widehat{f}(k)|^2 \right)^{1/2} \left( \sum_{k=1}^{\infty} e^{-2 \Lambda_k t} \frac{|\widehat{\mu}(k) - \widehat{\nu}(k)|^2}{\Lambda_k} \right)^{1/2} \\ &= \| \nabla f \|_2 \left( \sum_{k=1}^{\infty} e^{-2 \Lambda_k t} \frac{|\widehat{\mu}(k) - \widehat{\nu}(k)|^2}{\Lambda_k} \right)^{1/2} . \end{split} \]
By the definition \eqref{dualsobolevnorm} of the dual Sobolev norm and Lemma \ref{dualsobolevnormlemma}, this implies
\[ \| \mu - \nu \|_{\dot{H}_{-1}^p} \le \| \mu - \nu \|_{\dot{H}_{-1}^2} \le \left( \sum_{k=1}^{\infty} e^{-2 \Lambda_k t} \frac{|\widehat{\mu}(k) - \widehat{\nu}(k)|^2}{\Lambda_k} \right)^{1/2} . \]
The claim for $1<p \le 2$ follows from \eqref{manifoldheatkernelstep1}, \eqref{manifoldheatkernelstep2} and the previous formula.

The proof for $p=1$ and $c \ge 0$ is very similar. Let $f \in \mathrm{C}^{\infty}(M)$ be arbitrary. Using $|\nabla f| \le \| f \|_{\mathrm{Lip}}$, we similarly deduce
\[ \begin{split} \left| \int_M f \, \mathrm{d}(P_t * \mu - P_t * \nu) \right| &\le \| \nabla f \|_2 \left( \sum_{k=1}^{\infty} e^{-2 \Lambda_k t} \frac{|\widehat{\mu}(k) - \widehat{\nu}(k)|^2}{\Lambda_k} \right)^{1/2} \\ &\le \| f \|_{\mathrm{Lip}} \left( \sum_{k=1}^{\infty} e^{-2 \Lambda_k t} \frac{|\widehat{\mu}(k) - \widehat{\nu}(k)|^2}{\Lambda_k} \right)^{1/2} . \end{split} \]
The Kantorovich duality formula in the form given in Lemma \ref{dualsobolevnormlemma} thus implies
\[ W_1 (P_t * \mu, P_t* \nu) \le \left( \sum_{k=1}^{\infty} e^{-2 \Lambda_k t} \frac{|\widehat{\mu}(k) - \widehat{\nu}(k)|^2}{\Lambda_k} \right)^{1/2} . \]
The claim for $p=1$ follows from \eqref{manifoldheatkernelstep1} and the previous formula.
\end{proof}

The following lemma is the Hausdorff--Young inequality on a general compact manifold.
\begin{lem}\label{manifoldhausdorffyounglemma} Let $M$ be a compact, connected, smooth Riemannian manifold of dimension $1 \le d < \infty$ without boundary. Assume the Ricci curvature is bounded below by $-a$ with some constant $a \ge 0$. Let $f \in \mathrm{C}^{\infty}(M)$. For any $2 \le p < \infty$ and $1<q \le 2$ such that $1/p+1/q=1$,
\begin{align}
\Bigg( \sum_{\ell=0}^{\infty} (\ell+1)^{-(p-2)(d-1)/2} \Bigg( \sum_{\substack{k \ge 0 \\ \Lambda_k^{1/2} \in [\ell, \ell+1)}} |\widehat{f}(k)|^2 \Bigg)^{p/2} \Bigg)^{1/p} &\le K_{\mathrm{Weyl}}(M)^{(p-2)/(2p)} \| f \|_q , \label{manifoldHY1} \\ \sup_{\ell \ge 0} \, (\ell+1)^{-(d-1)/2} \Bigg( \sum_{\substack{k \ge 0 \\ \Lambda_k^{1/2} \in [\ell, \ell+1)}} |\widehat{f}(k)|^2 \Bigg)^{1/2} &\le K_{\mathrm{Weyl}}(M)^{1/2} \| f \|_1 , \label{manifoldHY2} \\ \Bigg( \sum_{\ell=0}^{\infty} (\ell+1)^{-(p-2)(d-1)/2} \Bigg( \sum_{\substack{k \ge 1 \\ \Lambda_k^{1/2} \in [\ell, \ell+1)}} (a+\Lambda_k) |\widehat{f}(k)|^2 \Bigg)^{p/2} \Bigg)^{1/p} &\le C \| \nabla f \|_q \label{manifoldHY3}
\end{align}
with $C=K_{\mathrm{Weyl}}(M)^{(p-2)/(2p)} (3 \sqrt{6} (p-1) + a^{1/2} K_{\mathrm{Poincar\acute{e}}}(M,q))$, where $K_{\mathrm{Weyl}}(M)$ and $K_{\mathrm{Poincar\acute{e}}}(M,q)$ are the constants in the pointwise Weyl law \eqref{weyllaw} and the Poincar\'e inequality \eqref{poincare}. If $a=0$, then the factor $3\sqrt{6}$ in the definition of $C$ can be replaced by 2.
\end{lem}

\begin{proof} Claim \eqref{manifoldHY1} for $p=q=2$ reduces to the Parseval formula. Now fix $\ell \ge 0$. Let $k_1<k_2<\cdots<k_n$ be the set of integers $k \ge 0$ for which $\Lambda_k^{1/2} \in [\ell, \ell+1)$. By the pointwise Weyl law \eqref{weyllaw}, the Euclidean norm of the vector-valued function $\Phi_{\ell}=(\phi_{k_1}, \phi_{k_2}, \ldots, \phi_{k_n})$ satisfies
\[ |\Phi_{\ell}(x)|^2 = \sum_{\substack{k \ge 0 \\ \Lambda_k^{1/2} \in [\ell, \ell+1)}} |\phi_k(x)|^2 \le K_{\mathrm{Weyl}}(M) (\ell+1)^{d-1} . \]
The triangle inequality for vector-valued integrals
\[ \left| \int_M f \Phi_{\ell} \, \mathrm{d}\mathrm{Vol} \right| \le \int_M |f| |\Phi_{\ell}| \, \mathrm{d}\mathrm{Vol} \]
thus yields
\[ \Bigg( \sum_{\substack{k \ge 0 \\ \Lambda_k^{1/2} \in [\ell, \ell+1)}} |\widehat{f}(k)|^2 \Bigg)^{1/2} \le K_{\mathrm{Weyl}}(M)^{1/2} (\ell+1)^{(d-1)/2} \| f \|_1 , \]
and claim \eqref{manifoldHY2} follows. Similarly to the proof of Lemma \ref{homogeneoushausdorffyounglemma}, an application of the Riesz--Thorin interpolation theorem finishes the proof of \eqref{manifoldHY1} for general $2 \le p < \infty$ and $1<q \le 2$.

We now prove claim \eqref{manifoldHY3}. Since both sides of \eqref{manifoldHY3} are invariant under adding a constant to $f$ (note that the summation is over $k \ge 1$), we may assume that $\int_M f \, \mathrm{d} \mathrm{Vol}=0$. By definition, the Fourier transform of $(a-\Delta)^{1/2} f$ is $\widehat{(a-\Delta)^{1/2} f}(k) = (a+\Lambda_k)^{1/2} \widehat{f}(k)$, $k \ge 0$. Applying \eqref{manifoldHY1} to $(a-\Delta)^{1/2} f$ thus leads to
\[ \Bigg( \sum_{\ell=0}^{\infty} (\ell+1)^{-(p-2)(d-1)/2} \Bigg( \sum_{\substack{k \ge 1 \\ \Lambda_k^{1/2} \in [\ell, \ell+1)}} (a+\Lambda_k) |\widehat{f}(k)|^2 \Bigg)^{p/2} \Bigg)^{1/p} \le K_{\mathrm{Weyl}}(M)^{(p-2)/(2p)} \| (a-\Delta)^{1/2} f \|_q . \]
An application of the Riesz transform estimate \eqref{riesztransform2} and the Poincar\'e inequality \eqref{poincare} shows that here
\[ \| (a-\Delta)^{1/2} f \|_q \le 3 \sqrt{6} (p-1) \| \nabla f \|_q + a^{1/2} \| f \|_q \le \left( 3\sqrt{6} (p-1) + a^{1/2} K_{\mathrm{Poincar\acute{e}}}(M,q) \right) \| \nabla f \|_q , \]
and claim \eqref{manifoldHY3} follows. If $a=0$, then we can use the Riesz transform estimate \eqref{riesztransform1} instead of \eqref{riesztransform2}, consequently we can replace the factor $3\sqrt{6}$ by $2$.
\end{proof}

\begin{proof}[Proof of Theorem \ref{manifoldheatkernelp>2theorem}] Let $2<p<\infty$ and $1<q<2$ be such that $1/p+1/q=1$, and fix $t>0$. We work with the heat kernel $P_t (x,y)$ on the manifold $M$. An application of Lemmas \ref{smoothinglemma} and \ref{heatkerneldispersionlemma} yields
\begin{equation}\label{manifoldheatkernelp>2step1}
W_p (\mu, \nu) \le C_1 t^{1/2} \left(1+C_3 t^{1/2} \right)^{1/p} + W_p(P_t * \mu, P_t * \nu)
\end{equation}
with
\[ \begin{split} C_1 &= (1+(1-c)^{1/p}) (2d+p)^{1/2}, \\ C_3 &= 2 (d-1) \sqrt{A} \exp \left( \frac{p+2}{2} + \frac{(d-1) \sqrt{A}}{2} \, \mathrm{diam}(M) \right) . \end{split} \]
Similarly to the proof of Theorem \ref{manifoldheatkernelp<2theorem}, an application of Lemma \ref{ledouxlemma} gives
\begin{equation}\label{manifoldheatkernelp>2step2}
W_p (P_t * \mu, P_t * \nu) \le \frac{p(c^{1/p} - \delta^{1/p})}{c-\delta} \| \mu - \nu \|_{\dot{H}_{-1}^p}
\end{equation}
with
\[ \delta = \min \left\{ b \inf_{\substack{x,y \in M \\ \rho (x,y) \le r}} P_t (x,y), c \right\} . \]

To estimate the dual Sobolev norm in \eqref{manifoldheatkernelp>2step2}, let $f \in \mathrm{C}^{\infty}(M)$ be arbitrary. An application of the Parseval formula, the Cauchy--Schwarz and the H\"older inequalities and Lemma \ref{manifoldhausdorffyounglemma} leads to
\[ \begin{split} \Bigg| \int_M f \, \mathrm{d}&(P_t * \mu - P_t*\nu) \Bigg| \\ &= \Bigg| \sum_{\ell=0}^{\infty} \sum_{\substack{k \ge 0 \\ \Lambda_k^{1/2} \in [\ell, \ell+1)}} \overline{\widehat{f}(k)} e^{-\Lambda_k t} (\widehat{\mu}(k) - \widehat{\nu}(k)) \Bigg| \\ &\le \sum_{\ell=0}^{\infty} \Bigg( \sum_{\substack{k \ge 1 \\ \Lambda_k^{1/2} \in [\ell, \ell+1)}} (a+\Lambda_k) |\widehat{f}(k)|^2 \Bigg)^{1/2} \Bigg( \sum_{\substack{k \ge 1 \\ \Lambda_k^{1/2} \in [\ell, \ell+1)}} e^{-2 \Lambda_k t} \frac{|\widehat{\mu}(k) - \widehat{\nu}(k)|^2}{a+\Lambda_k} \Bigg)^{1/2} \end{split} \]

\[ \begin{split} &\le \Bigg( \sum_{\ell=0}^{\infty} (\ell+1)^{-(p-2)(d-1)/2} \Bigg( \sum_{\substack{k \ge 1 \\ \Lambda_k^{1/2} \in [\ell, \ell+1)}} (a+\Lambda_k) |\widehat{f}(k)|^2 \Bigg)^{p/2} \Bigg)^{1/p} \\ &{\phantom{\le}} \times \Bigg( \sum_{\ell=0}^{\infty} (\ell+1)^{(p-2)(d-1)/(2p-2)} \Bigg( \sum_{\substack{k \ge 1 \\ \Lambda_k^{1/2} \in [\ell, \ell+1)}} e^{-2 \Lambda_k t} \frac{|\widehat{\mu}(k) - \widehat{\nu}(k)|^2}{a+\Lambda_k} \Bigg)^{q/2} \Bigg)^{1/q} \\ &\le C \| \nabla f \|_q \Bigg( \sum_{\ell=0}^{\infty} (\ell+1)^{(p-2)(d-1)/(2p-2)} \Bigg( \sum_{\substack{k \ge 1 \\ \Lambda_k^{1/2} \in [\ell, \ell+1)}} e^{-2 \Lambda_k t} \frac{|\widehat{\mu}(k) - \widehat{\nu}(k)|^2}{a+\Lambda_k} \Bigg)^{q/2} \Bigg)^{1/q} \end{split} \]
with $a=(d-1)A$ and $C= K_{\mathrm{Weyl}}(M)^{(p-2)/(2p)} (3 \sqrt{6} (p-1) + a^{1/2} K_{\mathrm{Poincar\acute{e}}}(M,q))$. By the definition \eqref{dualsobolevnorm} of the dual Sobolev norm and Lemma \ref{dualsobolevnormlemma}, this implies
\[ \| \mu- \nu \|_{\dot{H}_{-1}^p} \le C \Bigg( \sum_{\ell=0}^{\infty} (\ell+1)^{(p-2)(d-1)/(2p-2)} \Bigg( \sum_{\substack{k \ge 1 \\ \Lambda_k^{1/2} \in [\ell, \ell+1)}} e^{-2 \Lambda_k t} \frac{|\widehat{\mu}(k) - \widehat{\nu}(k)|^2}{a+\Lambda_k} \Bigg)^{q/2} \Bigg)^{1/q} . \]
The claim now follows from \eqref{manifoldheatkernelp>2step1}, \eqref{manifoldheatkernelp>2step2} and the previous formula.
\end{proof}

\vspace{5mm}

\noindent\textsc{Bence Borda}

\noindent\textsc{Department of Mathematics, University of Sussex, Brighton, BN1 9RH, United Kingdom}

\noindent\textsc{Email:} \texttt{b.borda@sussex.ac.uk}

\vspace{5mm}

\noindent\textsc{Jean-Claude Cuenin}

\noindent\textsc{Department of Mathematical Sciences, Loughborough University, Loughborough, LE11 3TU, United Kingdom}

\noindent\textsc{Email:} \texttt{j.cuenin@loughborough.ac.uk}

\end{document}